\documentclass[11pt,a4paper]{amsart}
\usepackage{amssymb,amsthm}

\usepackage[truedimen,margin=30truemm]{geometry}

\usepackage[pdftex,colorlinks=true,bookmarks=true,citecolor=blue,linkcolor=blue,pdfauthor={},pdftitle={}]{hyperref}

\theoremstyle{plain}
\newtheorem{theorem}{Theorem}[section]
\newtheorem{lemma}[theorem]{Lemma}

\theoremstyle{definition}
\newtheorem{definition}[theorem]{Definition}

\theoremstyle{remark}
\newtheorem{remark}[theorem]{Remark}

\numberwithin{equation}{section}

\begin{document}

\title[Gunzburg--Landau heat flow in a curved thin domain]{Thin-film limit of the Ginzburg--Landau heat flow in a curved thin domain}

\author[T.-H. Miura]{Tatsu-Hiko Miura}
\address{Graduate School of Science and Technology, Hirosaki University, 3, Bunkyo-cho, Hirosaki-shi, Aomori, 036-8561, Japan}
\email{thmiura623@hirosaki-u.ac.jp}



\begin{abstract}
  We consider the Ginzburg--Landau heat flow without magnetic effect in a curved thin domain under the Naumann boundary condition.
  When the curved thin domain shrinks to a given closed hypersurface as the thickness of the thin domain tends to zero, we show that the weighted average of a weak solution to the thin-domain problem converges weakly on the limit surface under the assumption that the initial data is of class $L^\infty$ and satisfies some conditions.
  Moreover, under the same assumption, we derive a limit equation by characterizing the limit function as a weak solution, and prove a difference estimate on the limit surface of an averaged weak solution to the thin-domain problem and a weak solution to the limit problem explicitly in terms of the thickness of the thin domain.
  We also derive a difference estimate in the curved thin domain of weak solutions to the thin-domain problem and to the limit problem, but without requiring that the initial data of the thin-domain problem is of class $L^\infty$.
\end{abstract}

\maketitle

\section{Introduction} \label{S:Intro}

\subsection{Motivation} \label{SS:In_Mot}
In this paper, we consider the Ginzburg--Landau heat flow without magnetic effect in a curved thin domain around a given closed hypersurface under the Neumann boundary condition.
The purposes of this paper are to derive a thin-film limit equation rigorously by convergence of a solution and characterization of the limit, and to estimate the difference of a solution to the thin-domain problem and a solution to the limit problem explicitly in terms of the thickness of the thin domain.

Partial differential equations (PDEs) in thin domains arise in many problems of natural sciences like elasticity, lubrication, and geophysical fluid dynamics (see e.g. \cite{Cia97,Cia00,OckOck95,Ped87}).
When a thin domain shrinks to a lower dimensional limit set as the thickness of the thin domain tends to zero, it is natural to try to derive a limit equation on the limit set from a PDE in the thin domain and to compare a thin-domain problem and a limit problem.
Such a thin-film limit problem is important in view of reduction of dimension and of modelling various phenomena in thin domains as PDEs on limit sets.
It is also important to analyze the effects of the thin direction(s) and the geometry of limit sets on PDEs in thin domains, which often appear as coefficients and differential operators of limit equations obtained by the thin-film limit.

Our main motivation for considering the Ginzburg--Landau heat flow without magnetic effect is the study of nematic liquid crystals in thin domains.
The motion of nematic liquid crystals is described by the Ericksen--Leslie equations proposed in their works \cite{Eri62,Les68}, which is a system of equations of fluids and the director of molecules.
Since the original system is highly involved, some simplified systems have been proposed for a mathematical study.
One of simplified systems we are interested in is the one introduced by Lin \cite{Lin89} and studied by Lin and Liu \cite{LinLiu95}.
It is a coupled system of the incompressible Navier--Stokes equations and the Ginzburg--Landau heat flow.
Recently, Nitschke, Reuther, and Voigt \cite{NiReVo19} carried out a formal thin-film limit of the simplified Ericksen--Leslie equations in a curved thin domain around a moving surface in order to model and analyze the motion of surface nematic liquid crystals.
We also refer to \cite{GoMoSt17,NiNePrLoVo18} for other models of surface liquid crystals obtained by the thin-film limit.
However, the formal thin-film limit carried out in \cite{NiReVo19} has not been mathematically justified.
For the fluid part, we rigorously justified the thin-film limit of the Navier--Stokes equations in a curved thin domain around a stationary surface in \cite{Miu20_03} by convergence of a solution and characterization of the limit.
Hence, it is natural for us to study the director part separately.
Our aims are to justify the thin-film limit of the Ginzburg--Landau heat flow in a curved thin domain as well as to develop analytical methods for a future study of the full system of the (simplified) Ericksen--Leslie equations in thin domains.

It should be also mentioned that the Ginzburg--Landau equations under magnetic forces in thin domains have been studied extensively.
The thin-film limit of the Ginzburg--Landau equations with magnetic effect is an important problem for modelling and analysis of superconductivity of thin films.
We refer to \cite{AlBrGa10,ChDuGu96,ChElQi98,Con11,ConSte10,DuGun93,Glo11,JimMor02,Mor04,RicRub00,RubSch98} and the references cited therein.
It is a possible future work to extend the results of this paper to the case of the Ginzburg--Landau equations with magnetic effect.

\subsection{Problem settings and main results} \label{S:In_Main}
We fix the settings of our problem and state main results of this paper.
For details of notations, we refer to Section \ref{S:Prelim}.

Let $\Gamma$ be a $C^3$ closed hypersurface in $\mathbb{R}^n$ with $n\in\mathbb{N}$, and let $\nu$ be the unit outward normal vector field of $\Gamma$.
Also, let $g_0$ and $g_1$ be $C^1$ functions on $\Gamma$ such that $g=g_1-g_0\geq c$ on $\Gamma$ with some constant $c>0$.
For a small $\varepsilon>0$, we define a curved thin domain
\begin{align} \label{E:Def_CTD}
  \Omega_\varepsilon = \{y+r\nu(y) \mid y\in\Gamma, \, \varepsilon g_0(y)<r<\varepsilon g_1(y)\} \subset \mathbb{R}^n
\end{align}
and consider the Neumann problem of the Ginzburg--Landau heat flow
\begin{align} \label{E:GL_CTD}
  \left\{
  \begin{alignedat}{3}
    \partial_tu^\varepsilon-\Delta u^\varepsilon+\lambda(|u^\varepsilon|^2-1)u^\varepsilon &= 0 &\quad &\text{in} &\quad & \Omega_\varepsilon\times(0,\infty), \\
    \partial_{\nu_\varepsilon}u^\varepsilon &= 0 &\quad &\text{on} &\quad &\partial\Omega_\varepsilon\times(0,\infty), \\
    u^\varepsilon|_{t=0} &= u_0^\varepsilon &\quad &\text{in} &\quad &\Omega_\varepsilon.
  \end{alignedat}
  \right.
\end{align}
Here, the unknown function $u^\varepsilon$ and the given initial data $u_0^\varepsilon$ are $\mathbb{R}^N$-valued functions with $N\in\mathbb{N}$.
Also, $\lambda>0$ is a constant independent of $\varepsilon$.
The symbol $\partial_{\nu_\varepsilon}$ stands for the outer normal derivative on the boundary $\partial\Omega_\varepsilon$.
It seems that the most physically related case is $n=2,3$ and $N=1,2,3$, where \eqref{E:GL_CTD} reduces to the Allen--Cahn equation when $N=1$, but our analysis does not require such a restriction.

For each initial data $u_0^\varepsilon\in L^2(\Omega_\varepsilon)^N$, it can be shown by the Galerkin and energy methods that there exists a global-in-time unique weak solution $u^\varepsilon$ to \eqref{E:GL_CTD} (see Theorem \ref{T:GCT_Exi}).
Our aims is to analyze the behavior of $u^\varepsilon$ as $\varepsilon\to0$.

To state the main results, let us fix some notations (see Sections \ref{S:Prelim} and \ref{S:Ave} for details).
Let $\nabla_\Gamma$ and $\mathrm{div}_\Gamma$ be the tangential gradient and the surface divergence on $\Gamma$, respectively.
Also, for a function $\varphi$ on $\Omega_\varepsilon$, we define the weighted average of $\varphi$ in the thin direction by
\begin{align*}
  \mathcal{M}_\varepsilon\varphi(y) = \frac{1}{\varepsilon g(y)}\int_{\varepsilon g_0(y)}^{\varepsilon g_1(y)}\varphi\bigl(y+r\nu(y)\bigr)J(y,r)\,dr, \quad y\in\Gamma,
\end{align*}
where $J(y,r)=\prod_{\alpha=1}^{n-1}\{1-r\kappa_\alpha(y)\}$.
Here, $\kappa_1,\dots,\kappa_{n-1}$ are the principal curvatures of $\Gamma$, and the function $J(y,r)$ is the Jacobian appearing in the change of variables of an integral over $\Omega_\varepsilon$ (see \eqref{E:Ch_Va}).
The weighted average is useful when we take the average in the thin direction of a weak form of \eqref{E:GL_CTD}.
It was effectively used in the study of the thin-film limit of the heat equation in a moving thin domain \cite{Miu17}.

Our first result is the weak convergence of the weighted average $\mathcal{M}_\varepsilon u^\varepsilon$ of a weak solution $u^\varepsilon$ to \eqref{E:GL_CTD} as $\varepsilon\to0$ under the assumption that the initial data $u_0^\varepsilon$ is of class $L^\infty$ and satisfies some conditions.
Moreover, we derive a limit equation on $\Gamma$ by characterizing the weak limit of $\mathcal{M}_\varepsilon u^\varepsilon$ as a weak solution to the limit equation.

\begin{theorem} \label{T:Weak}
  Suppose that $u_0^\varepsilon\in L^\infty(\Omega_\varepsilon)^N$ and
  \begin{itemize}
    \item[(a)] there exist constants $c_1\geq1$, $\alpha\in(0,1/3]$, and $\varepsilon_0\in(0,1)$ such that
    \begin{align*}
      \|u_0^\varepsilon\|_{L^\infty(\Omega_\varepsilon)} \leq c_1\varepsilon^{-1/3+\alpha} \quad\text{for all}\quad \varepsilon\in(0,\varepsilon_0),
    \end{align*}
    \item[(b)] there exists a function $v_0\in L^2(\Gamma)^N$ such that
    \begin{align*}
      \lim_{\varepsilon\to0}\mathcal{M}_\varepsilon u_0^\varepsilon = v_0 \quad\text{weakly in}\quad L^2(\Gamma)^N.
    \end{align*}
  \end{itemize}
  Then, there exists a unique global weak solution
  \begin{align*}
    u^\varepsilon \in C([0,\infty);L^2(\Omega_\varepsilon)^N)\cap L_{loc}^2([0,T);H^1(\Omega_\varepsilon)^N)\cap L_{loc}^4([0,\infty);L^4(\Omega_\varepsilon)^N)
  \end{align*}
  to \eqref{E:GL_CTD} for each $\varepsilon\in(0,\varepsilon_0)$.
  Moreover, there exists a function
  \begin{align*}
    v \in C([0,\infty);L^2(S)^N)\cap L_{loc}^2([0,T);H^1(S)^N)\cap L_{loc}^4([0,\infty);L^4(S)^N)
  \end{align*}
  such that
  \begin{align*}
    \lim_{\varepsilon\to0}\mathcal{M}_\varepsilon u^\varepsilon = v \quad\text{weakly in}\quad L^2(0,T;H^1(S)^N)\cap L^4(0,T;L^4(\Gamma)^N)
  \end{align*}
  for each $T>0$ and $v$ is a unique global weak solution to the limit problem
  \begin{align} \label{E:GL_Lim}
    \left\{
    \begin{alignedat}{3}
      \partial_tv-\frac{1}{g}\,\mathrm{div}_\Gamma(g\nabla_\Gamma v)+\lambda(|v|^2-1)v &= 0 &\quad &\text{on} &\quad &\Gamma\times(0,\infty), \\
      v|_{t=0} &= v_0 &\quad &\text{on} &\quad &\Gamma.
    \end{alignedat}
    \right.
  \end{align}
\end{theorem}

Theorem \ref{T:Weak} also gives the global existence of a weak solution to \eqref{E:GL_Lim} for a certain class of initial data.
For a general initial data $v_0\in L^2(\Gamma)^N$, the global existence and uniqueness of a weak solution to \eqref{E:GL_Lim} can be shown by the Galerkin and energy methods as in the case of the thin-domain problem \eqref{E:GL_CTD} (see also Section \ref{SS:WS_Lim}).

Next, we get a difference estimate on $\Gamma$ of the weighted average $\mathcal{M}_\varepsilon u^\varepsilon$ of a weak solution $u^\varepsilon$ to \eqref{E:GL_CTD} and a weak solution $v$ to \eqref{E:GL_Lim} explicitly in terms of $\varepsilon$.
This difference estimate gives the strong convergence of $\mathcal{M}_\varepsilon u^\varepsilon$ towards $v$ as $\varepsilon\to0$.

\begin{theorem} \label{T:Df_Sur}
  Suppose that $u_0^\varepsilon\in L^\infty(\Omega_\varepsilon)^N$ and the condition \emph{(a)} of Theorem \ref{T:Weak} is satisfied.
  Let $u^\varepsilon$ be a global weak solution to \eqref{E:GL_CTD}.
  Also, let $v_0\in L^2(\Gamma)^N$ and $v$ be a global weak solution to \eqref{E:GL_Lim}.
  Then,
  \begin{multline} \label{E:Df_Sur}
    \|\mathcal{M}_\varepsilon u^\varepsilon-v\|_{C([0,T];L^2(\Gamma))}+\|\nabla_\Gamma\mathcal{M}_\varepsilon u^\varepsilon-\nabla_\Gamma v\|_{L^2(0,T;L^2(\Gamma))} \\
    \leq ce^{c(1+\lambda)T}\Bigl\{\|\mathcal{M}_\varepsilon u_0^\varepsilon-v_0\|_{L^2(\Gamma)}+\varepsilon^{3\alpha}(1+\lambda)\Bigr\}
  \end{multline}
  for all $T>0$, where $c>0$ is a constant independent of $\varepsilon$, $\lambda$, and $T$.
  In particular,
  \begin{align*}
    \lim_{\varepsilon\to0}\mathcal{M}_\varepsilon u^\varepsilon = v \quad\text{strongly in} \quad C([0,T];L^2(\Gamma)^N)\cap L^2(0,T;H^1(\Gamma)^N)
  \end{align*}
  provided that $\mathcal{M}_\varepsilon u_0^\varepsilon\to v_0$ strongly in $L^2(\Gamma)^N$ as $\varepsilon\to0$.
\end{theorem}

Note that we assume that the constant $\lambda$ is independent of $\varepsilon$.
If $\alpha=1/3$ in the condition (a) of Theorem \ref{T:Weak}, then $\varepsilon^{3\alpha}=\varepsilon$ in \eqref{E:Df_Sur} and thus $v$ approximates $\mathcal{M}_\varepsilon u^\varepsilon$ of order $\varepsilon$.
This is the case when $\|u_0^\varepsilon\|_{L^\infty(\Omega_\varepsilon)}$ is uniformly bounded in $\varepsilon$ (e.g. $|u_0^\varepsilon|\leq1$ a.e. in $\Omega_\varepsilon$).

We also estimate the difference of $u^\varepsilon$ and $v$ in $\Omega_\varepsilon$.
It seems that the difference estimate \eqref{E:Df_Sur} on $\Gamma$ directly gives a difference estimate in $\Omega_\varepsilon$.
However, as we explain in Section \ref{SS:In_Idea}, the use of \eqref{E:Df_Sur} requires a higher order regularity of $u^\varepsilon$.
To avoid such a requirement, we take another and somewhat new approach here.
Surprisingly, this approach enables us to remove the assumption $u_0^\varepsilon\in L^\infty(\Omega_\varepsilon)^N$.
For a function $\eta$ on $\Gamma$, we write $\bar{\eta}$ for the constant extension of $\eta$ in the normal direction of $\Gamma$ (see Section \ref{SS:Pr_Sur} for details).

\begin{theorem} \label{T:Df_CTD}
  Let $u_0^\varepsilon\in L^2(\Omega_\varepsilon)^N$ and $v_0\in L^2(\Gamma)^N$, let $u^\varepsilon$ and $v$ be global weak solutions to \eqref{E:GL_CTD} and \eqref{E:GL_Lim}, respectively.
  Then,
  \begin{multline} \label{E:Df_CTD}
    \varepsilon^{-1/2}\Bigl(\|u^\varepsilon-\bar{v}\|_{C([0,T];L^2(\Omega_\varepsilon))}+\|\nabla u^\varepsilon-\nabla\bar{v}\|_{L^2(0,T;L^2(\Omega_\varepsilon))}\Bigr) \\
    \leq ce^{c(1+\lambda)T}\Bigl(\varepsilon^{-1/2}\|u_0^\varepsilon-\bar{v}_0\|_{L^2(\Omega_\varepsilon)}+\varepsilon\|v_0\|_{L^2(\Gamma)}\Bigr)
  \end{multline}
  for all $T>0$, where $c>0$ is a constant independent of $\varepsilon$, $\lambda$, and $T$.
\end{theorem}

Note that the $L^2(\Omega_\varepsilon)$-norm is divided by $\varepsilon^{1/2}$ in \eqref{E:Df_CTD}, since it involves the square root of the thickness of $\Omega_\varepsilon$ which is of order $\varepsilon$.
Since $\|v_0\|_{L^2(\Gamma)}$ is independent of $\varepsilon$, we can say by \eqref{E:Df_CTD} that $v$ approximates $u^\varepsilon$ of order $\varepsilon$.

\subsection{Idea of proof} \label{SS:In_Idea}
We prove Theorems \ref{T:Weak} and \ref{T:Df_Sur} in Section \ref{S:TFL}.
A basic idea is similar to the cases of the heat equation \cite{Miu17} and the Navier--Stokes equations \cite{Miu20_03}.
We start form a weak form of the thin-domain problem \eqref{E:GL_CTD} and take the constant extension of a function on $\Gamma$ as a test function.
Then, we take the average in the thin direction of integrals over $\Omega_\varepsilon$ by using a change of variables formula that involves the Jacobian $J(y,r)$ used in the definition of the weighted average (see \eqref{E:Ch_Va} for the formula), and derive a weak form on $\Gamma$ satisfied by the weighted average $\mathcal{M}_\varepsilon u^\varepsilon$ of a weak solution $u^\varepsilon$ to \eqref{E:GL_CTD}.
The weak form of $\mathcal{M}_\varepsilon u^\varepsilon$ is similar to a weak form of the limit problem \eqref{E:GL_Lim}, but it contains a residual term coming from errors in the averages of the Dirichlet form corresponding to the Laplacian in \eqref{E:GL_CTD} and of the cubic term $|u^\varepsilon|^2u^\varepsilon$.
We estimate the residual term properly to show that it is sufficiently small as $\varepsilon\to0$.
After that, we derive an energy estimate for $\mathcal{M}_\varepsilon u^\varepsilon$ by using the weak form of $\mathcal{M}_\varepsilon u^\varepsilon$, which gives the weak convergence of $\mathcal{M}_\varepsilon u^\varepsilon$ (up to a subsequence) as $\varepsilon\to0$.
We also estimate the time derivative of $\mathcal{M}_\varepsilon u^\varepsilon$ and apply the Aubin--Lions lemma (see Lemma \ref{L:ET_AuLi}) to get the strong convergence of $\mathcal{M}_\varepsilon u^\varepsilon$.
This strong convergence result is essential for the weak convergence of the cubic term $|\mathcal{M}_\varepsilon u^\varepsilon|^2\mathcal{M}_\varepsilon u^\varepsilon$, which is shown by a weak version of the dominated convergence theorem (see Lemma \ref{L:AC_Non}).
Using the convergence results of $\mathcal{M}_\varepsilon u^\varepsilon$, we send $\varepsilon\to0$ in the weak form of $\mathcal{M}_\varepsilon u^\varepsilon$ and find that the limit function $v$ of $\mathcal{M}_\varepsilon u^\varepsilon$ is indeed a weak solution to the limit problem \eqref{E:GL_Lim}.
We also estimate the difference of $\mathcal{M}_\varepsilon u^\varepsilon$ and $v$ by taking the difference of the weak forms of $\mathcal{M}_\varepsilon u^\varepsilon$ and $v$ and by using an energy method.
Note that, in calculations of the energy method, the difference of the cubic terms can be neglected since
\begin{align} \label{E:In_Mono}
  (|a|^2a-|b|^2b)\cdot(a-b) \geq (|a|^3-|b|^3)(|a|-|b|) \geq 0
\end{align}
for all $a,b\in\mathbb{R}^N$ (see \eqref{Pf_GCU:Mono} for details).
This is also used in the proof of the uniqueness of weak solutions to \eqref{E:GL_CTD} and \eqref{E:GL_Lim} (see Lemmas \ref{L:GCT_Uni} and \ref{L:GLim_Uni}).

In the above arguments, the main difficulty arises in the estimate for the residual term appearing in the weak form of $\mathcal{M}_\varepsilon u^\varepsilon$.
As mentioned above, the residual term comes from errors in the averages of the Dirichlet form, i.e. the $L^2$-inner product of the gradients, and of the cubic term.
For the Dirichlet form, we can compute the tangential gradient of $\mathcal{M}_\varepsilon u^\varepsilon$ explicitly in terms of the gradient of $u^\varepsilon$ (see Lemma \ref{L:Ave_TGr}), so we can easily estimate the error in the average of the Dirichlet form (see Lemma \ref{L:AT_Diri}).
In fact, the error estimate for the average of the Dirichlet form in a curved thin domain was already shown in the study of the heat equation \cite{Miu17}, but the proof given there is based on somewhat complicated calculations under local coordinates of $\Gamma$ and $\Omega_\varepsilon$.
Here, we give a more direct proof without using local coordinates.

A more careful analysis is required in the error estimate for the average of the cubic term.
In the study of the thin-film limit of the Navier--Stokes eqautions \cite{Miu20_03}, we estimated the error of the average of a nonlinear term (the convection term) by using the Sobolev embedding with explicit dependence of constants on $\varepsilon$.
This may be a possible approach for our case, but we consider a weak solution $u^\varepsilon$ to \eqref{E:GL_CTD} which does not have an enough regularity.
Hence, instead, we use the (weak) maximum principle
\begin{align*}
  \|u^\varepsilon\|_{L^\infty(\Omega_\varepsilon\times(0,\infty))} \leq \max\{1,\|u_0^\varepsilon\|_{L^\infty(\Omega_\varepsilon)}\}
\end{align*}
under the assumption $u_0^\varepsilon\in L^\infty(\Omega_\varepsilon)^N$ (see Lemma \ref{L:GCT_Linf}).
With this in mind, we estimate the error of the average of the cubic term as
\begin{multline*}
    |(|u^\varepsilon|^2u^\varepsilon,\bar{\zeta})_{L^2(\Omega_\varepsilon)}-\varepsilon(g|\mathcal{M}_\varepsilon u^\varepsilon|^2\mathcal{M}_\varepsilon u^\varepsilon,\zeta)_{L^2(\Gamma)}| \leq c\varepsilon^{3/2}\|u^\varepsilon\|_{L^\infty(\Omega_\varepsilon)}^2\|u^\varepsilon\|_{H^1(\Omega_\varepsilon)}\|\zeta\|_{L^2(\Gamma)},
\end{multline*}
where $\zeta$ is a test function on $\Gamma$ and $\bar{\zeta}$ is the constant extension of $\zeta$ in the normal direction of $\Gamma$, and show that the right-hand side is sufficiently small as $\varepsilon\to0$ by using an energy estimate for $u^\varepsilon$, the maximum principle, and the assumption (a) of Theorem \ref{T:Weak}.
The proof of the above error estimate is based on the change of variables formula \eqref{E:Ch_Va} and an estimate for the difference of $|u^\varepsilon|^2$ and $|\mathcal{M}_\varepsilon u^\varepsilon|^2$.
For details, we refer to Section \ref{SS:Ave_Sq}.
We note that the use of the maximum principle enables us to avoid any restriction on the dimension of the curved thin domain, which is usually necessary if one uses the Sobolev embedding.
Also, we mention that the maximum principle was used in \cite{Miu23} for the proof of a difference estimate in the sup-norm of classical solutions to the heat equation in a moving thin domain and to a limit equation.

Let us also explain the idea of the proof of Theorem \ref{T:Df_CTD} (see Section \ref{S:Diff_CT} for the actual proof).
To derive the difference estimate \eqref{E:Df_CTD} in $\Omega_\varepsilon$, one may naturally consider the use of the difference estimate \eqref{E:Df_Sur} on $\Gamma$, since the difference of a weak solution $u^\varepsilon$ to \eqref{E:GL_CTD} and its weighted average $\mathcal{M}_\varepsilon u^\varepsilon$ is expected to be small as $\varepsilon\to0$.
This idea seems to work well, but the difference estimate requires a higher regularity of $u^\varepsilon$ like
\begin{align*}
  \bigl\|u^\varepsilon-\overline{\mathcal{M}_\varepsilon u^\varepsilon}\bigr\|_{L^2(\Omega_\varepsilon)} \leq c\varepsilon\|u^\varepsilon\|_{H^1(\Omega_\varepsilon)},
\end{align*}
where $\overline{\mathcal{M}_\varepsilon u^\varepsilon}$ is the constant extension of $\mathcal{M}_\varepsilon u^\varepsilon$ (see Lemma \ref{L:Ave_Dif} for the above estimate).
In particular, if we would like to derive \eqref{E:Df_CTD} from \eqref{E:Df_Sur} by using the difference estimate of $u^\varepsilon$ and $\mathcal{M}_\varepsilon u^\varepsilon$, then the $H^2$-regularity of $u^\varepsilon$ is required, which is not the case for a weak solution.
In fact, this idea was successful in the study of the Navier--Stokes equations \cite{Miu20_03}, since a strong solution to the thin-domain problem was used in that study.
To circumvent the above issue, we instead consider a weak solution $v$ to the limit problem \eqref{E:GL_Lim} as a weak solution to the thin-domain problem \eqref{E:GL_CTD} with some error.
More precisely, we take a test function $\psi$ defined on $\Omega_\varepsilon$ and substitute its weighted average $\mathcal{M}_\varepsilon\psi$ for a test function of \eqref{E:GL_Lim}.
Then, we ``unfold'' the weighted average $\mathcal{M}_\varepsilon$, i.e. compute like
\begin{align*}
  &\varepsilon\int_\Gamma g(y)v(y)\cdot\mathcal{M}_\varepsilon\psi(y)\,d\mathcal{H}^{n-1}(y) \\
  &= \int_\Gamma\left(\int_{\varepsilon g_0(y)}^{\varepsilon g_1(y)}v(y)\cdot\psi\bigl(y+r\nu(y)\bigr)J(y,r)\,dr\right)\,d\mathcal{H}^{n-1}(y) \\
  &= \int_{\Omega_\varepsilon}\bar{v}(x)\cdot\psi(x)\,dx
\end{align*}
by the change of variables formula \eqref{E:Ch_Va}.
Here, $\mathcal{H}^{n-1}$ is the Hausdorff measure of dimension $n-1$ and $\bar{v}$ is the constant extension of $v$ in the normal direction of $\Gamma$.
As a result, we get a weak form in $\Omega_\varepsilon$ satisfied by $\bar{v}$ that is similar to a weak form of \eqref{E:GL_CTD} but has a residual term.
Now, the residual term consists only of an error in the unfolding of the Dirichlet form, since all other terms are linear and of zeroth order in $\psi$.
In particular, we can recover the cubic term of \eqref{E:GL_CTD} from that of \eqref{E:GL_Lim} in a weak form without any error.
This enables us to remove the assumption $u_0^\varepsilon\in L^\infty(\Omega_\varepsilon)^N$ in Theorem \ref{T:Df_CTD}.
Moreover, the residual term is estimated in terms of $\|\psi\|_{H^1(\Omega_\varepsilon)}$ and $\|v\|_{H^1(\Gamma)}$ by Lemma \ref{L:AT_Diri}, and the latter norm is further estimated by an energy estimate for $v$ (when integrated in time).
Hence, we can get the difference estimate \eqref{E:Df_CTD} in $\Omega_\varepsilon$ by taking the difference of the weak forms of $u^\varepsilon$ and $\bar{v}$ and by applying an energy method with the aid of \eqref{E:In_Mono}.

The notion of ``unfolding'' has been used in the study of homogenization (see e.g. \cite{CiDaGr18}), but the idea of unfolding the (weighted) average in a weak form seems to be somewhat new in the study of the thin-film limit of PDEs.
Also, we point out that the idea explained here is applicable to the problem
\begin{align*}
  \partial_tu^\varepsilon-\Delta u^\varepsilon+f(u^\varepsilon) = 0 \quad\text{in}\quad \Omega_\varepsilon\times(0,\infty)
\end{align*}
under the Neumann boundary condition, where $f\colon\mathbb{R}^N\to\mathbb{R}^N$ is any nonlinearity satisfying the monotonicity $\{f(a)-f(b)\}\cdot(a-b)\geq0$ for all $a,b\in\mathbb{R}^N$ as in \eqref{E:In_Mono}.

Lastly, we note that we do not rescale the thickness of $\Omega_\varepsilon$ in all of the proofs in this paper.
This makes the proofs a quite readable, since the rescaling argument in the case of a curved thin domain results in highly involved expressions of rescaled equations due to the nonconstant curvatures of the limit hypersurface $\Gamma$.

\subsection{Organization of the paper} \label{SS:In_Org}
The rest of this paper is organized as follows.
We fix notations on a closed hypersurface and a curved thin domain in Section \ref{S:Prelim}.
Section \ref{S:Func_Sp} provides some results of function spaces used in analysis of \eqref{E:GL_CTD} and \eqref{E:GL_Lim}.
In Section \ref{S:Weak_Sol}, we define weak solutions to \eqref{E:GL_CTD} and \eqref{E:GL_Lim} and give basic results on weak solutions.
Section \ref{S:Ave} is devoted to analysis of the weighted average operator.
In Section \ref{S:TFL}, we study the thin-film limit problem of \eqref{E:GL_CTD} and establish Theorems \ref{T:Weak} and \ref{T:Df_Sur}.
Also, we prove Theorem \ref{T:Df_CTD} in Section \ref{S:Diff_CT}.
Section \ref{S:Galerkin} gives the outline of the Galerkin method for construction of a weak solution to \eqref{E:GL_CTD}.
The proofs of Lemmas \ref{L:ET_Den}, \ref{L:Trn_Sp}, \ref{L:Trn_Ti} are given in Section \ref{S:Pf_Aux}.

\section{Preliminaries} \label{S:Prelim}
We fix notations on a closed hypersurface and a curved thin domain.
Throughout this paper, the symbol $c$ denotes a general positive constant independent of the parameter $\varepsilon$.

\subsection{Basic notations} \label{SS:Pr_Nota}
We fix a coordinate system of $\mathbb{R}^n$ with $n\in\mathbb{N}$ and write $x_i$ for the $i$-th component of a point $x\in\mathbb{R}^n$ under the fixed coordinate system.
A vector $a\in\mathbb{R}^n$ and a matrix $A\in\mathbb{R}^{n\times N}$ are written as
\begin{align*}
  a =
  \begin{pmatrix}
    a_1 \\ \vdots \\ a_n
  \end{pmatrix}
  = (a_1,\dots,a_n)^T, \quad
  A = (A_{ij})_{i,j} =
  \begin{pmatrix}
    A_{11} & \cdots & A_{1N} \\
    \vdots & & \vdots \\
    A_{n1} & \cdots & A_{nN}
  \end{pmatrix}.
\end{align*}
We denote by $A^T$ for the transpose of $A$ and by $I_n$ the $n\times n$ identity matrix.
Let $a\cdot b$ be the inner product of $a,b\in\mathbb{R}^n$ and $|a|=\sqrt{a\cdot a}$ be the Euclidean norm of $a$.
We set the inner product $A:B=\mathrm{tr}[A^TB]$ of $A,B\in\mathbb{R}^{n\times N}$ and the Frobenius norm $|A|=\sqrt{A:A}$.

For a scalar-valued function $\varphi$ on $\mathbb{R}^n$, let
\begin{align*}
  \nabla\varphi = (\partial_1\varphi,\dots,\partial_n\varphi)^T, \quad \nabla^2\varphi = (\partial_i\partial_j\varphi)_{i,j}
\end{align*}
be the gradient and the Hessian matrix of $f$, respectively, where $\partial_i=\partial/\partial x_i$.
Also, when $u=(u_1,\dots,u_N)^T$ is an $\mathbb{R}^N$-valued function on $\mathbb{R}^n$, we write
\begin{align*}
  \nabla u =
  \begin{pmatrix}
    \nabla u_1 & \cdots & \nabla u_N
  \end{pmatrix}
  =
  \begin{pmatrix}
    \partial_1u_1 & \cdots & \partial_1u_N \\
    \vdots & & \vdots \\
    \partial_nu_1 & \cdots & \partial_nu_N
  \end{pmatrix}.
\end{align*}
Note that $\nabla(u\circ\Phi)=(\nabla\Phi)\nabla u$ for $\Phi\colon\mathbb{R}^n\to\mathbb{R}^n$ under our notation.

We write $X'$ and $\langle\cdot,\cdot\rangle_X$ for the dual space of a Banach space $X$ and the duality product between $X'$ and $X$.
For spaces $\mathcal{X}(S)$ and $\mathcal{Y}(S)$ of scalar-valued functions on a set $S$, let
\begin{align*}
  (\mathcal{X}\cap\mathcal{Y})(S) = \mathcal{X}(S)\cap\mathcal{Y}(S), \quad \mathcal{X}(S)^N = \{u=(u_1,\dots,u_N)^T \mid u_1,\dots,u_N\in\mathcal{X}(S)\}.
\end{align*}
When we use the norm and the inner product on $\mathcal{X}(S)^N$, we suppress the superscript $N$ and write $\|\cdot\|_{\mathcal{X}(S)}$ and $(\cdot,\cdot)_{\mathcal{X}(S)}$, respectively.
Similarly, we denote by $\langle\cdot,\cdot\rangle_{\mathcal{X}(S)}$ the duality product between $[\mathcal{X}(S)^N]'$ and $\mathcal{X}(S)^N$.

\subsection{Closed hypersurface} \label{SS:Pr_Sur}
Let $\Gamma$ be a $C^3$ closed (i.e., compact and without boundary), connected, and oriented hypersurface in $\mathbb{R}^n$.
We assume that $\Gamma$ is the boundary of a bounded domain $\Omega$ in $\mathbb{R}^n$ and write $\nu$ for the unit outward normal vector field of $\Gamma$ which points from $\Omega$ into $\mathbb{R}^n\setminus\Omega$.
Let $d$ be the signed distance function from $\Gamma$ increasing in the direction of $\nu$.
Also, let $\kappa_1,\dots,\kappa_{n-1}$ be the principal curvatures of $\Gamma$.
Then, $\nu$ is of class $C^2$ and $\kappa_1,\dots,\kappa_{n-1}$ are of class $C^1$ on $\Gamma$ by the regularity of $\Gamma$, and they are bounded on $\Gamma$ since $\Gamma$ is compact.
Hence, there exists a tubular neighborhood $\mathcal{N}_\delta=\{x\in\mathbb{R}^n \mid -\delta<d(x)<\delta\}$ of $\Gamma$ with radius $\delta>0$ such that each $x\in\overline{\mathcal{N}_\delta}$ has a unique $\pi(x)\in\Gamma$ satisfying
\begin{align} \label{E:Fermi}
  x = \pi(x)+d(x)\nu(\pi(x)), \quad \nabla d(x) = \nu(\pi(x)).
\end{align}
Moreover, $d$ and $\pi$ are of class $C^2$ and $C^1$ on $\overline{\mathcal{N}_\delta}$, respectively (see \cite[Section 14.6]{GilTru01}).
Taking $\delta>0$ sufficiently small, we may also assume that
\begin{align} \label{E:Curv}
  c^{-1} \leq 1-r\kappa_\alpha(y) \leq c \quad\text{for all}\quad y\in\Gamma, \, r\in[-\delta,\delta], \, \alpha=1,\dots,n-1
\end{align}
with some constant $c>0$.

Let $P=I_n-\nu\otimes\nu$ be the orthogonal projection onto the tangent plane of $\Gamma$, where $\nu\otimes\nu=(\nu_i\nu_j)_{i,j}$ is the tensor product of $\nu$ with itself.
We define the tangential gradient of $\eta\in C^1(\Gamma)$ by $\nabla_\Gamma\eta=P\nabla\tilde{\eta}$ on $\Gamma$, where $\tilde{\eta}$ is an extension of $\eta$ to $\mathcal{N}_\delta$.
Also, we write $\underline{D}_i\eta$ for the $i$-th component of $\nabla_\Gamma\eta$ and call it the $i$-th tangential derivative of $\eta$ for $i=1,\dots,n$.
Here, the value of $\nabla_\Gamma\eta$ is independent of the choice of the extension $\tilde{\eta}$.
Moreover,
\begin{align} \label{E:TG_Pnu}
  P\nabla_\Gamma\eta = \nabla_\Gamma\eta, \quad \nu\cdot\nabla_\Gamma\eta = 0 \quad\text{on}\quad \Gamma.
\end{align}
Let $\bar{\eta}=\eta\circ\pi$ be the constant extension of $\eta$ in the normal direction of $\Gamma$.
Then,
\begin{align} \label{E:CoDe_Sur}
  \nabla\bar{\eta}(y) = \nabla_\Gamma\eta(y), \quad \partial_i\bar{\eta}(y) = \underline{D}_i\eta(y), \quad y\in\Gamma,
\end{align}
since $\nabla\pi(y)=P(y)$ for $y\in\Gamma$ by \eqref{E:Fermi} and $d(y)=0$.
In what follows, we always write $\bar{\eta}$ for the constant extension of a function $\eta$ on $\Gamma$ in the normal direction of $\Gamma$.

When $v=(v_1,\dots,v_N)^T$ is an $\mathbb{R}^N$-valued function on $\Gamma$, we write
\begin{align*}
  \nabla_\Gamma v =
  \begin{pmatrix}
    \nabla_\Gamma v_1 & \cdots & \nabla_\Gamma v_N
  \end{pmatrix}
  =
  \begin{pmatrix}
    \underline{D}_1v_1 & \cdots & \underline{D}_1v_N \\
    \vdots & & \vdots \\
    \underline{D}_nv_1 & \cdots & \underline{D}_nv_N
  \end{pmatrix}.
\end{align*}
Some authors define $\nabla_\Gamma v$ as the transpose of the above matrix.
Under our notation, we have $\nabla_\Gamma v=P\nabla\tilde{v}$ on $\Gamma$ for any extension $\tilde{v}$ of $v$ to $\mathcal{N}_\delta$.
When $N=n$, we define the surface divergence of $v=(v_1,\dots,v_n)^T$ by $\mathrm{div}_\Gamma v=\mathrm{tr}[\nabla_\Gamma v]$.
Moreover, for
\begin{align*}
  A =
  \begin{pmatrix}
    A_1 & \cdots & A_N
  \end{pmatrix}\colon\Gamma\to\mathbb{R}^{n\times N},
  \quad\text{where}\quad A_1,\dots,A_N\colon\Gamma\to\mathbb{R}^n,
\end{align*}
we define the $\mathbb{R}^N$-valued function $\mathrm{div}_\Gamma A=(\mathrm{div}_\Gamma A_1,\dots,\mathrm{div}_\Gamma A_N)^T$ on $\Gamma$.

Let $W=-\nabla_\Gamma\nu$ and $H=\mathrm{tr}[W]$ on $\Gamma$.
We call $W$ and $H$ the Weingarten map (or the shape operator) and the mean curvature of $\Gamma$, respectively.
Since $\Gamma$ is of class $C^2$, the functions $W$ and $H$ are of class $C^1$ and thus bounded on $\Gamma$.
Moreover, it follows from  \eqref{E:Fermi} and \eqref{E:CoDe_Sur} that $W=-\nabla\bar{\nu}=-\nabla^2d$ on $\Gamma$.
Hence, $W$ is symmetric.
We also have
\begin{align*}
  W\nu = -\nabla_\Gamma(|\nu|^2/2) = 0, \quad PW = WP = W \quad\text{on}\quad \Gamma
\end{align*}
by $|\nu|=1$.
The first relation shows that $W$ has the eigenvalue zero.
It is also known (see e.g. \cite{Lee18}) that the other eigenvalues of $W$ are $\kappa_1,\dots,\kappa_{n-1}$.
Hence, taking an orthonormal basis of $\mathbb{R}^n$ which consists of eigenvectors of $W$, and using the boundedness of $\kappa_1,\dots,\kappa_{n-1}$ on $\Gamma$ and \eqref{E:Curv}, we easily observe that $I_n-rW(y)$ is invertible and
\begin{align} \label{E:W_Inv}
  \bigl|\{I_n-rW(y)\}^{-1}\bigr| \leq c, \quad \bigl|I_n-\{I_n-rW(y)\}^{-1}\bigr| \leq c|r|
\end{align}
for all $y\in\Gamma$ and $r\in[-\delta,\delta]$ with some constant $c>0$.

\begin{lemma} \label{L:Pro_Gr}
  Let $\pi$ be the mapping given in \eqref{E:Fermi}.
  Then,
  \begin{align} \label{E:Pro_Gr}
    \nabla\pi(x) = \bigl\{I_n-d(x)\overline{W}(x)\bigr\}^{-1}\overline{P}(x), \quad x\in\overline{\mathcal{N}_\delta}.
  \end{align}
  Let $\eta\in C^1(\Gamma)$ and $\bar{\eta}=\eta\circ\pi$ be the constant extension of $\eta$.
  Then,
  \begin{align} \label{E:CoDr_N}
    \nabla\bar{\eta}(x) = \bigl\{I_n-d(x)\overline{W}(x)\bigr\}^{-1}\overline{\nabla_\Gamma\eta}(x), \quad x\in\overline{\mathcal{N}_\delta}.
  \end{align}
  Moreover, there exists a constant $c>0$ independent of $\eta$ such that
  \begin{align} \label{E:CDN_Ineq}
    |\nabla\bar{\eta}(x)| \leq c\bigl|\overline{\nabla_\Gamma\eta}(x)\bigr|, \quad \bigl|\nabla\bar{\eta}(x)-\overline{\nabla_\Gamma\eta}(x)\bigr| \leq c|d(x)|\bigl|\overline{\nabla_\Gamma\eta}(x)\bigr|, \quad x\in\overline{\mathcal{N}_\delta}.
  \end{align}
\end{lemma}

\begin{proof}
  By \eqref{E:Fermi}, we have $\pi(x)=x-d(x)\bar{\nu}(\pi(x))$.
  We differentiate both sides and use
  \begin{align*}
    \nabla d(x) = \bar{\nu}(\pi(x)) = \bar{\nu}(x), \quad \nabla\bar{\nu}(\pi(x)) = \nabla_\Gamma\nu(\pi(x)) = -W(\pi(x)) = -\overline{W}(x),
  \end{align*}
  which follow from $\pi(x)\in\Gamma$, \eqref{E:Fermi}, and \eqref{E:CoDe_Sur}.
  Then, we get
  \begin{align*}
    \nabla\pi(x) = I_n-\bar{\nu}(x)\otimes\bar{\nu}(x)+d(x)\nabla\pi(x)\overline{W}(x) = \overline{P}(x)+d(x)\nabla\pi(x)\overline{W}(x).
  \end{align*}
  By this equality and $PW=WP$ on $\Gamma$, we have \eqref{E:Pro_Gr}.
  Also, we differentiate $\bar{\eta}(x)=\bar{\eta}(\pi(x))$, use \eqref{E:Pro_Gr}, and apply \eqref{E:TG_Pnu} and \eqref{E:CoDe_Sur} with $y=\pi(x)\in\Gamma$ to get \eqref{E:CoDr_N}.
  We also have \eqref{E:CDN_Ineq} by \eqref{E:W_Inv} and \eqref{E:CoDr_N} with $y=\pi(x)$ and $r=d(x)$.
\end{proof}

\begin{lemma} \label{L:Pull}
  For a function $\varphi$ on $\mathcal{N}_\delta$, we define
  \begin{align} \label{E:Def_Pull}
    \varphi^\sharp(y,r) = \varphi\bigl(y+r\nu(y)\bigr), \quad y\in\Gamma,\,r\in(-\delta,\delta).
  \end{align}
  Then, the tangential gradient $\nabla_\Gamma\varphi^\sharp(y,r)$ with respect to the variable $y$ is of the form
  \begin{align} \label{E:TG_Pull}
    \nabla_\Gamma\varphi^\sharp(y,r) = \{P(y)-rW(y)\}(\nabla\varphi)^\sharp(y,r).
  \end{align}
\end{lemma}

\begin{proof}
  We extend $\Gamma\ni y\mapsto\varphi^\sharp(y,r)$ to $N$ by $\tilde{\varphi}^\sharp(x,r)=\varphi(x+r\bar\nu(x))$.
  Then, we differentiate both sides with respect to $x$ and set $x=y\in\Gamma$ to get
  \begin{align*}
    \nabla\tilde{\varphi}^\sharp(y,r) = \{I_n+r\nabla\bar{\nu}(y)\}(\nabla\varphi)\bigl(y+r\bar{\nu}(y)\bigr) = \{I_n-rW(y)\}(\nabla\varphi)^\sharp(y,r)
  \end{align*}
  by \eqref{E:CoDe_Sur}.
  By this equality and $PW=W$ on $\Gamma$, we obtain \eqref{E:TG_Pull}.
\end{proof}

We write $\int_\Gamma\eta\,d\mathcal{H}^{n-1}$ for the integral of a function $\eta$ on $\Gamma$, where $\mathcal{H}^{n-1}$ is the Hausdorff measure of dimension $n-1$.
Also, we denote by $\|\cdot\|_{L^p(\Gamma)}$ and $(\cdot,\cdot)_{L^2(\Gamma)}$ the $L^p$-norm and the $L^2$-inner product on $\Gamma$, respectively.
For $\eta,\zeta\in C^1(\Gamma)$, it is known (see \cite[Lemma 16.1]{GilTru01} and \cite[Theorem 2.10]{DziEll13}) that the integration by parts formula
\begin{align} \label{E:IBP_Sur}
  \int_\Gamma\zeta\underline{D}_i\eta\,d\mathcal{H}^{n-1} = -\int_\Gamma\eta(\underline{D}_i\zeta+\zeta H\nu_i)\,d\mathcal{H}^{n-1}, \quad i=1,\dots,n
\end{align}
holds.
Based on this formula, we say that $\eta\in L^2(\Gamma)$ has the $i$-th weak tangential derivative if there exists a function $\underline{D}_i\eta\in L^2(\Gamma)$ such that \eqref{E:IBP_Sur} holds for all $\zeta\in C^1(\Gamma)$.
Moreover, we define the Sobolev space $H^1(\Gamma)$ and its inner product by
\begin{gather*}
  H^1(\Gamma) = \{\eta\in L^2(\Gamma) \mid \nabla_\Gamma\eta=(\underline{D}_1\eta,\dots,\underline{D}_n\eta)^T\in L^2(\Gamma)^n\}, \\
  (\eta,\zeta)_{H^1(\Gamma)} = (\eta,\zeta)_{L^2(\Gamma)}+(\nabla_\Gamma\eta,\nabla_\Gamma\zeta)_{L^2(\Gamma)}.
\end{gather*}
Note that $C^1(\Gamma)$ is dense in $H^1(\Gamma)$ by localization and mollification arguments.

\subsection{Curved thin domain} \label{SS:Pr_CTD}
Let $g_0,g_1\in C^1(\Gamma)$ and $g=g_1-g_0$.
We assume that there exists a constant $c>0$ such that
\begin{align} \label{E:G_Bdd}
  c^{-1} \leq g(y) \leq c, \quad y\in\Gamma.
\end{align}
For a small $\varepsilon>0$, we define the curved thin domain $\Omega_\varepsilon$ by \eqref{E:Def_CTD}.
Its boundary is
\begin{align*}
  \partial\Omega_\varepsilon = \Gamma_\varepsilon^0\cup\Gamma_\varepsilon^1, \quad \Gamma_\varepsilon^i = \{y+\varepsilon g_i(y)\nu(y) \mid y\in\Gamma\}, \quad i=0,1,
\end{align*}
so we see that $\partial\Omega_\varepsilon$ is of class $C^1$ by the regularity of $\Gamma$, $g_0$, and $g_1$.

Since $g_0$ and $g_1$ are bounded on $\Gamma$, we can take a constant $\tilde{\varepsilon}\in(0,1)$ such that $\varepsilon|g_i|<\delta$ on $\Gamma$ for all $\varepsilon\in(0,\tilde{\varepsilon})$ and $i=0,1$, where $\delta$ is the radius of the tubular neighborhood $\mathcal{N}_\delta$ of $\Gamma$ given in Section \ref{SS:Pr_Sur}.
Then, $\overline{\Omega}_\varepsilon\subset\mathcal{N}_\delta$ and we can use the results of Section \ref{SS:Pr_Sur} in $\overline{\Omega}_\varepsilon$ for all $\varepsilon\in(0,\tilde{\varepsilon})$.
In what follows, we always assume that $\varepsilon\in(0,\tilde{\varepsilon})$.
Also, sometimes we use the relation $0<\varepsilon<1$ without mention.

For $y\in\Gamma$ and $r\in[-\delta,\delta]$, we define
\begin{align} \label{E:Def_J}
  J(y,r) = \det[I_n-rW(y)] = \prod_{\alpha=1}^{n-1}\{1-r\kappa_\alpha(y)\}.
\end{align}
Since \eqref{E:Curv} holds and $\kappa_1,\dots,\kappa_{n-1}$ are bounded on $\Gamma$, we see that
\begin{align} \label{E:J_Bdd}
  c^{-1} \leq J(y,r) \leq c, \quad |\partial_rJ(y,r)| \leq c \quad y\in\Gamma, \, r\in[-\delta,\delta].
\end{align}
Moreover, since $g_0$ and $g_1$ are bounded on $\Gamma$ and $\kappa_1,\dots,\kappa_{n-1}\in C^1(\Gamma)$, we have
\begin{align} \label{E:J_TGr}
  |J(y,r)-1| \leq c\varepsilon, \quad |\nabla_\Gamma J(y,r)| \leq c\varepsilon, \quad y\in\Gamma, \, r\in[\varepsilon g_0(y),\varepsilon g_1(y)],
\end{align}
where $\nabla_\Gamma J$ is the tangential gradient of $J$ with respect to the variable $y\in\Gamma$.
The function $J$ appears as the Jacobian of the change of variables formula
\begin{align} \label{E:Ch_Va}
  \int_{\Omega_\varepsilon}\varphi(x)\,dx = \int_\Gamma\int_{\varepsilon g_0(y)}^{\varepsilon g_1(y)}\varphi^\sharp(y,r)J(y,r)\,dr\,d\mathcal{H}^{n-1}(y)
\end{align}
for a function $\varphi$ on $\Omega_\varepsilon$, where $\varphi^\sharp$ is given by \eqref{E:Def_Pull} (see \cite[Section 14.6]{GilTru01}).
When $p\in[1,\infty)$ and $\varphi\in L^p(\Omega_\varepsilon)$, we see by \eqref{E:J_Bdd} and \eqref{E:Ch_Va} that
\begin{align} \label{E:CVF_Lp}
  c^{-1}\|\varphi\|_{L^p(\Omega_\varepsilon)}^p \leq \int_\Gamma\int_{\varepsilon g_0(y)}^{\varepsilon g_1(y)}|\varphi^\sharp(y,r)|^p\,dr\,d\mathcal{H}^{n-1}(y) \leq c\|\varphi\|_{L^p(\Omega_\varepsilon)}^p.
\end{align}
For $\eta\in L^p(\Gamma)$, it follows from \eqref{E:G_Bdd}, \eqref{E:J_Bdd}, and \eqref{E:CVF_Lp} that
\begin{align} \label{E:CoEx_Lp}
  c^{-1}\varepsilon^{1/p}\|\eta\|_{L^p(\Gamma)} \leq \|\bar{\eta}\|_{L^p(\Omega_\varepsilon)} \leq c\varepsilon^{1/p}\|\eta\|_{L^p(\Gamma)},
\end{align}
where $\bar{\eta}=\eta\circ\pi$ is the constant extension of $\eta$.
Also, when $\eta\in H^1(\Gamma)$, we use \eqref{E:CDN_Ineq} and $|d|\leq c\varepsilon$ in $\Omega_\varepsilon$, and then apply \eqref{E:CoEx_Lp} with $\eta$ replaced by $\nabla_\Gamma\eta$ to get
\begin{align} \label{E:CoDe_L2}
  \|\nabla\bar{\eta}\|_{L^2(\Omega_\varepsilon)} \leq c\varepsilon^{1/2}\|\nabla_\Gamma\eta\|_{L^2(\Gamma)}, \quad \bigl\|\nabla\bar{\eta}-\overline{\nabla_\Gamma\eta}\bigr\|_{L^2(\Omega_\varepsilon)} \leq c\varepsilon^{3/2}\|\nabla_\Gamma\eta\|_{L^2(\Gamma)}.
\end{align}

\section{Results of function spaces} \label{S:Func_Sp}
We give some results of function spaces used in analysis of \eqref{E:GL_CTD} and \eqref{E:GL_Lim}.

\subsection{Abstract theory} \label{SS:FS_Ab}
Let $X_0$ and $X_1$ be Banach spaces such that both $X_0$ and $X_1$ are continuously embedded into a Hausdorff topological vector space $\mathcal{V}$.
Then, the intersection $X_0\cap X_1$ and the sum $X_0+X_1=\{u_0+u_1\mid u_0\in X_0, \, u_1\in X_1\}$ equipped with norms
\begin{align} \label{E:Def_X0X1}
  \begin{aligned}
    \|u\|_{X_0\cap X_1} &= \max\{\|u\|_{X_0},\|u\|_{X_1}\}, \\
    \|u\|_{X_0+X_1} &= \inf\{\|u_0\|_{X_0}+\|u_1\|_{X_1} \mid u = u_0+u_1, \, u_0\in X_0, \, u_1\in X_1\}
  \end{aligned}
\end{align}
are Banach spaces (see e.g. \cite[Chapter 3, Theorem 1.3]{BenSha88}).
For $i=0,1$, let $X_i'$ be the dual space of $X_i$.
If $X_0\cap X_1$ is dense in both $X_0$ and $X_1$, then $X_0'$ and $X_1'$ are continuously embedded into $[X_0\cap X_1]'$, which is a Banach space and thus a Hausdorff topological vector space.
Hence, we can consider $X_0'+X_1'$ as a Banach space equipped with norm $\|\cdot\|_{X_0'+X_1'}$ defined as above.

\begin{lemma} \label{L:XY_Dual}
  Suppose that $X_0\cap X_1$ is dense in $X_0$ and $X_1$.
  Then,
  \begin{align*}
    [X_0\cap X_1]' = X_0'+X_1' = \{f=f_0+f_1 \mid f_0\in X_0', f_1\in X_1'\}.
  \end{align*}
  More precisely, $f\in[X_0\cap X_1]'$ if and only if there exist $f_0\in X_0$ and $f_1\in X_1$ such that
  \begin{align*}
    \langle f,u \rangle_{X_0\cap X_1} = \langle f_0,u\rangle_{X_0}+\langle f_1,u\rangle_{X_1} \quad\text{for all}\quad u\in X_0\cap X_1.
  \end{align*}
  Moreover, $\|f\|_{[X_0\cap X_1]'}=\|f\|_{X_0'+X_1'}$ for all $f\in [X_0\cap X_1]'$.
\end{lemma}

\begin{proof}
  We refer to \cite[Theorem 2.7.1]{BerLof76} for the proof.
\end{proof}

\subsection{Function spaces} \label{SS:FS_TS}
For $S=\Omega_\varepsilon$ or $S=\Gamma$, let
\begin{align*}
  X_0 = H^1(S)^N, \quad X_1 = L^4(S)^N, \quad \mathcal{V} = L^2(S)^N.
\end{align*}
Clearly, $X_0$ is continuously embedded into $\mathcal{V}$.
Since $S$ is bounded, $X_1$ is also continuously embedded into $\mathcal{V}$ by H\"{o}lder's inequality.
Also, $X_0\cap X_1$ is dense in both $X_0$ and $X_1$, since it contains the dense subspace $C^1(\overline{S})^N$ of $X_0$ and $X_1$.
Hence, by Lemma \ref{L:XY_Dual},
\begin{align} \label{E:Dual_Sp}
  [(H^1\cap L^4)(S)^N]' = [H^1(S)^N]'+L^{4/3}(S)^N \quad (\text{$S=\Omega_\varepsilon$ or $S=\Gamma$}).
\end{align}
Moreover, since $S$ is bounded, we can consider
\begin{align*}
  (H^1\cap L^4)(S)^N \hookrightarrow L^2(S)^N \hookrightarrow [(H^1\cap L^4)(S)^N]' = [H^1(S)^N]'+L^{4/3}(S)^N
\end{align*}
by setting $\langle u,v\rangle_{(H^1\cap L^4)(S)}=(u,v)_{L^2(S)}$ for $u\in L^2(S)^N$ and $v\in(H^1\cap L^4)(S)^N$.

Similarly, if we set
\begin{align} \label{E:Def_ZT}
  \begin{aligned}
    Z_T(S) &= L^2(0,T;H^1(S)^N)\cap L^4(0,T;L^4(S)^N), \\
    \|u\|_{Z_T(S)} &= \max\{\|u\|_{L^2(0,T;H^1(S))},\|u\|_{L^4(0,T;L^4(S))}\}
  \end{aligned}
\end{align}
for $S=\Omega_\varepsilon$ or $S=\Gamma$, and for each fixed $T>0$, then
\begin{align} \label{E:Dual_ZT}
  [Z_T(S)]' = L^2(0,T;[H^1(S)^N]')+L^{4/3}(0,T;L^{4/3}(S)^N).
\end{align}
In particular,
\begin{align*}
  [Z_T(S)]' \subset L^1(0,T;[H^1(S)^N]'+L^{4/3}(S)^N) = L^1(0,T;[(H^1\cap L^4)(S)^N]'),
\end{align*}
and $f(t)$ makes sense in $[(H^1\cap L^4)(S)^N]'$ for $f\in[Z_T(S)]'$ and a.a. $t\in(0,T)$.

Let $u,v\in L^1(0,T;[(H^1\cap L^4)(S)^N]')$.
We write $v=\partial_tu$ if
\begin{align*}
  \int_0^T\langle u(t),\partial_t\psi(t)\rangle_{(H^1\cap L^4)(S)}\,dt = -\int_0^T\langle v(t),\psi(t)\rangle_{(H^1\cap L^4)(S)}\,dt
\end{align*}
for all $\psi\in C_c^1(0,T;(H^1\cap L^4)(S)^N)$.
Moreover, we define
\begin{align} \label{E:Def_ET}
  \begin{aligned}
    E_T(S) &= \{u\in Z_T(S) \mid \partial_tu\in[Z_T(S)]'\}, \\
    \|u\|_{E_T(S)} &= \|u\|_{Z_T(S)}+\|\partial_tu\|_{[Z_T(S)]'}.
  \end{aligned}
\end{align}
Note that $\partial_tu\in [Z_T(S)]'$ if and only if there exists a constant $c>0$ such that
\begin{align*}
  \left|\int_0^T\langle u(t),\partial_t\psi(t)\rangle_{(H^1\cap L^4)(S)}\,dt\right| \leq c\|\psi\|_{Z_T(S)} \quad\text{for all}\quad \psi \in C_c^1(0,T;(H^1\cap L^4)(S)^N),
\end{align*}
since $C_c^1(0,T;(H^1\cap L^4)(S)^N)$ is dense in $Z_T(S)$ by the next lemma.
In what follows, we use this fact without mention.

\begin{lemma} \label{L:ET_Den}
  Let $S=\Omega_\varepsilon$ or $S=\Gamma$.
  Then,
  \begin{enumerate}
    \item $C_c^\infty(0,T;(H^1\cap L^4)(S)^N)$ is dense in $Z_T(S)$, and
    \item $C^\infty([0,T];(H^1\cap L^4)(S)^N)$ is dense in $E_T(S)$.
  \end{enumerate}
\end{lemma}

The statement (i) can be shown by standard cut-off and mollification arguments, so we omit the proof.
Also, the proof of (ii) is similar to the one of a density result for
\begin{align*}
  \{u\in L^p(0,T;\mathcal{X}) \mid \partial_tu\in L^q(0,T;\mathcal{Y})\}, \quad \mathcal{X}\hookrightarrow\mathcal{Y}, \quad p,q\in(1,\infty)
\end{align*}
with Banach spaces $\mathcal{X}$ and $\mathcal{Y}$ (see e.g. \cite[Lemma II.5.10]{BoyFab13}), but we need to carefully deal with $\partial_tu$ for $u\in E_T(S)$ in our case because of the structure \eqref{E:Dual_ZT} of $[Z_T(S)]'$.
We give the proof of (ii) in Section \ref{S:Pf_Aux} for the completeness.

\begin{lemma} \label{L:ET_Con}
  Let $S=\Omega_\varepsilon$ or $S=\Gamma$.
  Then, the continuous embedding
  \begin{align*}
    E_T(S) \hookrightarrow C([0,T];L^2(S)^N)
  \end{align*}
  holds.
  Moreover, for all $u_1,u_2\in E_T(S)$, we have
  \begin{align} \label{E:ET_IbP}
    \frac{d}{dt}\bigl(u_1(t),u_2(t)\bigr)_{L^2(S)} = \langle\partial_tu_1(t),u_2(t)\rangle_{(H^1\cap L^4)(S)}+\langle\partial_tu_2(t),u_1(t)\rangle_{(H^1\cap L^4)(S)}
  \end{align}
  in $\mathcal{D}'(0,T)$, the space of distributions on $(0,T)$.
\end{lemma}

\begin{proof}
  The proof is the same as in the Hilbertian case shown in \cite[Theorems II.5.12 and II.5.13]{BoyFab13}, if we apply Lemma \ref{L:ET_Den}, (ii) and the next lemma.
  We omit details.
\end{proof}

\begin{lemma} \label{L:ET_Pair}
  Let $S=\Omega_\varepsilon$ or $S=\Gamma$.
  For $f\in[Z_T(S)]'$ and $u\in Z_T(S)$, let
  \begin{align*}
    [\Phi(f,u)](t) = \langle f(t),u(t)\rangle_{(H^1\cap L^4)(S)}, \quad t\in(0,T).
  \end{align*}
  Then, $\Phi$ is a bilinear continuous mapping from $[Z_T(S)]'\times Z_T(S)$ into $L^1(0,T)$.
\end{lemma}

Note that this lemma is not obvious because of the structure \eqref{E:Dual_ZT} of $[Z_T(S)]'$.

\begin{proof}
  Clearly, $\Phi$ is bilinear.
  Let us show
  \begin{align} \label{Pf_EP:Bound}
    \|\Phi(f,u)\|_{L^1(0,T)} \leq \|f\|_{[Z_T(S)]'}\|u\|_{Z_T(S)} \quad\text{for all}\quad f\in [Z_T(S)]', \, u\in Z_T(S).
  \end{align}
  Let $f\in[Z_T(S)]'$.
  Since $[Z_T(S)]'$ is of the form \eqref{E:Dual_ZT} and the norm $\|\cdot\|_{X_0+X_1}$ is given by \eqref{E:Def_X0X1} for Banach spaces $X_0$ and $X_1$, we can take $F_k$ and $G_k$ such that
  \begin{gather*}
    f = F_k+G_k, \quad F_k\in L^2(0,T;[H^1(S)^N]'), \quad G_k\in L^{4/3}(0,T;L^{4/3}(S)^N), \\
    \lim_{k\to\infty}\Bigl(\|F_k\|_{L^2(0,T;[H^1(S)]')}+\|G_k\|_{L^{4/3}(0,T;L^{4/3}(S))}\Bigr) = \|f\|_{[Z_T(S)]'}.
  \end{gather*}
  Then, since $f(t)=F_k(t)+G_k(t)$ in $[(H^1\cap L^4)(S)^N]'$ for a.a. $t\in(0,T)$, we have
  \begin{align*}
    |[\Phi(f,u)](t)| &= \Bigl|\langle F_k(t),u(t)\rangle_{H^1(S)}+\bigl(G_k(t),u(t)\bigr)_{L^2(S)}\Bigr| \\
    &\leq \|F_k(t)\|_{[H^1(S)]'}\|u(t)\|_{H^1(S)}+\|G_k(t)\|_{L^{4/3}(S)}\|u(t)\|_{L^4(S)}
  \end{align*}
  for $u\in Z_T(S)$ and a.a. $t\in(0,T)$.
  Hence,
  \begin{align*}
    \|\Phi(f,u)\|_{L^1(0,T)} \leq \Bigl(\|F_k\|_{L^2(0,T;[H^1(S)]')}+\|G_k\|_{L^{4/3}(0,T;L^{4/3}(S))}\Bigr)\|u\|_{Z_T(S)}
  \end{align*}
  by H\"{o}lder's inequality and \eqref{E:Def_ZT}, and we get \eqref{Pf_EP:Bound} by sending $k\to\infty$.
\end{proof}

\begin{lemma} \label{L:ET_AuLi}
  Let $S=\Omega_\varepsilon$ or $S=\Gamma$.
  Then, the embedding
  \begin{align*}
    E_T(S) \hookrightarrow L^2(0,T;L^2(S)^N)
  \end{align*}
  is compact.
\end{lemma}

\begin{proof}
  We see that $E_T(S)$ is continuously embedded into
  \begin{align*}
    \widetilde{E}_T(S) = \{u\in L^2(0,T;H^1(S)^N) \mid \partial_tu\in L^1(0,T;[(H^1\cap L^4)(S)^N]')\}
  \end{align*}
  by the definition of $E_T(S)$ and \eqref{E:Dual_ZT}.
  Moreover, since the embeddings
  \begin{align*}
    H^1(S)^N \hookrightarrow L^2(S)^N \hookrightarrow [(H^1\cap L^4)(S)^N]'
  \end{align*}
  are continuous and the first embedding is compact, the embedding
  \begin{align*}
    \widetilde{E}_T(S) \hookrightarrow L^2(0,T;L^2(S)^N)
  \end{align*}
  is compact by the Aubin--Lions lemma (see \cite[Theorem II.5.16]{BoyFab13}).
  Hence, the embedding
  \begin{align*}
    E_T(S) \hookrightarrow \widetilde{E}_T(S) \hookrightarrow L^2(0,T;L^2(S)^N)
  \end{align*}
  is also compact.
\end{proof}

Now, let $S=\Omega_\varepsilon$.
For $z\in\mathbb{R}$, we define $z_+=\max\{z,0\}$.

\begin{lemma} \label{L:Trn_Sp}
  Let $C_0>0$ be a constant.
  We define a mapping $F\colon\mathbb{R}^N\to\mathbb{R}^N$ by
  \begin{align} \label{E:Def_Trn}
    F(a) = \frac{(|a|-C_0)_+}{|a|}\,a, \quad a\in\mathbb{R}^N.
  \end{align}
  Then, for all $u\in H^1(\Omega_\varepsilon)^N$, we have $F(u)\in H^1(\Omega_\varepsilon)^N$ and
  \begin{align} \label{E:Trn_Dxi}
    \frac{\partial}{\partial x_i}\Bigl(F_j(u)\Bigr) = \frac{(|u|-C_0)_+}{|u|}\frac{\partial u_j}{\partial x_i}+\sum_{k=1}^N\chi_{(C_0,\infty)}(|u|)\frac{C_0u_ju_k}{|u|^3}\frac{\partial u_k}{\partial x_i} \quad\text{a.e. in}\quad \Omega_\varepsilon
  \end{align}
  for $i=1,\dots,n$ and $j=1,\dots,N$.
  Here, $F_j$ is the $j$-th component of $F$ and $\chi_{(C_0,\infty)}(z)$ is the characteristic function of the interval $(C_0,\infty)$.
\end{lemma}

\begin{lemma} \label{L:Trn_Ti}
  Let $C_0>0$ be a constant and $F$ be given by \eqref{E:Def_Trn}.
  Then, $F(u)\in Z_T(\Omega_\varepsilon)$ for all $u\in Z_T(\Omega_\varepsilon)$.
  Moreover, if $u\in E_T(\Omega_\varepsilon)$, then
  \begin{align} \label{E:Trn_Dt}
    \frac{1}{2}\frac{d}{dt}\Bigl(\|(|u(t)|-C_0)_+\|_{L^2(\Omega_\varepsilon)}^2\Bigr) = \bigl\langle\partial_tu(t),F\bigl(u(t)\bigr)\bigr\rangle_{(H^1\cap L^4)(\Omega_\varepsilon)} \quad\text{in}\quad \mathcal{D}'(0,T).
  \end{align}
\end{lemma}

Lemmas \ref{L:Trn_Sp} and \ref{L:Trn_Ti} can be shown by approximation of $|a|$ and $z_+$ by $C^1$ functions.
We give the proofs in Section \ref{S:Pf_Aux} for the completeness.

Lastly, we recall the weak dominated convergence theorem, which is used to show the weak convergence of the cubic term in \eqref{E:GL_CTD} and \eqref{E:GL_Lim} in approximation of solutions.

\begin{lemma} \label{L:AC_Non}
  Let $S=\Omega_\varepsilon$ or $S=\Gamma$.
  Also, let $p\in(1,\infty)$ and $T\in(0,\infty)$.
  Suppose that
  \begin{align*}
    \psi_k,\psi\in L^p(0,T;L^p(S)^N), \quad \sup_{k\in\mathbb{N}}\|\psi_k\|_{L^p(0,T;L^p(S))} < \infty,
  \end{align*}
  and $\psi_k\to\psi$ a.e. on $S\times(0,T)$ as $k\to\infty$.
  Then,
  \begin{align*}
    \lim_{k\to\infty}\psi_k = \psi \quad\text{weakly in}\quad L^p(0,T;L^p(S)^N).
  \end{align*}
\end{lemma}

\begin{proof}
  We refer to \cite[Lemma 8.3]{Rob01} for the proof.
\end{proof}

\section{Definition and properties of weak solutions} \label{S:Weak_Sol}
This section provides the definition and some properties of weak solutions to the thin-domain problem \eqref{E:GL_CTD} and the limit problem \eqref{E:GL_Lim}.
For $S=\Omega_\varepsilon$ or $S=\Gamma$, let $Z_T(S)$ and $E_T(S)$ be given by \eqref{E:Def_ZT} and \eqref{E:Def_ET}, respectively.
We abuse the notation
\begin{align*}
  (\varphi_1,\varphi_2)_{L^2(S)} = \int_S\varphi_1\varphi_2\,d\mu, \quad d\mu =
  \begin{cases}
    dx &(S=\Omega_\varepsilon), \\
    d\mathcal{H}^{n-1} &(S=\Gamma)
  \end{cases}
\end{align*}
for $\varphi_1\in L^p(S)$ and $\varphi_2\in L^q(S)$ with $p,q\in[1,\infty]$ satisfying $1/p+1/q=1$.

\subsection{Thin-domain problem} \label{SS:WS_Thin}
First we consider \eqref{E:GL_CTD}.
Based on integration by parts and the Neumann boundary condition, we define a weak solution to \eqref{E:GL_CTD} as follows.

\begin{definition} \label{D:WSC_Loc}
  For $T>0$ and a given $u_0^\varepsilon\in L^2(\Omega_\varepsilon)^N$, we say that a function $u^\varepsilon$ is a weak solution to \eqref{E:GL_CTD} on $[0,T)$ if $u^\varepsilon\in E_T(\Omega_\varepsilon)$ and it satisfies
  \begin{multline} \label{E:WF_GCT}
    \int_0^T\langle\partial_tu^\varepsilon(t),\psi(t)\rangle_{(H^1\cap L^4)(\Omega_\varepsilon)}\,dt+\int_0^T\bigl(\nabla u^\varepsilon(t),\nabla\psi(t)\bigr)_{L^2(\Omega_\varepsilon)}\,dt \\
    +\lambda\int_0^T\bigl((|u^\varepsilon(t)|^2-1)u^\varepsilon(t),\psi(t)\bigr)_{L^2(\Omega_\varepsilon)}\,dt = 0
  \end{multline}
  for all $\psi\in Z_T(\Omega_\varepsilon)$ and $u^\varepsilon(0)=u_0^\varepsilon$ in $L^2(\Omega_\varepsilon)^N$.
\end{definition}

\begin{definition} \label{D:WSC_Glo}
  For a given $u_0^\varepsilon\in L^2(\Omega_\varepsilon)^N$, we say that a function $u^\varepsilon$ is a global weak solution to \eqref{E:GL_CTD} if it is a weak solution to \eqref{E:GL_CTD} on $[0,T)$ for all $T>0$.
\end{definition}

Note that the weak formulation \eqref{E:WF_GCT} makes sense, since
\begin{align*}
  \left|\int_0^T\bigl(|u^\varepsilon(t)|^2u^\varepsilon(t),\psi(t)\bigr)_{L^2(\Omega_\varepsilon)}\,dt\right| \leq \|u^\varepsilon\|_{L^4(0,T;L^4(\Omega_\varepsilon))}^3\|\psi\|_{L^4(0,T;L^4(\Omega_\varepsilon))}
\end{align*}
by H\"{o}lder's inequality.
The initial condition also makes sense by Lemma \ref{L:ET_Con}.

Let us give basic results on a weak solution to \eqref{E:GL_CTD}.

\begin{lemma} \label{L:GCT_Ena}
  Let $u^\varepsilon$ be a weak solution to \eqref{E:GL_CTD} on $[0,T)$ with initial data $u_0^\varepsilon\in L^2(\Omega_\varepsilon)^N$.
  Then, for all $t\in[0,T]$, we have
  \begin{align} \label{E:GCT_Ena}
    \|u^\varepsilon(t)\|_{L^2(\Omega_\varepsilon)}^2+2\int_0^t\|\nabla u^\varepsilon(s)\|_{L^2(\Omega_\varepsilon)}^2\,ds+2\lambda\int_0^t\|u^\varepsilon(s)\|_{L^4(\Omega_\varepsilon)}^4\,ds \leq e^{2\lambda t}\|u_0^\varepsilon\|_{L^2(\Omega_\varepsilon)}^2.
  \end{align}
\end{lemma}

Note that the factor $e^{2\lambda t}$ in the right-hand side of \eqref{E:GCT_Ena} is independent of $\varepsilon$.

\begin{proof}
  Let $\psi=u^\varepsilon$ in \eqref{E:WF_GCT} with $T$ replaced by each $t\in[0,T]$.
  Then, by \eqref{E:ET_IbP},
  \begin{multline*}
    \frac{1}{2}\|u^\varepsilon(t)\|_{L^2(\Omega_\varepsilon)}^2+\int_0^t\|\nabla u^\varepsilon(s)\|_{L^2(\Omega_\varepsilon)}^2\,ds+\lambda\int_0^t\|u^\varepsilon(s)\|_{L^4(\Omega_\varepsilon)}^4\,ds \\
    = \frac{1}{2}\|u_0^\varepsilon\|_{L^2(\Omega_\varepsilon)}^2+\lambda\int_0^t\|u^\varepsilon(s)\|_{L^2(\Omega_\varepsilon)}^2\,ds
  \end{multline*}
  and thus $\|u^\varepsilon(t)\|_{L^2(\Omega_\varepsilon)}^2\leq\|u_0^\varepsilon\|_{L^2(\Omega_\varepsilon)}^2+2\lambda\int_0^t\|u^\varepsilon(s)\|_{L^2(\Omega_\varepsilon)}^2\,ds$ by $\lambda>0$.
  Hence,
  \begin{align*}
    \|u^\varepsilon(t)\|_{L^2(\Omega_\varepsilon)}^2 \leq e^{2\lambda t}\|u_0^\varepsilon\|_{L^2(\Omega_\varepsilon)}^2 \quad\text{for all}\quad t\in[0,T]
  \end{align*}
  by Gronwall's inequality, and we get \eqref{E:GCT_Ena} by the above relations.
\end{proof}

\begin{lemma} \label{L:GCT_Uni}
  For each $T>0$ and $u_0^\varepsilon\in L^2(\Omega_\varepsilon)^N$, there exists at most one weak solution to \eqref{E:GL_CTD} on $[0,T)$.
\end{lemma}

\begin{proof}
  Let $u_1^\varepsilon$ and $u_2^\varepsilon$ be weak solutions to \eqref{E:GL_CTD} on $[0,T)$ with same initial data $u_0^\varepsilon$, and let $U^\varepsilon=u_1^\varepsilon-u_2^\varepsilon$.
  Then, for each $t\in[0,T]$ and $\psi\in Z_t(\Omega_\varepsilon)$, we take the difference of \eqref{E:WF_GCT} satisfied by $u_1^\varepsilon$ and $u_2^\varepsilon$ to get
  \begin{multline*}
    \int_0^t\langle\partial_sU^\varepsilon(s),\psi(s)\rangle_{(H^1\cap L^4)(\Omega_\varepsilon)}\,ds+\int_0^t\bigl(\nabla U^\varepsilon(s),\nabla\psi(s)\bigr)_{L^2(\Omega_\varepsilon)}\,ds \\
    +\lambda\int_0^t\bigl(|u_1^\varepsilon(s)|^2u_1^\varepsilon(s)-|u_2^\varepsilon(s)|^2u_2^\varepsilon(s),\psi(s)\bigr)_{L^2(\Omega_\varepsilon)}\,ds = \lambda\int_0^t\bigl(U^\varepsilon(s),\psi(s)\bigr)_{L^2(\Omega_\varepsilon)}\,ds.
  \end{multline*}
  Let $\psi=U^\varepsilon=u_1^\varepsilon-u_2^\varepsilon$ in this equality.
  Then, since $\lambda>0$ and
  \begin{align} \label{Pf_GCU:Mono}
    \begin{aligned}
      (|a|^2a-|b|^2b)\cdot(a-b) &= |a|^4-(|a|^2+|b|^2)(a\cdot b)+|b|^4 \\
      &\geq |a|^4-(|a|^2+|b|^2)|a||b|+|b|^4 \\
      &= (|a|^3-|b|^3)(|a|-|b|) \geq 0
    \end{aligned}
  \end{align}
  for all $a,b\in\mathbb{R}^N$, we see by \eqref{E:ET_IbP} and $U^\varepsilon(0)=u_1^\varepsilon(0)-u_2^\varepsilon(0)=0$ that
  \begin{align*}
    \frac{1}{2}\|U^\varepsilon(t)\|_{L^2(\Omega_\varepsilon)}^2 \leq \lambda\int_0^t\|U^\varepsilon(s)\|_{L^2(\Omega_\varepsilon)}^2\,ds.
  \end{align*}
  Thus, $\|U^\varepsilon(t)\|_{L^2(\Omega_\varepsilon)}^2=0$ for all $t\in[0,T]$ by Gronwall's inequality, i.e. $u_1^\varepsilon=u_2^\varepsilon$.
\end{proof}

\begin{theorem} \label{T:GCT_Exi}
  For each $u_0^\varepsilon\in L^2(\Omega_\varepsilon)^N$, there exists a unique global weak solution to \eqref{E:GL_CTD}.
\end{theorem}

\begin{proof}
  The uniqueness is due to Lemma \ref{L:GCT_Uni}.
  The existence can be shown by the Galerkin method as in the scalar-valued case described in \cite[Section 8.3]{Rob01}.
  In our case, however, the dimension of $\Omega_\varepsilon$ is arbitrary and the regularity of $\partial\Omega_\varepsilon$ is only $C^1$.
  Due to this, we cannot use the eigenfunctions of the Neumann Laplacian in $L^2(\Omega_\varepsilon)^N$ to approximate test functions of class $H^1\cap L^4$, since it requires the Sobolev embedding and the elliptic regularity theorem (see \cite[Section 8.2]{Rob01}).
  Hence, we need to take basis functions of $(H^1\cap L^4)(\Omega_\varepsilon)^N$ directly and modify some arguments.
  We give the outline of the Galerkin method in Section \ref{S:Galerkin} for the completeness.
\end{proof}

We also have the (weak) maximum principle for a weak solution to \eqref{E:GL_CTD}.

\begin{lemma} \label{L:GCT_Linf}
  Let $u^\varepsilon$ be a global weak solution to \eqref{E:GL_CTD} with $u_0^\varepsilon\in L^\infty(\Omega_\varepsilon)^N$.
  Then,
  \begin{align} \label{E:GCT_Linf}
    \|u^\varepsilon\|_{L^\infty(\Omega_\varepsilon\times(0,\infty))} \leq \max\{1,\|u_0^\varepsilon\|_{L^\infty(\Omega_\varepsilon)}\}.
  \end{align}
\end{lemma}

\begin{proof}
  We follow the argument of \cite[Lemma 3.7]{Du94} and \cite[Lemma 6]{MoShWa23}, but in the framework of weak solutions.

  Let $C_0=\max\{1,\|u_0^\varepsilon\|_{L^\infty(\Omega_\varepsilon)}\}$ and $F$ be given by \eqref{E:Def_Trn}.
  Since $u^\varepsilon \in E_T(\Omega_\varepsilon)$ for all $T>0$, it follows from Lemma \ref{L:Trn_Ti} that $F(u^\varepsilon)\in Z_T(\Omega_\varepsilon)$ and we can take $\psi=F(u^\varepsilon)$ in \eqref{E:WF_GCT} with $T$ replaced by each $t\geq0$.
  Then, we observe by \eqref{E:Trn_Dt} that
  \begin{multline*}
    \frac{1}{2}\|(|u^\varepsilon(t)|-C_0)_+\|_{L^2(\Omega_\varepsilon)}^2+\int_0^t\Bigl(\nabla u^\varepsilon(s),\nabla\bigl[F\bigl(u^\varepsilon(s)\bigr)\bigr]\Bigr)_{L^2(\Omega_\varepsilon)}\,ds \\
    +\lambda\int_0^t\Bigl((|u^\varepsilon(s)|^2-1)u^\varepsilon(s),F\bigl(u^\varepsilon(s)\bigr)\Bigr)_{L^2(\Omega_\varepsilon)}\,ds = \frac{1}{2}\|(|u_0^\varepsilon|-C_0)_+\|_{L^2(\Omega_\varepsilon)}^2.
  \end{multline*}
  In this equality, $(|u_0^\varepsilon|-C_0)_+=0$ a.e. in $\Omega_\varepsilon$ since $\|u_0^\varepsilon\|_{L^\infty(\Omega_\varepsilon)}\leq C_0$.
  Also,
  \begin{align*}
    \nabla u^\varepsilon:\nabla[F(u^\varepsilon)] &= \sum_{i=1}^n\sum_{j=1}^N\frac{\partial u^\varepsilon_j}{\partial x_i}\frac{\partial}{\partial x_i}\Bigl(F_j(u^\varepsilon)\Bigr) = J_1+J_2, \\
    J_1 &= \sum_{i=1}^n\sum_{j=1}^N\frac{(|u^\varepsilon|-C_0)_+}{|u^\varepsilon|}\left|\frac{\partial u_j^\varepsilon}{\partial x_i}\right|^2, \\
    J_2 &= \sum_{i=1}^n\sum_{j,k=1}^N\chi_{(C_0,\infty)}(|u^\varepsilon|)\frac{C_0u_j^\varepsilon u_k^\varepsilon}{|u^\varepsilon|^3}\frac{\partial u_j^\varepsilon}{\partial x_i}\frac{\partial u_k^\varepsilon}{\partial x_i}
  \end{align*}
  a.e. in $\Omega_\varepsilon\times(0,\infty)$ by \eqref{E:Trn_Dxi}.
  Clearly, $J_1\geq0$.
  Moreover,
  \begin{align*}
    J_2 &= \sum_{i=1}^n\sum_{j,k=1}^N\chi_{(C_0,\infty)}(|u^\varepsilon|)\frac{C_0}{|u^\varepsilon|^3}\cdot\frac{1}{2}\frac{\partial}{\partial x_i}\Bigl((u_j^\varepsilon)^2\Bigr)\cdot\frac{1}{2}\frac{\partial}{\partial x_i}\Bigl((u_k^\varepsilon)^2\Bigr) \\
    &= \frac{C_0}{4|u^\varepsilon|^3}\chi_{(C_0,\infty)}(|u^\varepsilon|)\Bigl|\nabla\bigl(|u^\varepsilon|^2\bigr)\Bigr|^2 \geq 0.
  \end{align*}
  Hence, $\nabla u^\varepsilon:\nabla[F(u^\varepsilon)]\geq0$ a.e. in $\Omega_\varepsilon\times(0,\infty)$.
  We further observe that
  \begin{align*}
    (|u^\varepsilon|^2-1)u^\varepsilon\cdot F(u^\varepsilon) = |u^\varepsilon|(|u^\varepsilon|^2-1)(|u^\varepsilon|-C_0)_+ \geq 0 \quad\text{a.e. in}\quad \Omega_\varepsilon\times(0,\infty),
  \end{align*}
  since $|u^\varepsilon|^2\geq1$ when $|u^\varepsilon|\geq C_0\geq1$.
  From the above results, we deduce that
  \begin{align*}
    \|(|u^\varepsilon(t)|-C_0)_+\|_{L^2(\Omega_\varepsilon)}^2 \leq 0, \quad\text{i.e.}\quad \|(|u^\varepsilon(t)|-C_0)_+\|_{L^2(\Omega_\varepsilon)}^2 = 0 \quad\text{for all}\quad t\geq0.
  \end{align*}
  Hence, $(|u^\varepsilon|-C_0)_+=0$, i.e. $|u^\varepsilon|\leq C_0$ a.e. in $\Omega_\varepsilon\times(0,\infty)$, which gives \eqref{E:GCT_Linf}.
\end{proof}

\subsection{Limit problem} \label{SS:WS_Lim}
Next we consider \eqref{E:GL_Lim}.
To give the definition of a weak solution, we take a test function $\zeta=(\zeta_1,\dots,\zeta_N)^T$ and multiply \eqref{E:GL_Lim} by $g\zeta$.
Then, since
\begin{align*}
  \mathrm{div}_\Gamma(g\nabla_\Gamma v) = \bigl(\mathrm{div}_\Gamma(g\nabla_\Gamma v_1),\dots,\mathrm{div}_\Gamma(g\nabla_\Gamma v_N)\bigr)^T, \quad \mathrm{div}_\Gamma(g\nabla_\Gamma v_j) = \sum_{i=1}^n\underline{D}_i(g\underline{D}_iv_j)
\end{align*}
for $v=(v_1,\dots,v_N)^T$ under our notations given in Section \ref{SS:Pr_Sur}, we have
\begin{align*}
  \int_\Gamma\mathrm{div}_\Gamma(g\nabla_\Gamma v)\cdot\zeta\,d\mathcal{H}^{n-1} &= \sum_{j=1}^N\int_\Gamma\{\mathrm{div}_\Gamma(g\nabla_\Gamma v_j)\}\zeta_j\,d\mathcal{H}^{n-1} \\
  &= -\sum_{j=1}^N\int_\Gamma g\nabla_\Gamma v_j\cdot(\nabla_\Gamma\zeta_j+\zeta_jH\nu)\,d\mathcal{H}^{n-1} \\
  &= -\int_\Gamma g\nabla_\Gamma v:\nabla_\Gamma\zeta\,d\mathcal{H}^{n-1}
\end{align*}
by \eqref{E:IBP_Sur} and $\nabla_\Gamma v_j\cdot\nu=0$ on $\Gamma$.
Thus, we define a weak solution to \eqref{E:GL_Lim} as follows.

\begin{definition} \label{D:WSL_Loc}
  For $T>0$ and a given $v_0\in L^2(\Gamma)^N$, we say that a function $v$ is a weak solution to \eqref{E:GL_Lim} on $[0,T)$ if $v\in E_T(\Gamma)$ and it satisfies
  \begin{multline} \label{E:WF_GLim}
    \int_0^T\langle\partial_tv(t),g\zeta(t)\rangle_{(H^1\cap L^4)(\Gamma)}\,dt+\int_0^T\bigl(g\nabla_\Gamma v(t),\nabla_\Gamma\zeta(t)\bigr)_{L^2(\Gamma)}\,dt \\
    +\lambda\int_0^T\bigl(g(|v(t)|^2-1)v(t),\zeta(t)\bigr)_{L^2(\Gamma)}\,dt = 0
  \end{multline}
  for all $\zeta\in Z_T(\Gamma)$ and $v(0)=v_0$ in $L^2(\Gamma)^N$.
\end{definition}

\begin{definition} \label{D:WSL_Glo}
  For a given $v_0\in L^2(\Gamma)^N$, we say that a function $v$ is a global weak solution to \eqref{E:GL_Lim} if it is a weak solution to \eqref{E:GL_Lim} on $[0,T)$ for all $T>0$.
\end{definition}

For a weak solution to \eqref{E:GL_Lim}, we have similar results to the ones in Section \ref{SS:WS_Thin} by the same proofs.
Here, we just give results used later and omit the proofs, but note that $g$ is independent of time and satisfies \eqref{E:G_Bdd}.
Also, a suitable basis for the Galerkin method is an orthonromal basis of $L^2(\Gamma)^N$ with respect to the weighted inner product $(g\cdot,\cdot)_{L^2(\Gamma)}$

\begin{lemma} \label{L:GLim_Ena}
  Let $v$ be a weak solution to \eqref{E:GL_Lim} on $[0,T)$ with initial data $v_0\in L^2(\Gamma)^N$.
  Then, for all $t\in[0,T]$, we have
  \begin{align} \label{E:GLim_Ena}
    \|v(t)\|_{L^2(\Gamma)}^2+\int_0^t\|\nabla_\Gamma v(s)\|_{L^2(\Gamma)}^2\,ds+\lambda\int_0^t\|v(s)\|_{L^4(\Gamma)}^4\,ds \leq ce^{c\lambda t}\|v_0\|_{L^2(\Gamma)}^2,
  \end{align}
  where $c>0$ is a constant depending only on the constants appearing in \eqref{E:G_Bdd}.
\end{lemma}

\begin{lemma} \label{L:GLim_Uni}
  For each $T>0$ and $v_0\in L^2(\Gamma)^N$, there exists at most one weak solution to \eqref{E:GL_Lim} on $[0,T)$.
\end{lemma}

\begin{theorem} \label{T:GLim_Exi}
  For each $v_0\in L^2(\Gamma)^N$, there exists a unique global solution to \eqref{E:GL_Lim}.
\end{theorem}

\section{Weighted average operator} \label{S:Ave}
In this section, we define and study a weighted average operator.
Recall that $c$ denotes a general positive constant independent of $\varepsilon$.
Also, we write $\bar{\eta}=\eta\circ\pi$ for the constant extension of a function $\eta$ on $\Gamma$ in the normal direction of $\Gamma$.

\subsection{Definition and basic properties} \label{SS:Ave_Def}
Let $J(y,r)$ be the function given by \eqref{E:Def_J}.
For a function $\varphi$ on $\Omega_\varepsilon$, we define the weighted average of $\varphi$ in the thin direction by
\begin{align} \label{E:Def_Ave}
  \mathcal{M}_\varepsilon\varphi(y) = \frac{1}{\varepsilon g(y)}\int_{\varepsilon g_0(y)}^{\varepsilon g_1(y)}\varphi\bigl(y+r\nu(y)\bigr)J(y,r)\,dr, \quad y\in\Gamma.
\end{align}
The main advantage of taking the weighted average is that the formula
\begin{align} \label{E:Ave_Pair}
  \int_{\Omega_\varepsilon}\varphi(x)\bar{\eta}(x)\,dx = \varepsilon\int_\Gamma g(y)\mathcal{M}_\varepsilon\varphi(y)\eta(y)\,d\mathcal{H}^{n-1}(y)
\end{align}
holds by \eqref{E:Ch_Va} for $\varphi\colon\Omega_\varepsilon\to\mathbb{R}$ and $\eta\colon\Gamma\to\mathbb{R}$, which does not include any residual term.

Let us give some basic properties of $\mathcal{M}_\varepsilon$.
For a function $\varphi$ on $\Omega_\varepsilon$, we write $\partial_\nu\varphi=\bar{\nu}\cdot\nabla\varphi$ for the derivative of $\varphi$ in the normal direction of $\Gamma$.
Also, we often use the notation \eqref{E:Def_Pull} in what follows.
Hence, we write the definition of $\mathcal{M}_\varepsilon\varphi$ as
\begin{align*}
  \mathcal{M}_\varepsilon\varphi(y) = \frac{1}{\varepsilon g(y)}\int_{\varepsilon g_0(y)}^{\varepsilon g_1(y)}\varphi^\sharp(y,r)J(y,r)\,dr.
\end{align*}
Note that, under the notation \eqref{E:Def_Pull}, we have
\begin{align*}
  \partial_r\varphi^\sharp(y,r) = (\partial_\nu\varphi)^\sharp(y,r), \quad \bar{\eta}^\sharp(y,r) = \eta(y)
\end{align*}
for a function $\varphi$ on $\Omega_\varepsilon$ and a function $\eta$ on $\Gamma$.

\begin{lemma} \label{L:Ave_Lp}
  Let $p\in[1,\infty)$ and $\varphi\in L^p(\Omega_\varepsilon)$.
  Then, $\mathcal{M}_\varepsilon\varphi\in L^p(\Gamma)$ and
  \begin{align} \label{E:Ave_Lp}
    \|\mathcal{M}_\varepsilon\varphi\|_{L^p(\Gamma)} \leq c\varepsilon^{-1/p}\|\varphi\|_{L^p(\Omega_\varepsilon)}.
  \end{align}
\end{lemma}

\begin{proof}
  By \eqref{E:G_Bdd}, \eqref{E:J_Bdd}, and H\"{o}lder's inequality, we have
  \begin{align*}
    |\mathcal{M}_\varepsilon\varphi(y)| \leq c\varepsilon^{-1/p}\left(\int_{\varepsilon g_0(y)}^{\varepsilon g_1(y)}|\varphi^\sharp(y,r)|^p\,dr\right)^{1/p}, \quad y\in\Gamma.
  \end{align*}
  We integrate the $p$-th power of both sides over $\Gamma$ and use \eqref{E:CVF_Lp} to get \eqref{E:Ave_Lp}.
\end{proof}

\begin{lemma} \label{L:Ave_Dif}
  Let $\varphi\in H^1(\Omega_\varepsilon)$.
  Then,
  \begin{align} \label{E:Ave_Dif}
    \bigl\|\varphi-\overline{\mathcal{M}_\varepsilon\varphi}\bigr\|_{L^2(\Omega_\varepsilon)} \leq c\varepsilon\Bigl(\|\varphi\|_{L^2(\Omega_\varepsilon)}+\|\partial_\nu\varphi\|_{L^2(\Omega_\varepsilon)}\Bigr).
  \end{align}
\end{lemma}

\begin{proof}
  Let $y\in\Gamma$ and $r\in(\varepsilon g_0(y),\varepsilon g_1(y))$.
  Then,
  \begin{align*}
    \varphi^\sharp(y,r)-\mathcal{M}_\varepsilon\varphi(y) = \varphi^\sharp(y,r)-\frac{1}{\varepsilon g(y)}\int_{\varepsilon g_0(y)}^{\varepsilon g_1(y)}\varphi^\sharp(y,r_1)J(y,r_1)\,dr_1 = \frac{K_1+K_2}{\varepsilon g(y)},
  \end{align*}
  where
  \begin{align*}
    K_1 = \int_{\varepsilon g_0(y)}^{\varepsilon g_1(y)}\{\varphi^\sharp(y,r)-\varphi^\sharp(y,r_1)\}\,dr, \quad K_2 = \int_{\varepsilon g_0(y)}^{\varepsilon g_1(y)}\varphi^\sharp(y,r_1)\{1-J(y,r_1)\}\,dr_1.
  \end{align*}
  For $r,r_1\in(\varepsilon g_0(y),\varepsilon g_1(y))$, we have
  \begin{align*}
    |\varphi^\sharp(y,r)-\varphi^\sharp(y,r_1)| = \left|\int_{r_1}^r\partial_r\varphi^\sharp(y,r_2)\,dr_2\right| \leq \int_{\varepsilon g_0(y)}^{\varepsilon g_1(y)}|\partial_r\varphi^\sharp(y,r_2)|\,dr_2.
  \end{align*}
  Here, the right-hand side is independent of $r$ and $r_1$.
  Hence,
  \begin{align*}
    |K_1| \leq \varepsilon g(y)\int_{\varepsilon g_0(y)}^{\varepsilon g_1(y)}|\partial_r\varphi^\sharp(y,r_2)|\,dr_2 = \varepsilon g(y)\int_{\varepsilon g_0(y)}^{\varepsilon g_1(y)}|(\partial_\nu\varphi)^\sharp(y,r_1)|\,dr_1,
  \end{align*}
  where we used $\partial_r\varphi^\sharp=(\partial_\nu\varphi)^\sharp$ under the notation \eqref{E:Def_Pull} and replaced the variable $r_2$ by $r_1$ in the second equality.
  Also, it follows from \eqref{E:J_TGr} that
  \begin{align*}
    |K_2| \leq c\varepsilon\int_{\varepsilon g_0(y)}^{\varepsilon g_1(y)}|\varphi^\sharp(y,r_1)|\,dr_1.
  \end{align*}
  By these estimates, \eqref{E:G_Bdd}, and H\"{o}lder's inequality, we see that
  \begin{align} \label{Pf_AD:Est}
    \begin{aligned}
      |\varphi^\sharp(y,r)-\mathcal{M}_\varepsilon\varphi(y)| &\leq c\int_{\varepsilon g_0(y)}^{\varepsilon g_1(y)}\Bigl[|\varphi^\sharp|+|(\partial_\nu\varphi)^\sharp|\Bigr](y,r_1)\,dr_1 \\
      &\leq c\varepsilon^{1/2}\left(\int_{\varepsilon g_0(y)}^{\varepsilon g_1(y)}\Bigl[|\varphi^\sharp|^2+|(\partial_\nu\varphi)^\sharp|^2\Bigr](y,r_1)\,dr_1\right)^{1/2}.
    \end{aligned}
  \end{align}
  Here, we used the notation
  \begin{align*}
    \Bigl[|\varphi^\sharp|+|(\partial_\nu\varphi)^\sharp|\Bigr](y,r_1) = |\varphi^\sharp(y,r_1)|+|(\partial_\nu\varphi)^\sharp(y,r_1)|
  \end{align*}
  and a similar one for $|\varphi^\sharp|^2+|(\partial_\nu\varphi)^\sharp|^2$.
  We integrate the square of \eqref{Pf_AD:Est} with respect to $y$ and $r$.
  Then, since the right-hand side is independent of $r$, we have
  \begin{multline*}
    \int_\Gamma\int_{\varepsilon g_0(y)}^{\varepsilon g_1(y)}|\varphi^\sharp(y,r)-\mathcal{M}_\varepsilon\varphi(y)|^2\,dr\,d\mathcal{H}^{n-1}(y) \\
    \leq c\varepsilon^2\int_\Gamma\int_{\varepsilon g_0(y)}^{\varepsilon g_1(y)}\Bigl[|\varphi^\sharp|^2+|(\partial_\nu\varphi)^\sharp|^2\Bigr](y,r_1)\,dr_1\,d\mathcal{H}^{n-1}(y),
  \end{multline*}
  where we also used \eqref{E:G_Bdd}.
  By this inequality and \eqref{E:CVF_Lp}, we get \eqref{E:Ave_Dif}.
\end{proof}

\subsection{Weighted average of the square} \label{SS:Ave_Sq}
Let $\mathcal{M}_\varepsilon(|\varphi|^2)$ be the weighted average of $|\varphi|^2$.
We see that $\mathcal{M}_\varepsilon(|\varphi|^2)$ is close to $|\mathcal{M}_\varepsilon\varphi|^2$ when $\varepsilon$ is small.

\begin{lemma} \label{L:Ave_Sq}
  Let $\varphi\in (H^1\cap L^\infty)(\Omega_\varepsilon)$.
  Then,
  \begin{align} \label{E:Ave_Sq}
    \bigl\|\mathcal{M}_\varepsilon(|\varphi|^2)-|\mathcal{M}_\varepsilon\varphi|^2\bigr\|_{L^2(\Gamma)} \leq c\varepsilon^{1/2}\|\varphi\|_{L^\infty(\Omega_\varepsilon)}\Bigl(\|\varphi\|_{L^2(\Omega_\varepsilon)}+\|\partial_\nu\varphi\|_{L^2(\Omega_\varepsilon)}\Bigr).
  \end{align}
\end{lemma}

\begin{proof}
  Let $y\in\Gamma$.
  Since
  \begin{align*}
    [\mathcal{M}_\varepsilon(|\varphi|^2)](y) &= \frac{1}{\varepsilon g(y)}\int_{\varepsilon g_0(y)}^{\varepsilon g_1(y)}\varphi^\sharp(y,r)\varphi^\sharp(y,r)J(y,r)\,dr, \\
    |\mathcal{M}_\varepsilon\varphi(y)|^2 &= \frac{1}{\varepsilon g(y)}\int_{\varepsilon g_0(y)}^{\varepsilon g_1(y)}\varphi^\sharp(y,r)\mathcal{M}_\varepsilon\varphi(y)J(y,r)\,dr,
  \end{align*}
  we observe by \eqref{E:G_Bdd}, \eqref{E:J_Bdd}, and H\"{o}lder's inequality that
  \begin{align*}
    &\bigl|[\mathcal{M}_\varepsilon(|\varphi|^2)](y)-|\mathcal{M}_\varepsilon\varphi(y)|^2\bigr| \\
    &\qquad \leq \frac{1}{\varepsilon g(y)}\int_{\varepsilon g_0(y)}^{\varepsilon g_1(y)}|\varphi^\sharp(y,r)\{\varphi^\sharp(y,r)-\mathcal{M}_\varepsilon\varphi(y)\}|J(y,r)\,dr \\
    &\qquad \leq c\varepsilon^{-1}\|\varphi\|_{L^\infty(\Omega_\varepsilon)}\int_{\varepsilon g_0(y)}^{\varepsilon g_1(y)}|\varphi^\sharp(y,r)-\mathcal{M}_\varepsilon\varphi(y)|\,dr \\
    &\qquad \leq c\varepsilon^{-1/2}\|\varphi\|_{L^\infty(\Omega_\varepsilon)}\left(\int_{\varepsilon g_0(y)}^{\varepsilon g_1(y)}|\varphi^\sharp(y,r)-\mathcal{M}_\varepsilon\varphi(y)|^2\,dr\right)^{1/2}.
  \end{align*}
  We integrate the square of both sides and use \eqref{E:CVF_Lp} to get
  \begin{align*}
    \bigl\|\mathcal{M}_\varepsilon(|\varphi|^2)-|\mathcal{M}_\varepsilon\varphi|^2\bigr\|_{L^2(\Gamma)} \leq c\varepsilon^{-1/2}\|\varphi\|_{L^\infty(\Omega_\varepsilon)}\bigl\|\varphi-\overline{\mathcal{M}_\varepsilon\varphi}\bigr\|_{L^2(\Omega_\varepsilon)}.
  \end{align*}
  Applying \eqref{E:Ave_Dif} to the right-hand side, we obtain \eqref{E:Ave_Sq}.
\end{proof}

Using Lemmas \ref{L:Ave_Dif} and \ref{L:Ave_Sq}, we can compare $|\varphi|^2$ and $|\mathcal{M}_\varepsilon\varphi|^2$ in $L^2(\Omega_\varepsilon)$.

\begin{lemma} \label{L:ASq_Dif}
  Let $\varphi\in (H^1\cap L^\infty)(\Omega_\varepsilon)$.
  Then,
  \begin{align} \label{E:ASq_Dif}
    \Bigl\|\,|\varphi|^2-\bigl|\overline{\mathcal{M}_\varepsilon\varphi}\bigr|^2\,\Bigr\|_{L^2(\Omega_\varepsilon)} \leq c\varepsilon\|\varphi\|_{L^\infty(\Omega_\varepsilon)}\Bigl(\|\varphi\|_{L^2(\Omega_\varepsilon)}+\|\partial_\nu\varphi\|_{L^2(\Omega_\varepsilon)}\Bigr).
  \end{align}
\end{lemma}

\begin{proof}
  Noting that $\partial_\nu(|\varphi|^2)=2\varphi\partial_\nu\varphi$, we see by \eqref{E:Ave_Dif} that
  \begin{align*}
    \Bigl\|\,|\varphi|^2-\overline{\mathcal{M}_\varepsilon(|\varphi|^2)}\,\Bigr\|_{L^2(\Omega_\varepsilon)} &\leq c\varepsilon\Bigl(\bigl\|\,|\varphi|^2\,\bigr\|_{L^2(\Omega_\varepsilon)}+\|\partial_\nu(|\varphi|^2)\|_{L^2(\Omega_\varepsilon)}\Bigr) \\
    &\leq c\varepsilon\|\varphi\|_{L^\infty(\Omega_\varepsilon)}\Bigl(\|\varphi\|_{L^2(\Omega_\varepsilon)}+\|\partial_\nu\varphi\|_{L^2(\Omega)}\Bigr).
  \end{align*}
  Also, it follows from \eqref{E:CoEx_Lp} and \eqref{E:Ave_Sq} that
  \begin{align*}
    \Bigl\|\,\overline{\mathcal{M}_\varepsilon(|\varphi|^2)}-\bigl|\overline{\mathcal{M}_\varepsilon\varphi}\bigr|^2\,\Bigr\|_{L^2(\Omega_\varepsilon)} &\leq c\varepsilon^{1/2}\bigl\|\mathcal{M}_\varepsilon(|\varphi|^2)-|\mathcal{M}_\varepsilon\varphi|^2\bigr\|_{L^2(\Gamma)} \\
    &\leq c\varepsilon\|\varphi\|_{L^\infty(\Omega_\varepsilon)}\Bigl(\|\varphi\|_{L^2(\Omega_\varepsilon)}+\|\partial_\nu\varphi\|_{L^2(\Omega_\varepsilon)}\Bigr).
  \end{align*}
  Applying these estimates to
  \begin{align*}
    |\varphi|^2-\bigl|\overline{\mathcal{M}_\varepsilon\varphi}\bigr|^2 = \Bigl\{|\varphi|^2-\overline{\mathcal{M}_\varepsilon(|\varphi|^2)}\Bigr\}+\Bigl\{\overline{\mathcal{M}_\varepsilon(|\varphi|^2)}-\bigl|\overline{\mathcal{M}_\varepsilon\varphi}\bigr|^2\Bigr\},
  \end{align*}
  we obtain \eqref{E:ASq_Dif}.
\end{proof}

We also observe that the cubic term of \eqref{E:GL_CTD} is approximated by that of \eqref{E:GL_Lim} and the weighted average in the following weak sense.

\begin{lemma} \label{L:Ave_Non}
  Let $u\in (H^1\cap L^\infty)(\Omega_\varepsilon)^N$ and $\zeta\in L^2(\Gamma)^N$.
  Then,
  \begin{multline} \label{E:Ave_Non}
    \left|\int_{\Omega_\varepsilon}|u|^2u\cdot\bar{\zeta}\,dx-\varepsilon\int_\Gamma g|\mathcal{M}_\varepsilon u|^2\mathcal{M}_\varepsilon u\cdot\zeta\,d\mathcal{H}^{n-1}\right| \\
    \leq c\varepsilon^{3/2}\|u\|_{L^\infty(\Omega_\varepsilon)}^2\Bigl(\|u\|_{L^2(\Omega_\varepsilon)}+\|\partial_\nu u\|_{L^2(\Omega_\varepsilon)}\Bigr)\|\zeta\|_{L^2(\Gamma)}.
  \end{multline}
\end{lemma}

\begin{proof}
  We split the difference as
  \begin{align*}
    &\int_{\Omega_\varepsilon}|u|^2u\cdot\bar{\zeta}\,dx-\varepsilon\int_\Gamma g|\mathcal{M}_\varepsilon u|^2\mathcal{M}_\varepsilon u\cdot\zeta\,d\mathcal{H}^{n-1} = K_1+K_2, \\
    &K_1 = \int_{\Omega_\varepsilon}|u|^2u\cdot\bar{\zeta}\,dx-\int_{\Omega_\varepsilon}\bigl|\overline{\mathcal{M}_\varepsilon u}\bigr|^2u\cdot\bar{\zeta}\,dx, \\
    &K_2 = \int_{\Omega_\varepsilon}\bigl|\overline{\mathcal{M}_\varepsilon u}\bigr|^2u\cdot\bar{\zeta}\,dx-\varepsilon\int_{\Gamma}g|\mathcal{M}_\varepsilon u|^2\mathcal{M}_\varepsilon u\cdot\zeta\,d\mathcal{H}^{n-1}.
  \end{align*}
  Since $\mathcal{M}_\varepsilon u$ and $\zeta$ are functions on $\Gamma$, we have $K_2=0$ by  \eqref{E:Ave_Pair}.
  Also,
  \begin{align*}
    |K_1| &\leq \|u\|_{L^\infty(\Omega_\varepsilon)}\Bigl\|\,|u|^2-\bigl|\overline{\mathcal{M}_\varepsilon u}\bigr|^2\,\Bigr\|_{L^2(\Omega_\varepsilon)}\|\bar{\zeta}\|_{L^2(\Omega_\varepsilon)} \\
    &\leq c\varepsilon^{3/2}\|u\|_{L^\infty(\Omega_\varepsilon)}^2\Bigl(\|u\|_{L^2(\Omega_\varepsilon)}+\|\partial_\nu u\|_{L^2(\Omega_\varepsilon)}\Bigr)\|\zeta\|_{L^2(\Gamma)}
  \end{align*}
  by \eqref{E:CoEx_Lp} and \eqref{E:ASq_Dif}.
  Hence, \eqref{E:Ave_Non} follows.
\end{proof}

\subsection{Tangential gradient of the weighted average} \label{SS:Ave_TGr}
Let us give an explicit formula for the tangential gradient of the weighted average.

\begin{lemma} \label{L:Ave_TGr}
  Let $\varphi\in C(\overline{\Omega}_\varepsilon)\cap C^1(\Omega_\varepsilon)$.
  Then,
  \begin{align} \label{E:Ave_TGr}
    \nabla_\Gamma\mathcal{M}_\varepsilon\varphi = \mathcal{M}_\varepsilon(B\nabla\varphi)+\mathcal{M}_\varepsilon\bigl((\partial_\nu\varphi+\varphi f_J)\Psi_\varepsilon\bigr)+\mathcal{M}_\varepsilon(\varphi\Psi_J) \quad\text{on}\quad \Gamma.
  \end{align}
  Here, $B\colon\mathcal{N}_\delta\to\mathbb{R}^{n\times n}$, $\Psi_\varepsilon\colon\mathcal{N}_\delta\to\mathbb{R}^n$, $f_J\colon\mathcal{N}_\delta\to\mathbb{R}$, $\Psi_J\colon\mathcal{N}_\delta\to\mathbb{R}^n$ are given by
  \begin{align*}
    B^\sharp(y,r) &= P(y)-rW(y), \\
    \Psi_\varepsilon^\sharp(y,r) &= \frac{1}{g(y)}\{(r-\varepsilon g_0(y))\nabla_\Gamma g_1(y)+(\varepsilon g_1(y)-r)\nabla_\Gamma g_0(y)\}, \\
    f_J^\sharp(y,r) &= \frac{\partial_rJ(y,r)}{J(y,r)}, \quad \Psi_J^\sharp(y,r) = \frac{\nabla_\Gamma J(y,r)}{J(y,r)}
  \end{align*}
  for $x=y+r\nu(y)\in\mathcal{N}_\delta$ with $y\in\Gamma$ and $r\in(-\delta,\delta)$ under the notation \eqref{E:Def_Pull}.
\end{lemma}

\begin{proof}
  By the definition of $\mathcal{M}_\varepsilon\varphi$, we have $\nabla_\Gamma\mathcal{M}_\varepsilon\varphi=\sum_{i=1}^4K_i$, where
  \begin{align*}
    K_1 &= -\frac{\nabla_\Gamma g(y)}{\varepsilon g(y)^2}\int_{\varepsilon g_0(y)}^{\varepsilon g_1(y)}\varphi^\sharp(y,r)J(y,r)\,dr, \\
    K_2 &= \frac{1}{\varepsilon g(y)}\{[\varphi^\sharp J](y,\varepsilon g_1(y))\cdot\varepsilon\nabla_\Gamma g_1(y)-[\varphi^\sharp J](y,\varepsilon g_0(y))\cdot\varepsilon\nabla_\Gamma g_0(y)\}, \\
    K_3 &= \frac{1}{\varepsilon g(y)}\int_{\varepsilon g_0(y)}^{\varepsilon g_1(y)}\nabla_\Gamma\varphi^\sharp(y,r)\cdot J(y,r)\,dr, \\
    K_4 &= \frac{1}{\varepsilon g(y)}\int_{\varepsilon g_0(y)}^{\varepsilon g_1(y)}\varphi^\sharp(y,r)\nabla_\Gamma J(y,r)\,dr.
  \end{align*}
  Here and in what follows, we write $[\varphi^\sharp J](y,r)=\varphi^\sharp(y,r)J(y,r)$.
  We see that
  \begin{align} \label{Pf_AT:K34}
    \begin{aligned}
      K_3 &= \frac{1}{\varepsilon g(y)}\int_{\varepsilon g_0(y)}\{P(y)-rW(y)\}(\nabla\varphi)^\sharp(y,r)\cdot J(y,r)\,dr = \mathcal{M}_\varepsilon(B\nabla\varphi), \\
      K_4 &= \frac{1}{\varepsilon g(y)}\int_{\varepsilon g_0(y)}^{\varepsilon g_1(y)}\varphi^\sharp(y,r)\frac{\nabla_\Gamma J(y,r)}{J(y,r)}\cdot J(y,r)\,dr = \mathcal{M}_\varepsilon(\varphi\Psi_J)
    \end{aligned}
  \end{align}
  by \eqref{E:TG_Pull}.
  Since $\Psi_\varepsilon^\sharp(y,\varepsilon g_i(y))=\varepsilon\nabla_\Gamma g_i(y)$ for $i=0,1$,
  \begin{align*}
    K_2 = \frac{1}{\varepsilon g(y)}\Bigl[[\varphi^\sharp J\Psi_\varepsilon^\sharp](y,r)\Bigr]_{r=\varepsilon g_0(y)}^{\varepsilon g_1(y)} = \frac{1}{\varepsilon g(y)}\int_{\varepsilon g_0(y)}^{\varepsilon g_1(y)}\frac{\partial}{\partial r}\Bigl([\varphi^\sharp J\Psi_\varepsilon^\sharp](y,r)\Bigr)\,dr.
  \end{align*}
  Moreover, since $\partial_r\varphi^\sharp=(\partial_\nu\varphi)^\sharp$ under the notation \eqref{E:Def_Pull} and
  \begin{align*}
    \partial_rJ(y,r) = f_J^\sharp(y,r)J(y,r), \quad \partial_r\Psi_\varepsilon^\sharp(y,r) = \frac{\nabla_\Gamma g(y)}{g(y)},
  \end{align*}
  it follows that
  \begin{align*}
    K_2 &= \frac{1}{\varepsilon g(y)}\int_{\varepsilon g_0(y)}^{\varepsilon g_1(y)}[\{(\partial_\nu\varphi)^\sharp+\varphi^\sharp f_J^\sharp\}\Psi_\varepsilon^\sharp](y,r)J(y,r)\,dr \\
    &\qquad +\frac{\nabla_\Gamma g(y)}{\varepsilon g(y)^2}\int_{\varepsilon g_0(y)}^{\varepsilon g_1(y)}\varphi^\sharp(y,r)J(y,r)\,dr \\
    &= \mathcal{M}_\varepsilon\bigl((\partial_\nu\varphi+\varphi f_J)\Psi_\varepsilon\bigr)-K_1.
  \end{align*}
  This equality and \eqref{Pf_AT:K34} show that $\nabla_\Gamma\mathcal{M}_\varepsilon\varphi=\sum_{i=1}^4K_i$ is of the form \eqref{E:Ave_TGr}.
\end{proof}

By the above result, we see that $\nabla_\Gamma\mathcal{M}_\varepsilon\varphi$ is close to $\mathcal{M}_\varepsilon\bigl(\overline{P}\nabla\varphi\bigr)$.

\begin{lemma} \label{L:Ave_H1}
  Let $\varphi\in H^1(\Omega_\varepsilon)$.
  Then, $\mathcal{M}_\varepsilon\varphi\in H^1(\Gamma)$ and
  \begin{align}
    \|\nabla_\Gamma\mathcal{M}_\varepsilon\varphi\|_{L^2(\Gamma)} &\leq c\varepsilon^{-1/2}\|\varphi\|_{H^1(\Omega_\varepsilon)}, \label{E:Ave_H1} \\
    \bigl\|\nabla_\Gamma\mathcal{M}_\varepsilon\varphi-\mathcal{M}_\varepsilon\bigl(\overline{P}\nabla\varphi\bigr)\bigr\|_{L^2(\Gamma)} &\leq c\varepsilon^{1/2}\|\varphi\|_{H^1(\Omega_\varepsilon)}. \label{E:ATGr_L2}
  \end{align}
\end{lemma}

\begin{proof}
  For $y\in\Gamma$ and $r\in(\varepsilon g_0(y),\varepsilon g_1(y))$, we have
  \begin{align*}
    |B^\sharp(y,r)-P(y)| \leq c\varepsilon, \quad |\Psi_\varepsilon^\sharp(y,r)|+|\Psi_J^\sharp(y,r)| \leq c\varepsilon, \quad |f_J^\sharp(y,r)| \leq c
  \end{align*}
  by \eqref{E:G_Bdd}, \eqref{E:J_Bdd}, \eqref{E:J_TGr}, $W\in C(\Gamma)^{n\times n}$, and $g_0,g_1\in C^1(\Gamma)$.
  Also,
  \begin{align*}
    |P|^2 = \mathrm{tr}[P^TP] = \mathrm{tr}[P] = n-1 \quad\text{on}\quad \Gamma, \quad |\partial_\nu\varphi|=|\bar{\nu}\cdot\nabla\varphi|\leq|\nabla\varphi| \quad\text{in}\quad \Omega_\varepsilon.
  \end{align*}
  By these inequalities and \eqref{E:Ave_TGr}, we get $|\nabla_\Gamma\mathcal{M}_\varepsilon\varphi|\leq c\mathcal{M}_\varepsilon(|\varphi|+|\nabla\varphi|)$ and
  \begin{align*}
    \bigl|\nabla_\Gamma\mathcal{M}_\varepsilon\varphi-\mathcal{M}_\varepsilon\bigl(\overline{P}\nabla\varphi\bigr)\bigr| \leq c\varepsilon\mathcal{M}_\varepsilon(|\varphi|+|\nabla\varphi|) \quad\text{on}\quad \Gamma.
  \end{align*}
  These inequalities and \eqref{E:Ave_Lp} show that \eqref{E:Ave_H1} and \eqref{E:ATGr_L2} are valid.
\end{proof}

Moreover, we can show that the Dirichlet form on $\Omega_\varepsilon$ is approximated by the weighted one on $\Gamma$ which involves the tangential gradient of the weighted average.

\begin{lemma} \label{L:AT_Diri}
  Let $\varphi\in H^1(\Omega_\varepsilon)$ and $\eta\in H^1(\Gamma)$.
  Then,
  \begin{align} \label{E:AT_Diri}
    \left|\int_{\Omega_\varepsilon}\nabla\varphi\cdot\nabla\bar{\eta}\,dx-\varepsilon\int_\Gamma g\nabla_\Gamma\mathcal{M}_\varepsilon\varphi\cdot\nabla_\Gamma\eta\,d\mathcal{H}^{n-1}\right| \leq c\varepsilon^{3/2}\|\varphi\|_{H^1(\Omega_\varepsilon)}\|\nabla_\Gamma\eta\|_{L^2(\Gamma)}.
  \end{align}
\end{lemma}

\begin{proof}
  This result was shown in \cite[Lemma 5.6]{Miu17} by calculations under a local coordinate of $\Gamma$ and an associated one of $\Omega_\varepsilon$.
  Here, we give another simple proof.

  We split the difference as
  \begin{align*}
    &\int_{\Omega_\varepsilon}\nabla\varphi\cdot\nabla\bar{\eta}\,dx-\varepsilon\int_\Gamma g\nabla_\Gamma\mathcal{M}_\varepsilon\varphi\cdot\nabla_\Gamma\eta\,d\mathcal{H}^{n-1} = K_1+K_2, \\
    &K_1 = \int_{\Omega_\varepsilon}\nabla\varphi\cdot\nabla\bar{\eta}\,dx-\int_{\Omega_\varepsilon}\nabla\varphi\cdot\overline{\nabla_\Gamma\eta}\,dx, \\
    &K_2 = \int_{\Omega_\varepsilon}\nabla\varphi\cdot\overline{\nabla_\Gamma\eta}\,dx-\varepsilon\int_\Gamma g\nabla_\Gamma\mathcal{M}_\varepsilon\varphi\cdot\nabla_\Gamma\eta\,d\mathcal{H}^{n-1}.
  \end{align*}
  It follows from H\"{o}lder's inequality and \eqref{E:CoDe_L2} that
  \begin{align*}
    |K_1| \leq \|\nabla\varphi\|_{L^2(\Omega_\varepsilon)}\bigl\|\nabla\bar{\eta}-\overline{\nabla_\Gamma\eta}\bigr\|_{L^2(\Omega_\varepsilon)} \leq c\varepsilon^{3/2}\|\varphi\|_{H^1(\Omega_\varepsilon)}\|\nabla_\Gamma\eta\|_{L^2(\Gamma)}.
  \end{align*}
  Next, we see by $\nu\cdot\nabla_\Gamma\eta=0$ on $\Gamma$ and \eqref{E:Ave_Pair} that
  \begin{align*}
    \int_{\Omega_\varepsilon}\nabla\varphi\cdot\overline{\nabla_\Gamma\eta}\,dx = \int_{\Omega_\varepsilon}\bigl(\overline{P}\nabla\varphi)\cdot\overline{\nabla_\Gamma\eta}\,dx = \varepsilon\int_\Gamma g\mathcal{M}_{\varepsilon}\bigl(\overline{P}\nabla\varphi\bigr)\cdot\nabla_\Gamma\eta\,d\mathcal{H}^{n-1}.
  \end{align*}
  Hence, it follows from \eqref{E:G_Bdd}, H\"{o}lder's inequality, and \eqref{E:ATGr_L2} that
  \begin{align*}
    |K_2| \leq c\varepsilon\bigl\|\nabla_\Gamma\mathcal{M}_\varepsilon\varphi-\mathcal{M}_\varepsilon\bigl(\overline{P}\nabla\varphi\bigr)\bigr\|_{L^2(\Gamma)}\|\nabla_\Gamma\eta\|_{L^2(\Gamma)} \leq c\varepsilon^{3/2}\|\varphi\|_{H^1(\Omega_\varepsilon)}\|\nabla_\Gamma\eta\|_{L^2(\Gamma)}.
  \end{align*}
  By the above results, we obtain \eqref{E:AT_Diri}.
\end{proof}

\subsection{Time derivative of the weighted average} \label{SS:Ave_Dt}
We can consider the time derivative of the weighted average in the following weak sense.
Recall that the function spaces $Z_T(S)$ and $E_T(S)$ are given by \eqref{E:Def_ZT} and \eqref{E:Def_ET}, respectively, for $S=\Omega_\varepsilon$ or $S=\Gamma$.

\begin{lemma} \label{L:Ave_Dt}
  For a fixed $T>0$, let $u\in E_T(\Omega_\varepsilon)$.
  Then, $\mathcal{M}_\varepsilon u\in E_T(\Gamma)$ and
  \begin{align} \label{E:Ave_Dt}
    \int_0^T\langle\partial_tu(t),\bar{\zeta}(t)\rangle_{(H^1\cap L^4)(\Omega_\varepsilon)}\,dt = \varepsilon\int_0^T\langle\partial_t\mathcal{M}_\varepsilon u(t),g\zeta(t)\rangle_{(H^1\cap L^4)(\Gamma)}\,dt
  \end{align}
  for all $\zeta\in Z_T(\Gamma)$.
\end{lemma}

\begin{proof}
  We have $\mathcal{M}_\varepsilon u\in Z_T(\Gamma)$ by Lemmas \ref{L:Ave_Lp} and \ref{L:Ave_H1}.

  Next, let $\zeta\in C_c^1(0,T;(H^1\cap L^4)(\Gamma)^N)$.
  Then,
  \begin{align*}
    \bar{\zeta} \in C_c^1(0,T;(H^1\cap L^4)(\Omega_\varepsilon)^N), \quad \|\bar{\zeta}\|_{Z_T(\Omega_\varepsilon)} \leq c\varepsilon^{1/4}\|\zeta\|_{Z_T(\Gamma)}
  \end{align*}
  by \eqref{E:CoEx_Lp}, \eqref{E:CoDe_L2}, and $\varepsilon^{1/2}\leq\varepsilon^{1/4}$ due to $\varepsilon\in(0,1)$.
  Moreover,
  \begin{align} \label{Pf_ADt:Pair}
    \begin{aligned}
      \int_0^T\langle\partial_tu(t),\bar{\zeta}(t)\rangle_{(H^1\cap L^4)(\Omega_\varepsilon)}\,dt &= -\int_0^T\bigl(u(t),\partial_t\bar{\zeta}(t)\bigr)_{L^2(\Omega_\varepsilon)}\,dt \\
      &= -\varepsilon\int_0^T\bigl(g\mathcal{M}_\varepsilon u(t),\partial_t\zeta(t)\bigr)_{L^2(\Gamma)} \\
      &= -\varepsilon\int_0^T\bigl(\mathcal{M}_\varepsilon u(t),[\partial_t(g\zeta)](t)\bigr)_{L^2(\Gamma)}\,dt,
    \end{aligned}
  \end{align}
  since \eqref{E:Ave_Pair} holds and $g$ is independent of time.
  Replacing $\zeta$ with $g^{-1}\zeta$ and using
  \begin{align*}
    \|\bar{g}^{-1}\bar{\zeta}\|_{Z_T(\Omega_\varepsilon)} \leq c\varepsilon^{1/4}\|g^{-1}\zeta\|_{Z_T(\Gamma)} \leq c\varepsilon^{1/4}\|\zeta\|_{Z_T(\Gamma)}
  \end{align*}
  by $g\in C^1(\Gamma)$ and \eqref{E:G_Bdd}, we find that
  \begin{align*}
    \left|\int_0^T\bigl(\mathcal{M}_\varepsilon u(t),\partial_t\zeta(t)\bigr)_{L^2(\Gamma)}\,dt\right| &= \varepsilon^{-1}\left|\int_0^T\langle\partial_tu(t),\bar{g}^{-1}\bar{\zeta}(t)\rangle_{(H^1\cap L^4)(\Omega_\varepsilon)}\,dt\right| \\
    &\leq c\varepsilon^{-3/4}\|\partial_tu\|_{[Z_T(\Omega_\varepsilon)]'}\|\zeta\|_{Z_T(\Gamma)}
  \end{align*}
  for all $\zeta\in C_c^1(0,T;(H^1\cap L^4)(\Gamma)^N)$.
  Thus, $\partial_t\mathcal{M}_\varepsilon u\in [Z_T(\Gamma)]'$ and $\mathcal{M}_\varepsilon u\in E_T(\Gamma)$.

  The equality \eqref{E:Ave_Dt} follows from \eqref{Pf_ADt:Pair} when $\zeta\in C_c^1(0,T;(H^1\cap L^4)(\Gamma)^N)$.
  This space is dense in $Z_T(\Gamma)$, so the same equality holds for $\zeta\in Z_T(\Gamma)$ by a density argument.
\end{proof}

\section{Thin-film limit problem} \label{S:TFL}
The purpose of this section is to establish Theorems \ref{T:Weak} and \ref{T:Df_Sur}.
Let $\mathcal{M}_\varepsilon$ be the weighted average operator given in Section \ref{S:Ave}.
We write $c$ for a general positive constant independent of $\varepsilon$ and also of the constant $\lambda$ appearing in \eqref{E:GL_CTD}.
Also, for a function $\eta$ on $\Gamma$, we denote by $\bar{\eta}=\eta\circ\pi$ its constant extension in the normal direction of $\Gamma$.

\subsection{Average of the weak form of the thin domain problem} \label{SS:TFL_Ave}
First we give explicit estimates in terms of $\varepsilon$ for a weak solution $u^\varepsilon$ to \eqref{E:GL_CTD} and then derive a weak form satisfied by the weighted average $\mathcal{M}_\varepsilon u^\varepsilon$.

\begin{lemma} \label{L:Ue_Est}
  Let $c_1\geq1$, $\alpha\in(0,1/3]$, and $\varepsilon_0\in(0,1)$ be constants and let $\varepsilon\in(0,\varepsilon_0)$.
  Suppose that $u_0^\varepsilon\in L^\infty(\Omega_\varepsilon)^N$ and
  \begin{align} \label{E:U0_Sup}
    \|u_0^\varepsilon\|_{L^\infty(\Omega_\varepsilon)} \leq c_1\varepsilon^{-1/3+\alpha},
  \end{align}
  and let $u^\varepsilon$ be the global weak solution to \eqref{E:GL_CTD} given in Theorem \ref{T:GCT_Exi}.
  Then,
  \begin{align} \label{E:Ue_L2}
    \|u^\varepsilon(t)\|_{L^2(\Omega_\varepsilon)}^2+\int_0^t\|\nabla u^\varepsilon(s)\|_{L^2(\Omega_\varepsilon)}^2\,ds+\lambda\int_0^t\|u^\varepsilon(s)\|_{L^4(\Omega_\varepsilon)}^4\,ds \leq c\varepsilon^{1/3+2\alpha}e^{2\lambda t}
  \end{align}
  for all $t\geq0$, where $c>0$ depends on $c_1$ but is independent of $\varepsilon$ and $\lambda$.
  Moreover,
  \begin{align} \label{E:Ue_Sup}
    \|u^\varepsilon\|_{L^\infty(\Omega_\varepsilon\times(0,\infty))} \leq c_1\varepsilon^{-1/3+\alpha}.
  \end{align}
\end{lemma}

\begin{proof}
  Let $|\Omega_\varepsilon|$ be the volume of $\Omega_\varepsilon$.
  We have \eqref{E:Ue_L2} by \eqref{E:GCT_Ena} and
  \begin{align} \label{Pf_Ue:U0_L2}
    \|u_0^\varepsilon\|_{L^2(\Omega_\varepsilon)}^2 \leq \|u_0^\varepsilon\|_{L^\infty(\Omega_\varepsilon)}^2\cdot|\Omega_\varepsilon| \leq c_1^2\varepsilon^{-2/3+2\alpha}\cdot c\varepsilon = c_1^2c\varepsilon^{1/3+2\alpha}.
  \end{align}
  Also, \eqref{E:Ue_Sup} follows from \eqref{E:GCT_Linf}, \eqref{E:U0_Sup}, and $c_1\varepsilon^{-1/3+\alpha}\geq1$.
\end{proof}

\begin{lemma} \label{L:Av_Weak}
  Under the assumptions of Lemma \ref{L:Ue_Est}, we define
  \begin{align*}
    v^\varepsilon = \mathcal{M}_\varepsilon u^\varepsilon \quad\text{on}\quad \Gamma\times[0,\infty), \quad v_0^\varepsilon = \mathcal{M}_\varepsilon u_0^\varepsilon \quad\text{on}\quad \Gamma.
  \end{align*}
  Then, $v^\varepsilon\in E_T(\Gamma)$ for all $T>0$ and $v^\varepsilon(0)=v_0^\varepsilon$ in $L^2(\Gamma)^N$.
  Moreover,
  \begin{multline} \label{E:Av_Weak}
    \int_0^T\langle\partial_tv^\varepsilon(t),g\zeta(t)\rangle_{(H^1\cap L^4)(\Gamma)}\,dt+\int_0^T\bigl(g\nabla_\Gamma v^\varepsilon(t),\nabla_\Gamma\zeta(t)\bigr)_{L^2(\Gamma)}\,dt \\
    +\lambda\int_0^T\bigl(g(|v^\varepsilon(t)|^2-1)v^\varepsilon(t),\zeta(t)\bigr)_{L^2(\Gamma)}\,dt = R_\varepsilon(\zeta;T)
  \end{multline}
  for all $T>0$ and $\zeta\in Z_T(\Gamma)$.
  Here, $R_\varepsilon(\zeta;T)$ is linear in $\zeta$ and satisfies
  \begin{align} \label{E:AvW_Res}
    |R_\varepsilon(\zeta;T)| \leq c\varepsilon^{3\alpha}a_\lambda(T)\|\zeta\|_{L^2(0,T;H^1(\Gamma))},
  \end{align}
  where $a_\lambda(T)=(1+\lambda)(1+T)^{1/2}e^{\lambda T}$.
  Also, $c>0$ is independent of $\varepsilon$, $\lambda$, and $T$.
\end{lemma}

Note that, if $v^\varepsilon\in E_T(\Gamma)$, then $v^\varepsilon\in C([0,T];L^2(\Gamma)^N)$ by Lemma \ref{L:ET_Con}.

\begin{proof}
  The regularity and initial condition of $v^\varepsilon$ follow from those of $u^\varepsilon$ and Lemma \ref{L:Ave_Dt}.

  Let $\zeta\in Z_T(\Gamma)$.
  We take $\psi=\bar{\zeta}$ in \eqref{E:WF_GCT} and divide both sides by $\varepsilon$ to get
  \begin{multline*}
    \varepsilon^{-1}\int_0^T\langle\partial_tu^\varepsilon(t),\bar{\zeta}(t)\rangle_{(H^1\cap L^4)(\Omega_\varepsilon)}\,dt+\varepsilon^{-1}\int_0^T\bigl(\nabla u^\varepsilon(t),\nabla\bar{\zeta}(t)\bigr)_{L^2(\Omega_\varepsilon)}\,dt \\
    +\varepsilon^{-1}\lambda\int_0^T\bigl((|u^\varepsilon(t)|^2-1)u^\varepsilon(t),\bar{\zeta}(t)\bigr)_{L^2(\Omega_\varepsilon)}\,dt = 0.
  \end{multline*}
  In this equality, we use \eqref{E:Ave_Dt} with $v^\varepsilon=\mathcal{M}_\varepsilon u^\varepsilon$ and
  \begin{align*}
    \varepsilon^{-1}\lambda\bigl(u^\varepsilon(t),\bar{\zeta}(t)\bigr)_{L^2(\Omega_\varepsilon)} = \lambda\bigl(gv^\varepsilon(t),\zeta(t)\bigr)_{L^2(\Gamma)}
  \end{align*}
  by \eqref{E:Ave_Pair}.
  Then, we get \eqref{E:Av_Weak}, if we define the residual term as
  \begin{multline*}
    R_\varepsilon(\zeta;T) = -\int_0^T\left\{\varepsilon^{-1}\bigl(\nabla u^\varepsilon(t),\nabla\bar{\zeta}(t)\bigr)_{L^2(\Omega_\varepsilon)}-\bigl(g\nabla_\Gamma v^\varepsilon(t),\nabla_\Gamma\zeta(t)\bigr)_{L^2(\Gamma)}\right\}\,dt \\
    -\lambda\int_0^T\left\{\varepsilon^{-1}\bigl(|u^\varepsilon(t)|^2u^\varepsilon(t),\bar{\zeta}(t)\bigr)_{L^2(\Omega_\varepsilon)}-\bigl(g|v^\varepsilon(t)|^2v^\varepsilon(t),\zeta(t)\bigr)_{L^2(\Gamma)}\right\}\,dt.
  \end{multline*}
  This is linear in $\zeta$.
  Moreover, by \eqref{E:Ave_Non}, \eqref{E:AT_Diri}, and H\"{o}lder's inequality,
  \begin{multline*}
    |R_\varepsilon(\zeta;T)| \leq c\varepsilon^{1/2}\|u^\varepsilon\|_{L^2(0,T;H^1(\Omega_\varepsilon))}\|\nabla_\Gamma\zeta\|_{L^2(0,T;L^2(\Gamma))} \\
    +c\varepsilon^{1/2}\lambda\|u^\varepsilon\|_{L^\infty(\Omega_\varepsilon\times(0,T))}^2\|u^\varepsilon\|_{L^2(0,T;H^1(\Omega_\varepsilon))}\|\zeta\|_{L^2(0,T;L^2(\Gamma))}.
  \end{multline*}
  To the right-hand side, we further apply \eqref{E:Ue_Sup} and
  \begin{align*}
    \|u^\varepsilon\|_{L^2(0,T;H^1(\Omega_\varepsilon))} \leq c\varepsilon^{1/6+\alpha}(1+T)^{1/2}e^{\lambda T}
  \end{align*}
  by \eqref{E:Ue_L2}.
  Then, we get \eqref{E:AvW_Res}, since $\varepsilon^{1/2+1/6+\alpha}=\varepsilon^{2/3+\alpha}\leq\varepsilon^{3\alpha}$ by $\alpha\in(0,1/3]$.
\end{proof}

\subsection{Estimates for the averaged weak solution} \label{SS:TFL_Est}
Using \eqref{E:Av_Weak}, we derive estimates for the averaged weak solution $v^\varepsilon=\mathcal{M}_\varepsilon u^\varepsilon$ which give the uniform boundedness of $v^\varepsilon$.

\begin{lemma} \label{L:Av_Ena}
  Under the assumptions of Lemma \ref{L:Ue_Est}, we have
  \begin{multline} \label{E:Av_Ena}
    \max_{t\in[0,T]}\|v^\varepsilon(t)\|_{L^2(\Gamma)}^2+\int_0^T\|\nabla_\Gamma v^\varepsilon(t)\|_{L^2(\Gamma)}^2\,dt+\lambda\int_0^T\|v^\varepsilon(t)\|_{L^4(\Gamma)}^4\,dt \\
    \leq ce^{c(1+\lambda)T}\Bigl\{\|v_0^\varepsilon\|_{L^2(\Gamma)}^2+\varepsilon^{6\alpha}(1+\lambda)^2\Bigr\}
  \end{multline}
  for all $T>0$.
  Here, $c>0$ is a constant independent of $\varepsilon$, $\lambda$, and $T$.
\end{lemma}

\begin{proof}
  Let $\zeta=v^\varepsilon$ in \eqref{E:Av_Weak} with $T$ replaced by $t\in[0,T]$.
  Then, noting that $g$ is independent of time and positive by \eqref{E:G_Bdd}, we observe by \eqref{E:ET_IbP} that
  \begin{multline*}
    \frac{1}{2}\|g^{1/2}v^\varepsilon(t)\|_{L^2(\Gamma)}^2+\int_0^t\|g^{1/2}\nabla_\Gamma v^\varepsilon(s)\|_{L^2(\Gamma)}^2\,ds+\lambda\int_0^t\|g^{1/4}v^\varepsilon(s)\|_{L^4(\Gamma)}^4\,ds \\
    = \frac{1}{2}\|g^{1/2}v_0^\varepsilon\|_{L^2(\Gamma)}^2+\lambda\int_0^t\|g^{1/2}v^\varepsilon(s)\|_{L^2(\Gamma)}^2\,ds+R_\varepsilon(v^\varepsilon;t).
  \end{multline*}
  We apply \eqref{E:G_Bdd} and \eqref{E:AvW_Res} to this equality to get
  \begin{multline} \label{Pf_AE:c1c2}
    \|v^\varepsilon(t)\|_{L^2(\Gamma)}^2+\int_0^t\|\nabla_\Gamma v^\varepsilon(s)\|_{L^2(\Gamma)}^2\,ds+\lambda\int_0^t\|v^\varepsilon(s)\|_{L^4(\Gamma)}^4\,ds \\
    \leq c_2\left(\|v_0^\varepsilon\|_{L^2(\Gamma)}^2+\lambda\int_0^t\|v^\varepsilon(s)\|_{L^2(\Gamma)}^2\,ds\right)+c_3\varepsilon^{3\alpha}a_\lambda(t)\|v^\varepsilon\|_{L^2(0,t;H^1(\Gamma))}
  \end{multline}
  with some constants $c_2,c_3>0$ independent of $\varepsilon$, $\lambda$, and $T$.
  Moreover,
  \begin{align*}
    c_3\varepsilon^{3\alpha}a_\lambda(t)\|v^\varepsilon\|_{L^2(0,t;H^1(\Gamma))} &\leq c\varepsilon^{6\alpha}a_\lambda(t)^2+\frac{1}{2}\|v^\varepsilon\|_{L^2(0,t;H^1(\Gamma))}^2 \\
    &= c\varepsilon^{6\alpha}a_\lambda(t)^2+\frac{1}{2}\int_0^t\Bigl(\|v^\varepsilon(s)\|_{L^2(\Gamma)}^2+\|\nabla_\Gamma v^\varepsilon(s)\|_{L^2(\Gamma)}^2\Bigr)\,ds
  \end{align*}
  by Young's inequality, and $a_\lambda(t)\leq a_\lambda(T)$ for $t\in[0,T]$ by the definition of $a_\lambda(t)$.
  We apply these estimates to \eqref{Pf_AE:c1c2} and make the integral of $\nabla_\Gamma v^\varepsilon$ appearing in the right-hand side absorbed into the left-hand side.
  Then, we find that
  \begin{multline} \label{Pf_AE:Ineq}
    \|v^\varepsilon(t)\|_{L^2(\Gamma)}^2+\int_0^t\|\nabla_\Gamma v^\varepsilon(s)\|_{L^2(\Gamma)}^2\,ds+\lambda\int_0^t\|v^\varepsilon(s)\|_{L^4(\Gamma)}^4\,ds \\
    \leq c\left(\|v_0^\varepsilon\|_{L^2(\Gamma)}^2+\varepsilon^{6\alpha}a_\lambda(T)^2+(1+\lambda)\int_0^t\|v^\varepsilon(s)\|_{L^2(\Gamma)}^2\,ds\right)
  \end{multline}
  for all $t\in[0,T]$ and thus, by $\lambda>0$,
  \begin{align*}
    \|v^\varepsilon(t)\|_{L^2(\Gamma)}^2 \leq c\left(\|v_0^\varepsilon\|_{L^2(\Gamma)}^2+\varepsilon^{6\alpha}a_\lambda(T)^2+(1+\lambda)\int_0^t\|v^\varepsilon(s)\|_{L^2(\Gamma)}^2\,ds\right)
  \end{align*}
  for all $t\in[0,T]$.
  Hence, Gronwall's inequality implies that
  \begin{align*}
    \|v^\varepsilon(t)\|_{L^2(\Gamma)}^2 \leq ce^{c(1+\lambda)t}\Bigl(\|v_0^\varepsilon\|_{L^2(\Gamma)}^2+\varepsilon^{6\alpha}a_\lambda(T)^2\Bigr), \quad t\in[0,T].
  \end{align*}
  Applying this to \eqref{Pf_AE:Ineq} and noting that
  \begin{align*}
    a_\lambda(T)^2 = (1+\lambda)^2(1+T)e^{2\lambda T} \leq (1+\lambda)^2\cdot e^T \cdot e^{2\lambda T} \leq (1+\lambda)^2e^{c(1+\lambda)T},
  \end{align*}
  we obtain \eqref{E:Av_Ena}.
\end{proof}

\begin{lemma} \label{L:AvEn_Dt}
  Under the assumptions of Lemma \ref{L:Ue_Est}, we have
  \begin{align} \label{E:AvEn_Dt}
    \|\partial_tv^\varepsilon\|_{[Z_T(\Gamma)]'} \leq c_{\lambda,T}\Bigl(\|v_0^\varepsilon\|_{L^2(\Gamma)}+\|v_0^\varepsilon\|_{L^2(\Gamma)}^{3/2}+1\Bigr)
  \end{align}
  for all $T>0$, where $c_{\lambda,T}>0$ is a constant depending on $\lambda$ and $T$ but independent of $\varepsilon$.
\end{lemma}

\begin{proof}
  Throughout the proof, we write $c_{\lambda,T}$ for a general positive constant depending on $\lambda$ and $T$ but independent of $\varepsilon$.
  Then, we see by \eqref{E:Av_Ena} and $\varepsilon^{6\alpha}\leq1$ that
  \begin{align} \label{Pf_EnDt:Ena}
    \int_0^T\|v^\varepsilon(t)\|_{H^1(\Gamma)}^2\,dt+\int_0^T\|v^\varepsilon(t)\|_{L^4(\Gamma)}^4\,dt \leq c_{\lambda,T}\Bigl(\|v_0^\varepsilon\|_{L^2(\Gamma)}^2+1\Bigr).
  \end{align}
  For $\zeta\in Z_T(\Gamma)$, we substitute $g^{-1}\zeta$ for the test function of \eqref{E:Av_Weak}.
  Then, by
  \begin{align*}
    \nabla_\Gamma(g^{-1}\zeta) = \bigl(\underline{D}_i(g\zeta_j)\bigr)_{i,j} = \bigl(-g^{-2}(\underline{D}_ig)\zeta_j+g^{-1}\underline{D}_i\zeta_j\bigr)_{i,j}
  \end{align*}
  and by $g\in C^1(\Gamma)$, \eqref{E:G_Bdd}, and H\"{o}lder's inequality, we find that
  \begin{align*}
    \left|\int_0^T\langle\partial_tv^\varepsilon(t),\zeta(t)\rangle_{(H^1\cap L^4)(\Gamma)}\,dt\right| \leq cQ(v^\varepsilon,\zeta;T)+|R_\varepsilon(\zeta;T)|,
  \end{align*}
  where
  \begin{multline*}
    Q(v^\varepsilon,\zeta;T) = \|\nabla_\Gamma v^\varepsilon\|_{L^2(0,T;L^2(\Gamma))}\|\zeta\|_{L^2(0,T;H^1(\Gamma))}+\lambda\|v^\varepsilon\|_{L^4(0,T;L^4(\Gamma)}^3\|\zeta\|_{L^4(0,T;L^4(\Gamma))} \\
    +\lambda\|v^\varepsilon\|_{L^2(0,T;L^2(\Gamma))}\|\zeta\|_{L^2(0,T;L^2(\Gamma))}.
  \end{multline*}
  Since \eqref{Pf_EnDt:Ena} holds and the norms of $\zeta$ are bounded by $\|\zeta\|_{Z_T(\Gamma)}$, we have
  \begin{align*}
    Q(v^\varepsilon,\zeta;T) \leq c_{\lambda,T}\Bigl(\|v_0^\varepsilon\|_{L^2(\Gamma)}+\|v_0^\varepsilon\|_{L^2(\Gamma)}^{3/2}+1\Bigr)\|\zeta\|_{Z_T(\Gamma)}.
  \end{align*}
  Moreover, $|R_\varepsilon(\zeta;T)|\leq c_{\lambda,T}\|\zeta\|_{Z_T(\Gamma)}$ by \eqref{E:AvW_Res} and $\varepsilon^{3\alpha}\leq1$.
  Hence,
  \begin{align*}
    \left|\int_0^T\langle\partial_tv^\varepsilon(t),\zeta(t)\rangle_{(H^1\cap L^4)(\Gamma)}\,dt\right| \leq c_{\lambda,T}\Bigl(\|v_0^\varepsilon\|_{L^2(\Gamma)}+\|v_0^\varepsilon\|_{L^2(\Gamma)}^{3/2}+1\Bigr)\|\zeta\|_{Z_T(\Gamma)}
  \end{align*}
  for all $\zeta\in Z_T(\Gamma)$, which implies \eqref{E:AvEn_Dt}.
\end{proof}

\subsection{Weak convergence and characterization of the limit} \label{SS:TFL_WC}
Based on the results in the previous subsections, we prove Theorem \ref{T:Weak}.

\begin{proof}[Proof of Theorem \ref{T:Weak}]
  Suppose that $u_0^\varepsilon\in L^\infty(\Omega_\varepsilon)^N$ and the conditions (a) and (b) in Theorem \ref{T:Weak} are satisfied.
  Then, for each $\varepsilon\in(0,\varepsilon_0)$, there exists a unique global weak solution $u^\varepsilon$ to \eqref{E:GL_CTD} by Theorem \ref{T:GCT_Exi}.
  Moreover, we have \eqref{E:Ue_L2} and \eqref{E:Ue_Sup} by Lemma \ref{L:Ue_Est}, and thus we can use the results given in Sections \ref{SS:TFL_Ave}--\ref{SS:TFL_Est}.
  Let
  \begin{align*}
    v^\varepsilon = \mathcal{M}_\varepsilon u^\varepsilon \quad\text{on}\quad \Gamma\times[0,\infty), \quad v_0^\varepsilon = \mathcal{M}_\varepsilon u_0^\varepsilon \quad\text{on}\quad \Gamma.
  \end{align*}
  Then, $\{v_0^\varepsilon\}_\varepsilon$ is bounded in $L^2(\Gamma)^N$ by the condition (b).
  Hence, for each fixed $T>0$, we see by \eqref{E:Av_Ena} with $\varepsilon^{6\alpha}\leq1$ and by \eqref{E:AvEn_Dt} that $\{v^\varepsilon\}_\varepsilon$ is bounded in the space $E_T(\Gamma)$ given by \eqref{E:Def_ET} (note that $\lambda>0$ is independent of $\varepsilon$).
  By this fact, there exists a sequence $\{\varepsilon_k\}_{k=1}^\infty$ in $(0,\varepsilon_0)$ with $\varepsilon_k\to0$ and some $v_T\in E_T(\Gamma)$ such that $v^{\varepsilon_k}\to v_T$ weakly in $E_T(\Gamma)$, i.e.
  \begin{align} \label{Pf_TW:WeCo}
    \begin{alignedat}{3}
      \lim_{k\to\infty}v^{\varepsilon_k} &= v_T &\quad &\text{weakly in} &\quad &Z_T(\Gamma), \\
      \lim_{k\to\infty}\partial_tv^{\varepsilon_k} &= \partial_tv_T &\quad &\text{weakly in} &\quad &[Z_T(\Gamma)]'.
    \end{alignedat}
  \end{align}
  Moreover, since the embedding $E_T(\Gamma)\hookrightarrow L^2(0,T;L^2(\Gamma)^N)$ is compact by Lemma \ref{L:ET_AuLi}, we may assume, by taking a subsequence again, that
  \begin{align*}
    \lim_{k\to\infty}v^{\varepsilon_k} = v_T \quad\text{strongly in}\quad L^2(0,T;L^2(\Gamma)^N),
  \end{align*}
  and in particular, $v^{\varepsilon_k}\to v_T$ a.e. on $\Gamma\times(0,T)$ and thus
  \begin{align*}
    \lim_{k\to\infty}|v^{\varepsilon_k}|^2v^{\varepsilon_k} = |v_T|^2v_T \quad\text{a.e. on}\quad \Gamma\times(0,T).
  \end{align*}
  We also observe by the boundedness of $\{v^\varepsilon\}_\varepsilon$ in $E_T(\Gamma)$ that the norm
  \begin{align*}
    \|\,|v^{\varepsilon_k}|^2v^{\varepsilon_k}\|_{L^{4/3}(0,T;L^{4/3}(\Gamma))} = \|v^{\varepsilon_k}\|_{L^4(0,T;L^4(\Gamma))}^3
  \end{align*}
  is bounded uniformly with respect to $\varepsilon_k$.
  Hence, we can apply Lemma \ref{L:AC_Non} to get
  \begin{align} \label{Pf_TW:L43}
    \lim_{k\to\infty}|v^{\varepsilon_k}|^2v^{\varepsilon_k} = |v_T|^2v_T \quad\text{weakly in}\quad L^{4/3}(0,T;L^{4/3}(\Gamma)^N).
  \end{align}
  Let us show that $v_T$ satisfies \eqref{E:WF_GLim}.
  For each $\zeta\in Z_T(\Gamma)$, we have
  \begin{multline} \label{Pf_TW:WeEk}
    \int_0^T\langle\partial_tv^{\varepsilon_k}(t),g\zeta(t)\rangle_{(H^1\cap L^4)(\Gamma)}\,dt+\int_0^T\bigl(g\nabla_\Gamma v^{\varepsilon_k}(t),\nabla_\Gamma\zeta(t)\bigr)_{L^2(\Gamma)}\,dt \\
    +\lambda\int_0^T\bigl(g(|v^{\varepsilon_k}(t)|^2-1)v^{\varepsilon_k}(t),\zeta(t)\bigr)_{L^2(\Gamma)}\,dt = R_{\varepsilon_k}(\zeta;T)
  \end{multline}
  by \eqref{E:Av_Weak}.
  Let $\varepsilon_k\to0$ in this equality.
  Then,
  \begin{align*}
    \lim_{k\to\infty}\int_0^T\langle\partial_tv^{\varepsilon_k}(t),g\zeta(t)\rangle_{(H^1\cap L^4)(\Gamma)}\,dt &= \int_0^T\langle\partial_tv_T(t),g\zeta(t)\rangle_{(H^1\cap L^4)(\Gamma)}\,dt, \\
    \lim_{k\to\infty}\int_0^T\bigl(g\nabla_\Gamma v^{\varepsilon_k}(t),\nabla_\Gamma\zeta(t)\bigr)_{L^2(\Gamma)}\,dt &= \int_0^T\bigl(g\nabla_\Gamma v_T(t),\nabla_\Gamma\zeta(t)\bigr)_{L^2(\Gamma)}\,dt, \\
    \lim_{k\to\infty}\lambda\int_0^T\bigl(gv^{\varepsilon_k}(t),\zeta(t)\bigr)_{L^2(\Gamma)}\,dt &= \lambda\int_0^T\bigl(gv_T(t),\zeta(t)\bigr)_{L^2(\Gamma)}\,dt
  \end{align*}
  by \eqref{Pf_TW:WeCo}.
  Moreover, we see by \eqref{Pf_TW:L43} that
  \begin{align*}
    \lim_{k\to\infty}\lambda\int_0^T\bigl(g|v^{\varepsilon_k}(t)|^2v^{\varepsilon_k}(t),\zeta(t)\bigr)_{L^2(\Gamma)}\,dt = \lambda\int_0^T\bigl(g|v_T(t)|^2v_T(t),\zeta(t)\bigr)_{L^2(\Gamma)}\,dt.
  \end{align*}
  In the above, we also used the fact that $g$ satisfies \eqref{E:G_Bdd}.
  We also have
  \begin{align} \label{Pf_TW:Res}
    |R_{\varepsilon_k}(\zeta;T)| \leq c\varepsilon_k^{3\alpha}a_\lambda(T)\|\zeta\|_{L^2(0,T;H^1(\Gamma))} \to 0 \quad\text{as}\quad \varepsilon_k\to0
  \end{align}
  by \eqref{E:AvW_Res}.
  Hence, letting $\varepsilon_k\to0$ in \eqref{Pf_TW:WeEk}, we find that $v_T$ satisfies \eqref{E:WF_GLim}.

  Next, we verify the initial condition.
  For $\zeta_0\in C^1(\Gamma)$, let
  \begin{align*}
    \zeta(t) = (1-t/T)\zeta_0 \in C^1([0,T];(H^1\cap L^4)(\Gamma)^N) \subset E_T(\Gamma).
  \end{align*}
  We take this $\zeta$ in \eqref{Pf_TW:WeEk} and carry out integration by parts with respect to time by using \eqref{E:ET_IbP}.
  Then, by $\zeta(0)=\zeta_0$ and $\zeta(T)=0$, we have
  \begin{align*}
    -(v^{\varepsilon_k}(0),g\zeta_0)_{L^2(\Gamma)}+I(v^{\varepsilon_k}) = R_{\varepsilon_k}(\zeta;T),
  \end{align*}
  where $I(w)$ for $w\in E_T(\Gamma)$ is given by
  \begin{multline*}
    I(w) = -\int_0^T\bigl(gw(t),\partial_t\zeta(t)\bigr)_{L^2(\Gamma)}\,dt+\int_0^T\bigl(g\nabla_\Gamma w(t),\nabla_\Gamma\zeta(t)\bigr)_{L^2(\Gamma)}\,dt \\
    +\lambda\int_0^T\bigl(g(|w(t)|^2-1)w(t),\zeta(t)\bigr)_{L^2(\Gamma)}dt.
  \end{multline*}
  We send $\varepsilon_k\to0$, use $v^{\varepsilon_k}(0)=v_0^{\varepsilon_k}\to v_0$ weakly in $L^2(\Gamma)^N$ by the condition (b) of Theorem \ref{T:Weak}, and apply \eqref{Pf_TW:WeCo}, \eqref{Pf_TW:L43}, and \eqref{Pf_TW:Res}.
  Then, we find that
  \begin{align*}
    -(v_0,g\zeta_0)_{L^2(\Gamma)}+I(v_T) = 0.
  \end{align*}
  On the other hand, since $v_T$ satisfies \eqref{E:WF_GLim}, we take the above $\zeta$ and use \eqref{E:ET_IbP} to get
  \begin{align*}
    -(v_T(0),g\zeta_0)_{L^2(\Gamma)}+I(v_T) = 0.
  \end{align*}
  Comparing the above two equalities, we obtain
  \begin{align*}
    (v_0,g\zeta_0)_{L^2(\Gamma)} = (v_T(0),g\zeta_0)_{L^2(\Gamma)} \quad\text{for all}\quad \zeta_0\in C^1(\Gamma).
  \end{align*}
  This gives $v_T(0)=v_0$ in $L^2(\Gamma)^N$, since $C^1(\Gamma)$ is dense in $L^2(\Gamma)$ and $g$ satisfies \eqref{E:G_Bdd}.
  By the above results, we conclude that $v_T$ is a unique weak solution to \eqref{E:GL_Lim} on $[0,T)$ with initial data $v_0$.
  Here, the uniqueness follows from Lemma \ref{L:GLim_Uni}.

  By the above arguments, we can also prove the following statement: for any sequence $\{\varepsilon_\ell\}_{\ell=1}^\infty$ in $(0,\varepsilon_0)$ with $\varepsilon_\ell\to0$, the sequence $\{v^{\varepsilon_\ell}\}_{\ell=1}^\infty$ has a subsequence that converges to the same $v_T$ weakly in $E_T(\Gamma)$.
  This shows that the full sequence $v^\varepsilon$ converges to $v_T$ weakly in $E_T(\Gamma)$ as $\varepsilon\to0$.
  In particular, we have
  \begin{align} \label{Pf_TW:Full}
    \lim_{\varepsilon\to0}v^\varepsilon = \lim_{\varepsilon\to0}\mathcal{M}_\varepsilon u^\varepsilon = v_T \quad\text{weakly in}\quad Z_T(\Gamma)
  \end{align}
  as stated in Theorem \ref{T:Weak} (recall the definition \eqref{E:Def_ZT} of $Z_T(\Gamma)$).

  Lastly, if $T<T'$, then $v_T=v_{T'}$ by the uniqueness of a weak solution to \eqref{E:GL_Lim} on $[0,T)$.
  Thus, setting $v=v_T$ on $[0,T)$ for each $T>0$, we can define a function
  \begin{align*}
    v \in C([0,\infty);L^2(S)^N)\cap L_{loc}^2([0,T);H^1(S)^N)\cap L_{loc}^4([0,\infty);L^4(S)^N),
  \end{align*}
  and we find that \eqref{Pf_TW:Full} holds with $v_T$ replaced by $v$ for all $T>0$ and that $v$ is a unique global weak solution to \eqref{E:GL_Lim} with initial data $v_0$.
\end{proof}

\subsection{Difference estimate on the surface} \label{SS:TFL_DiS}
Using the weak forms \eqref{E:WF_GLim} and \eqref{E:Av_Weak}, we derive the difference estimate of $\mathcal{M}_\varepsilon u^\varepsilon$ and $v$ stated in Theorem \ref{T:Df_Sur}.

\begin{proof}[Proof of Theorem \ref{T:Df_Sur}]
  Under the assumptions of Theorem \ref{T:Weak}, let $u^\varepsilon$ and $v$ be unique global weak solutions to \eqref{E:GL_CTD} and \eqref{E:GL_Lim}, respectively.
  We set
  \begin{alignat*}{3}
    v^\varepsilon &= \mathcal{M}_\varepsilon u^\varepsilon, &\quad V^\varepsilon &= v^\varepsilon-v &\quad &\text{on} \quad \Gamma\times[0,\infty), \\
    v_0^\varepsilon &= \mathcal{M}_\varepsilon u_0^\varepsilon, &\quad V_0^\varepsilon &= v_0^\varepsilon-v_0 &\quad &\text{on} \quad \Gamma.
  \end{alignat*}
  Then, for each fixed $T>0$, it follows form Definition \ref{D:WSL_Loc} and Lemma \ref{L:Av_Weak} that
  \begin{align*}
    V^\varepsilon \in E_T(\Gamma) \subset C([0,T];L^2(\Gamma)^N), \quad V^\varepsilon(0) = V_0^\varepsilon \quad\text{in}\quad L^2(\Gamma)^N.
  \end{align*}
  Moreover, for each $t\in[0,T]$ and $\zeta\in Z_t(\Gamma)$, we subtract \eqref{E:WF_GLim} from \eqref{E:Av_Weak} to get
  \begin{multline*}
    \int_0^t\langle\partial_sV^\varepsilon(s),g\zeta(s)\rangle_{(H^1\cap L^4)(\Gamma)}\,ds+\int_0^t\bigl(g\nabla_\Gamma V^\varepsilon(s),\nabla_\Gamma\zeta(s)\bigr)_{L^2(\Gamma)}\,ds \\
    +\lambda\int_0^t\bigl(g\{|v^\varepsilon(s)|^2v^\varepsilon(s)-|v(s)|^2v(s)\},\zeta(s)\bigr)_{L^2(\Gamma)}\,ds \\
    = \lambda\int_0^t\bigl(gV^\varepsilon(s),\zeta(s)\bigr)_{L^2(\Gamma)}\,ds+R_\varepsilon(\zeta;t),
  \end{multline*}
  where $R_\varepsilon(\zeta;t)$ satisfies \eqref{E:AvW_Res}.
  Let $\zeta=V^\varepsilon=v^\varepsilon-v$.
  Then, by \eqref{E:ET_IbP}, \eqref{Pf_GCU:Mono}, and $\lambda>0$,
  \begin{multline*}
    \frac{1}{2}\|g^{1/2}V^\varepsilon(t)\|_{L^2(\Gamma)}^2+\int_0^t\|g^{1/2}\nabla_\Gamma V^\varepsilon(s)\|_{L^2(\Gamma)}^2\,ds \\
    \leq \frac{1}{2}\|g^{1/2}V_0^\varepsilon\|_{L^2(\Gamma)}^2+\lambda\int_0^t\|g^{1/2}V^\varepsilon(s)\|_{L^2(\Gamma)}^2\,ds+|R_\varepsilon(V^\varepsilon;t)|
  \end{multline*}
  for all $t\in[0,T]$.
  Based on this, we proceed as in the proof of Lemma \ref{L:Av_Ena} to get
  \begin{align*}
    \max_{t\in[0,T]}\|V^\varepsilon(t)\|_{L^2(\Gamma)}^2+\int_0^T\|\nabla_\Gamma V^\varepsilon(t)\|_{L^2(\Gamma)}^2\,dt \leq ce^{c(1+\lambda)T}\Bigl\{\|V_0^\varepsilon\|_{L^2(\Gamma)}^2+\varepsilon^{6\alpha}(1+\lambda)^2\Bigr\}
  \end{align*}
  which gives \eqref{E:Df_Sur} (with a different constant $c>0$).
\end{proof}

\section{Difference estimate in the thin domain} \label{S:Diff_CT}
In this section, we prove the difference estimate \eqref{E:Df_CTD} in $\Omega_\varepsilon$.
Let $\mathcal{M}_\varepsilon$ be the weighted average operator, and let $\bar{\eta}=\eta\circ\pi$ be the constant extension of a function $\eta$ on $\Gamma$ in the normal direction of $\Gamma$.
We write $c$ for a general positive constant independent of $\varepsilon$ and $\lambda$.

As mentioned in Section \ref{S:Intro}, the difference estimate \eqref{E:Df_Sur} on $\Gamma$ does not give \eqref{E:Df_CTD}, since the estimate for the difference of $u^\varepsilon$ and $\mathcal{M}_\varepsilon u^\varepsilon$ requires a higher order regularity of $u^\varepsilon$ as in \eqref{E:Ave_Dif}.
In particular, since we consider a weak solution $u^\varepsilon$ to \eqref{E:GL_CTD} which does not have the $H^2$-regularity, the estimate for the gradient in \eqref{E:Df_CTD} cannot be obtained from \eqref{E:Df_Sur}.
To overcome this difficulty, we derive a weak form in $\Omega_\varepsilon$ satisfied by the constant extension $\bar{v}$ of a weak solution $v$ to \eqref{E:GL_Lim} and apply an energy method in $\Omega_\varepsilon$ to $u^\varepsilon-\bar{v}$.

\begin{lemma} \label{L:Ext_Dt}
  For a fixed $T>0$, let $v\in E_T(\Gamma)$.
  Then, $\bar{v}\in E_T(\Omega_\varepsilon)$ and
  \begin{align} \label{E:Ext_Dt}
    \int_0^T\langle\partial_t\bar{v}(t),\psi(t)\rangle_{(H^1\cap L^4)(\Omega_\varepsilon)}\,dt = \varepsilon\int_0^T\langle\partial_tv(t),g\mathcal{M}_\varepsilon\psi(t)\rangle_{(H^1\cap L^4)(\Gamma)}\,dt.
  \end{align}
\end{lemma}

\begin{proof}
  It follows from \eqref{E:CoEx_Lp} and \eqref{E:CoDe_L2} that $\bar{v}\in Z_T(\Omega_\varepsilon)$.

  Let $\psi\in C_c^1(0,T;(H^1\cap L^4)(\Omega_\varepsilon)^N)$.
  By \eqref{E:Ave_Lp}, \eqref{E:Ave_H1}, and $\varepsilon^{-1/4}\leq\varepsilon^{-1/2}$, we have
  \begin{align*}
    \mathcal{M}_\varepsilon\psi \in C_c^1(0,T;(H^1\cap L^4)(\Gamma)^N), \quad \|\mathcal{M}_\varepsilon\psi\|_{Z_T(\Gamma)} \leq c\varepsilon^{-1/2}\|\psi\|_{Z_T(\Omega_\varepsilon)}.
  \end{align*}
  Moreover, $\partial_t\mathcal{M}_\varepsilon\psi=\mathcal{M}_\varepsilon(\partial_t\psi)$ since the definition \eqref{E:Def_Ave} of $\mathcal{M}_\varepsilon$ is independent of time.
  Thus, noting that $g$ is also independent of time, we see by \eqref{E:Ave_Pair} that
  \begin{align} \label{Pf_ExDt:Pair}
    \begin{aligned}
      \varepsilon\int_0^T\langle\partial_tv(t),g\mathcal{M}_\varepsilon\psi(t)\rangle_{(H^1\cap L^4)(\Gamma)}\,dt &= -\varepsilon\int_0^T\bigl(v(t),g\partial_t\mathcal{M}_\varepsilon\psi(t)\bigr)_{L^2(\Gamma)}\,dt \\
      &= -\varepsilon\int_0^T\bigl(v(t),g[\mathcal{M}_\varepsilon(\partial_t\psi)](t)\bigr)_{L^2(\Gamma)}\,dt \\
      &= -\int_0^T\bigl(\bar{v}(t),\partial_t\psi(t)\bigr)_{L^2(\Gamma)}\,dt.
    \end{aligned}
  \end{align}
  Using this equality and
  \begin{align*}
    \|g\mathcal{M}_\varepsilon\psi\|_{Z_T(\Gamma)} \leq c\|\mathcal{M}_\varepsilon\psi\|_{Z_T(\Gamma)} \leq c\varepsilon^{-1/2}\|\psi\|_{Z_T(\Omega_\varepsilon)}
  \end{align*}
  by $g\in C^1(\Gamma)$, we find that, for all $\psi\in C_c^1(0,T;(H^1\cap L^4)(\Omega_\varepsilon)^N)$,
  \begin{align*}
    \left|\int_0^T\bigl(\bar{v}(t),\psi(t)\bigr)_{L^2(\Gamma)}\,dt\right| \leq c\varepsilon^{1/2}\|\partial_tv\|_{[Z_T(\Gamma)]'}\|\psi\|_{Z_T(\Omega_\varepsilon)}.
  \end{align*}
  Hence, $\partial_t\bar{v}\in[Z_T(\Omega_\varepsilon)]'$ and $\bar{v}\in E_T(\Omega_\varepsilon)$.

  When $\psi\in C_c^1(0,T;(H^1\cap L^4)(\Omega_\varepsilon)^N)$, we have \eqref{E:Ext_Dt} by \eqref{Pf_ExDt:Pair}.
  Since this space is dense in $Z_T(\Omega_\varepsilon)$, it follows that \eqref{E:Ext_Dt} also holds for $\psi\in Z_T(\Omega_\varepsilon)$.
\end{proof}

\begin{lemma} \label{L:WF_Ext}
  For $T>0$ and $v_0\in L^2(\Gamma)^N$, let $v$ be a weak solution to \eqref{E:GL_Lim} on $[0,T)$.
  Then, $\bar{v}\in E_T(\Omega_\varepsilon)$ and $\bar{v}(0)=\bar{v}_0$ in $L^2(\Omega_\varepsilon)^N$.
  Moreover,
  \begin{multline} \label{E:WF_Ext}
    \int_0^T\langle\partial_t\bar{v},\psi(t)\rangle_{(H^1\cap L^4)(\Omega_\varepsilon)}\,dt+\int_0^T\bigl(\nabla\bar{v}(t),\nabla\psi(t)\bigr)_{L^2(\Omega_\varepsilon)}\,dt \\
    +\lambda\int_0^T\bigl((|\bar{v}(t)|^2-1)\bar{v}(t),\psi(t)\bigr)_{L^2(\Omega_\varepsilon)}\,dt = S_\varepsilon(\psi;T)
  \end{multline}
  for all $\psi\in Z_T(\Omega_\varepsilon)$.
  Here, $S_\varepsilon(\psi;T)$ is linear in $\psi$ and satisfies
  \begin{align} \label{E:WFE_Res}
    |S_\varepsilon(\psi;T)| \leq c\varepsilon^{3/2}\int_0^T\|\nabla_\Gamma v(t)\|_{L^2(\Gamma)}\|\psi(t)\|_{H^1(\Omega_\varepsilon)}\,dt,
  \end{align}
  where $c>0$ is a constant independent of $\varepsilon$, $\lambda$, and $T$.
\end{lemma}

Note that $\bar{v}\in E_T(\Omega_\varepsilon)$ implies $\bar{v}\in C([0,T];L^2(\Omega_\varepsilon)^N)$ by Lemma \ref{L:ET_Con}.

\begin{proof}
  We have the regularity and initial condition of $\bar{v}$ by those of $v$ and Lemma \ref{L:Ext_Dt}.

  Let $\psi\in Z_T(\Omega_\varepsilon)$.
  We set $\zeta=\mathcal{M}_\varepsilon\psi$ in \eqref{E:WF_GLim} and multiply both sides by $\varepsilon$ to get
  \begin{multline*}
    \varepsilon\int_0^T\langle\partial_tv(t),g\mathcal{M}_\varepsilon\psi(t)\rangle_{(H^1\cap L^4)(\Gamma)}\,dt+\varepsilon\int_0^T\bigl(g\nabla_\Gamma v(t),\nabla_\Gamma\mathcal{M}_\varepsilon\psi(t)\bigr)_{L^2(\Gamma)}\,dt \\
    +\varepsilon\lambda\int_0^T\bigl(g(|v(t)|^2-1)v(t),\mathcal{M}_\varepsilon\psi(t)\bigr)_{L^2(\Gamma)}\,dt = 0.
  \end{multline*}
  To the first and third terms, we use \eqref{E:Ave_Pair} and ``unfold'' the weighted average.
  Then, we find that \eqref{E:WF_Ext} holds with residual term
  \begin{align*}
    S_\varepsilon(\psi;T) = \int_0^T\bigl(\nabla\bar{v}(t),\nabla\psi(t)\bigr)_{L^2(\Omega_\varepsilon)}\,dt-\varepsilon\int_0^T\bigl(g\nabla_\Gamma v(t),\nabla_\Gamma\mathcal{M}_\varepsilon(t)\bigr)_{L^2(\Gamma)}\,dt,
  \end{align*}
  which is linear in $\psi$.
  Moreover, \eqref{E:WFE_Res} follows from \eqref{E:AT_Diri}.
\end{proof}

\begin{remark} \label{R:WFE_Grd}
  In fact, we can also unfold the integral involving $\nabla_\Gamma\mathcal{M}_\varepsilon\psi$ by  using \eqref{E:Ave_TGr}, but we avoid to do that here since the estimate \eqref{E:WFE_Res} is enough for our purpose.
\end{remark}

\begin{remark} \label{R:WFE_Non}
  Contrary to the averaging method (see Lemma \ref{L:Ave_Non}), we can recover the cubic term of \eqref{E:WF_Ext} from that of \eqref{E:WF_GLim} without any error by the unfolding method.
  This enables us to remove the assumption $u_0^\varepsilon\in L^\infty(\Omega_\varepsilon)^N$ in Theorem \ref{T:Df_CTD}.
\end{remark}

Now, we are ready to prove the difference estimate \eqref{E:Df_CTD} in $\Omega_\varepsilon$.

\begin{proof}[Proof of Theorem \ref{T:Df_CTD}]
  For $u_0^\varepsilon\in L^2(\Omega_\varepsilon)^N$ and $v_0\in L^2(\Gamma)^N$, let $u^\varepsilon$ and $v$ be global weak solutions to \eqref{E:GL_CTD} and \eqref{E:GL_Lim}, respectively.
  We define
  \begin{align*}
    w^\varepsilon = u^\varepsilon-\bar{v} \quad\text{in}\quad \Omega_\varepsilon\times[0,\infty), \quad w_0^\varepsilon = u_0^\varepsilon-\bar{v}_0 \quad\text{in}\quad \Omega_\varepsilon.
  \end{align*}
  For each fixed $T>0$, it follows from Definition \ref{D:WSC_Loc} and Lemma \ref{L:WF_Ext} that
  \begin{align*}
    w^\varepsilon \in E_T(\Omega_\varepsilon) \subset C([0,T];L^2(\Omega_\varepsilon)^N), \quad w^\varepsilon(0) = w_0^\varepsilon \quad\text{in}\quad \Omega_\varepsilon.
  \end{align*}
  Moreover, for each $t\in[0,T]$ and $\psi\in Z_t(\Omega_\varepsilon)$, we subtract \eqref{E:WF_Ext} from \eqref{E:WF_GCT} to get
  \begin{multline*}
    \int_0^t\langle\partial_tw^\varepsilon(s),\psi(s)\rangle_{(H^1\cap L^4)(\Omega_\varepsilon)}\,ds+\int_0^t\bigl(\nabla w^\varepsilon(s),\nabla\psi(s)\bigr)_{L^2(\Omega_\varepsilon)}\,ds \\
    +\lambda\int_0^t\bigl(|u^\varepsilon(s)|^2u^\varepsilon(s)-|\bar{v}(s)|^2\bar{v}(s),\psi(s)\bigr)_{L^2(\Omega_\varepsilon)}\,ds \\
    = \lambda\int_0^t\bigl(w^\varepsilon(s),\psi(s)\bigr)_{L^2(\Omega_\varepsilon)}\,ds-S_\varepsilon(\psi;t).
  \end{multline*}
  We set $\psi=w^\varepsilon=u^\varepsilon-\bar{v}$ and use \eqref{E:ET_IbP}, \eqref{Pf_GCU:Mono}, and $\lambda>0$.
  Then,
  \begin{multline} \label{Pf_DC:Ineq}
    \frac{1}{2}\|w^\varepsilon(t)\|_{L^2(\Omega_\varepsilon)}^2+\int_0^t\|\nabla w^\varepsilon(s)\|_{L^2(\Omega_\varepsilon)}^2\,ds \\
    \leq \frac{1}{2}\|w_0^\varepsilon\|_{L^2(\Omega_\varepsilon)}^2+\lambda\int_0^t\|w^\varepsilon(s)\|_{L^2(\Omega_\varepsilon)}^2\,ds+|S_\varepsilon(\psi;t)|.
  \end{multline}
  Moreover, by \eqref{E:WFE_Res}, Young's inequality, and \eqref{E:GLim_Ena},
  \begin{align*}
    |S_\varepsilon(\psi;t)| &\leq c\varepsilon^{3/2}\int_0^t\|\nabla_\Gamma v(s)\|_{L^2(\Gamma)}\|w^\varepsilon(s)\|_{H^1(\Omega_\varepsilon)}\,ds \\
    &\leq c\varepsilon^3\int_0^t\|\nabla_\Gamma v(s)\|_{L^2(\Gamma)}^2\,ds+\frac{1}{2}\int_0^t\|w^\varepsilon(s)\|_{H^1(\Omega_\varepsilon)}^2\,ds \\
    &\leq c\varepsilon^3e^{c\lambda t}\|v_0\|_{L^2(\Gamma)}^2+\frac{1}{2}\int_0^t\Bigl(\|w^\varepsilon(s)\|_{L^2(\Omega_\varepsilon)}^2+\|\nabla w^\varepsilon(s)\|_{L^2(\Omega_\varepsilon)}^2\Bigr)\,ds.
  \end{align*}
  We apply this to \eqref{Pf_DC:Ineq}, make the integral of $\nabla w^\varepsilon$ appearing in the right-hand side absorbed into the left-hand side, and use $e^{c\lambda t}\leq e^{c\lambda T}$ for $t\in[0,T]$.
  Then, we get
  \begin{multline*}
    \|w^\varepsilon(t)\|_{L^2(\Omega_\varepsilon)}^2+\int_0^t\|\nabla w^\varepsilon(s)\|_{L^2(\Omega_\varepsilon)}^2\,ds \\
    \leq c\left(\|w_0^\varepsilon\|_{L^2(\Omega_\varepsilon)}^2+\varepsilon^3e^{c\lambda T}\|v_0\|_{L^2(\Gamma)}^2+(1+\lambda)\int_0^t\|w^\varepsilon(s)\|_{L^2(\Omega_\varepsilon)}^2\,ds\right)
  \end{multline*}
  for all $t\in[0,T]$.
  Applying Gronwall's inequality to this, we find that
  \begin{multline*}
    \max_{t\in[0,T]}\|w^\varepsilon(t)\|_{L^2(\Omega_\varepsilon)}^2+\int_0^T\|\nabla w^\varepsilon(t)\|_{L^2(\Omega_\varepsilon)}^2\,dt \\
    \leq ce^{c(1+\lambda)T}\Bigl(\|w_0^\varepsilon\|_{L^2(\Omega_\varepsilon)}^2+\varepsilon^3e^{c\lambda T}\|v_0\|_{L^2(\Gamma)}^2\Bigr).
  \end{multline*}
  Taking the square root of both sides, and using $1\leq e^{c\lambda T}\leq e^{c(1+\lambda)T}$, we get
  \begin{align*}
    \|w^\varepsilon\|_{C([0,T];L^2(\Omega_\varepsilon))}+\|\nabla w^\varepsilon\|_{L^2(0,T;L^2(\Omega_\varepsilon))} \leq ce^{c(1+\lambda)T}\Bigl(\|w_0^\varepsilon\|_{L^2(\Omega_\varepsilon)}+\varepsilon^{3/2}\|v_0\|_{L^2(\Gamma)}\Bigr).
  \end{align*}
  Hence, we obtain \eqref{E:Df_CTD} by dividing the above inequality by $\varepsilon^{1/2}$.
\end{proof}

\section{Outline of the Galerkin method} \label{S:Galerkin}
In this section, we give the outline of the proof of the existence of a global weak solution to \eqref{E:GL_CTD} by the Galerkin method.
For the sake of simplicity, we use the following notations and assumptions throughout this section:
\begin{itemize}
  \item We suppress the parameter $\varepsilon$, since the explicit dependence on $\varepsilon$ is not required.
  For example, we write $\Omega$, $u$, and $u_0$ instead of $\Omega_\varepsilon$, $u^\varepsilon$, and $u_0^\varepsilon$.
  \item We only consider the scalar-valued case ($N=1$), since the vector-valued case can be shown in the same way.
  \item We abbreviate function spaces $\mathcal{X}(\Omega)$ to $\mathcal{X}$.
\end{itemize}
For the global existence of a weak solution, it is enough to show the local existence on any finite time interval.
Indeed, if $u_T$ and $u_{T'}$ are weak solutions to \eqref{E:GL_CTD} on $[0,T)$ and $[0,T')$ with $T<T'$, then $u_T=u_{T'}$ on $[0,T)$ by the uniqueness of a weak solution (Lemma \ref{L:GCT_Uni}).
Hence, we can get a global weak solution $u$ by setting $u=u_T$ on $[0,T)$ for all $T>0$.

Let us give the outline of construction of a weak solution on $[0,T)$ with any fixed $T>0$ by the Galerkin method.

\subsection{Basis functions} \label{SS:Gal_Bas}
We first take basis functions of $H^1\cap L^4$ directly.

\begin{lemma} \label{L:Basis}
  There exists a countable subset $\{\phi_\ell\}$ of $H^1\cap L^4$ such that
  \begin{enumerate}
    \item the subspace spanned by $\{\phi_\ell\}$, i.e. the space
    \begin{align*}
      \mathcal{L}(\{\phi_\ell\}) = \{\alpha_1\phi_1+\cdots+\alpha_L\phi_L \mid L\in\mathbb{N}, \, \alpha_1,\dots,\alpha_L\in\mathbb{R}\},
    \end{align*}
    is dense in $H^1\cap L^4$, and
    \item $\{\phi_\ell\}$ forms an orthonormal basis of $L^2$.
  \end{enumerate}
\end{lemma}

\begin{proof}
  The space $H^1\cap L^4$ can be identified with $\{(u,\partial_1u,\dots,\partial_nu)^T \mid u\in H^1\cap L^4\}$, which is a subspace of the separable space $L^4\times(L^2)^n$.
  Hence, $H^1\cap L^4$ is also separable and we can take a countable set $\{\psi_\ell\}$ of linearly independent functions in $H^1\cap L^4$ such that $\mathcal{L}(\{\psi_\ell\})$ is dense in $H^1\cap L^4$.
  Then, we easily find that $\mathcal{L}(\{\psi_\ell\})$ is also dense in $L^2$, since $H^1\cap L^4$ is dense in $L^2$ and $\|u\|_{L^2}\leq\|u\|_{H^1\cap L^4}$ for $u\in H^1\cap L^4$.
  Hence, applying the Gram--Schmidt orthonormalization in $L^2$ to $\{\psi_\ell\}$, we can get a countable subset $\{\phi_\ell\}$ of $H^1\cap L^4$ satisfying (i) and (ii) (note that $\mathcal{L}(\{\psi_\ell\})=\mathcal{L}(\{\phi_\ell\})$ by the construction of $\{\phi_\ell\}$).
\end{proof}

For each $L\in\mathbb{N}$, let $\mathcal{P}_Lu=\sum_{\ell=1}^L(u,\phi_\ell)_{L^2}\phi_\ell$ be the orthogonal projection from $L^2$ onto the subspace spanned by $\{\phi_\ell\}_{\ell=1}^L$.
Note that $\mathcal{P}_Lu\to u$ in $L^2$ as $L\to\infty$, but we cannot say that the same convergence holds in $H^1\cap L^4$ even if $u\in H^1\cap L^4$, since the functions $\phi_\ell$ are not the eigenfunctions of the Neumann Laplacian.

\subsection{Approximate solutions and weak convergence} \label{SS:Gal_Ap}
For each $L\in\mathbb{N}$, we look for a function of the form $u_L(t)=\sum_{\ell=1}^L\alpha_\ell(t)\phi_\ell$ which satisfies $u_L(0)=\mathcal{P}_Lu_0$ and
\begin{align} \label{E:Gal_App}
  (\partial_tu_L(t),\phi_\ell)_{L^2}+(\nabla u_L(t),\nabla\phi_\ell)_{L^2}+\lambda\bigl((|u_L(t)|^2-1)u_L(t),\phi_\ell\bigr)_{L^2} = 0, \quad t\in(0,T)
\end{align}
for all $\ell=1,\dots,L$.
Note that the term $(|u_L(t)|^2u_L(t),\phi_\ell)_{L^2}$ makes sense, since
\begin{align*}
  |(|u_L(t)|^2u_L(t),\phi_\ell)_{L^2}| \leq \|u_L(t)\|_{L^4}^3\|\phi_\ell\|_{L^4}
\end{align*}
by H\"{o}lder's inequality and by $\phi_\ell\in L^4$ and thus $u_L(t)\in L^4$.

Since $\{\phi_\ell\}$ is orthonormal in $L^2$, this problem can be formulated as a system of ODEs for $\alpha_1(t),\dots,\alpha_L(t)$, which can be uniquely solved on some time interval by the Cauchy--Lipschitz theorem.
Moreover, multiplying \eqref{E:Gal_App} by $\alpha_\ell(t)$, summing over $\ell=1,\dots,L$, and integrating with respect to time, we can get the energy estimate
\begin{align} \label{E:Gal_Ena}
  \|u_L(t)\|_{L^2}^2+2\int_0^t\|\nabla u_L(s)\|_{L^2}^2\,ds+2\lambda\int_0^t\|u_L(s)\|_{L^4}^4\,ds \leq e^{2\lambda t}\|\mathcal{P}_Lu_0\|_{L^2}^2
\end{align}
as long as $u_L(t)$ exists, as in the proof of Lemma \ref{L:GCT_Ena}.
This gives the uniform boundedness of $\alpha_1(t),\dots,\alpha_L(t)$ on $[0,T]$, so we can extend $u_L(t)$ to the whole time interval $[0,T]$.

Let $Z_T=Z_T(\Omega)$ be the function space given by \eqref{E:Def_ZT} (recall that we suppress $\varepsilon$ of $\Omega_\varepsilon$).
We observe by \eqref{E:Gal_Ena} and $\|\mathcal{P}_Lu_0\|_{L^2}\leq\|u_0\|_{L^2}$ that
\begin{align} \label{E:Gal_ZT}
  \|u_L\|_{Z_T} = \max\{\|u_L\|_{L^2(0,T;H^1)},\|u_L\|_{L^4(0,T;L^4)}\} \leq C.
\end{align}
Here and in what follows, $C$ denotes a general positive constant depending on $\lambda$, $T$, and $\|u_0\|_{L^2}$ but independent of $L$.
Therefore, up to a subsequence, we have
\begin{align} \label{E:GaZT_We}
  \lim_{L\to\infty}u_L = u \quad\text{weakly in}\quad Z_T = L^2(0,T;H^1)\cap L^4(0,T;L^4)
\end{align}
with some $u\in Z_T$.
However, we cannot get a uniform estimate for $\partial_tu_L$ in $[Z_T]'$.
Indeed, to estimate $\partial_tu_L$, we need to take $\psi\in H^1\cap L^4$, test $\mathcal{P}_L\psi$ by using \eqref{E:Gal_App}, and estimate the resulting equality in terms of $\|\psi\|_{H^1}$ and $\|\psi\|_{L^4}$ uniformly in $L$, but we do not have
\begin{align*}
  \|\nabla\mathcal{P}_L\psi\|_{H^1} \leq c\|\psi\|_{H^1}, \quad \|\mathcal{P}_L\psi\|_{L^4} \leq c\|\psi\|_{L^4}
\end{align*}
with some constant $c>0$ independent of $L$.
Thus, Lemma \ref{L:ET_AuLi} is not available here, which was essential for the weak convergence of the cubic term in the proof of Theorem \ref{T:Weak} (see Section \ref{SS:TFL_WC}), and we need another approach to get the weak convergence of $|u_L|^2u_L$.

\subsection{Strong convergence} \label{SS:Gal_Str}
To circumvent the above difficulty, we show the strong convergence of $\{u_L\}$ by following the idea used in \cite{Mas84} and \cite[Chapter V, Theorem 6.7]{LaSoUr68}.

\begin{lemma} \label{L:Equi}
  For each fixed $k\in\mathbb{N}$, the sequence $\{(u_L(t),\phi_k)_{L^2}\}_{L=k}^\infty$ is uniformly bounded and equicontinuous on $[0,T]$.
\end{lemma}

\begin{proof}
  By H\"{o}lder's inequality, $\|\phi_k\|_{L^2}=1$, and \eqref{E:Gal_Ena}, we have
  \begin{align*}
    \max_{t\in[0,T]}|(u_L(t),\phi_k)_{L^2}| \leq \max_{t\in[0,T]}\|u_L(t)\|_{L^2} \leq e^{\lambda T}\|\mathcal{P}_Lu_0\|_{L^2} \leq e^{\lambda T}\|u_0\|_{L^2}.
  \end{align*}
  Thus, $\{(u_L(t),\phi_k)_{L^2}\}_{L=k}^\infty$ is uniformly bounded.
  Also, when $L\geq k$ and $0\leq s\leq t\leq T$,
  \begin{align*}
    &(u_L(t),\phi_k)_{L^2}-(u_L(s),\phi_k)_{L^2} = \int_s^t(\partial_\tau u_L(\tau),\phi_k)_{L^2}\,d\tau \\
    &= -\int_s^t(\nabla u_L(\tau),\nabla\phi_k)_{L^2}\,d\tau-\lambda\int_s^t\bigl((|u_L(t)|^2-1)u_L(t),\phi_k\bigr)_{L^2}\,d\tau
  \end{align*}
  by \eqref{E:Gal_App}.
  Hence, by H\"{o}lder's inequality in space and time, and by \eqref{E:Gal_ZT}, we get
  \begin{multline*}
    |(u_L(t),\phi_k)_{L^2}-(u_L(s),\phi_k)_{L^2}| \\
    \leq C\Bigl\{(t-s)^{1/2}\|\nabla\phi_k\|_{L^2}+\lambda(t-s)^{1/4}\|\phi_k\|_{L^4}+\lambda(t-s)^{1/2}\|\phi_k\|_{L^2}\Bigr\},
  \end{multline*}
  which shows that $\{(u_L(t),\phi_k)_{L^2}\}_{L=k}^\infty$ is equicontinuous on $[0,T]$.
\end{proof}

Based on Lemma \ref{L:Equi}, we apply the Ascoli--Arzel\'{a} theorem and a diagonal argument to take a subsequence of $\{u_L\}$, which is again denoted by $\{u_L\}$, such that $\{(u_L(t),\phi_k)_{L^2}\}_{L=1}^\infty$ converges uniformly on $[0,T]$ for each $k\in\mathbb{N}$.
Then, we can show the strong convergence of $\{u_L\}$ by using the following Friedrichs inequality.

\begin{lemma} \label{L:Fried}
  For each $\gamma>0$, there exists an $N_\gamma\in\mathbb{N}$ such that
  \begin{align} \label{E:Fried}
    \|w\|_{L^2}^2 \leq (1+\gamma)\sum_{k=1}^{N_\gamma}|(w,\phi_k)_{L^2}|^2+\gamma\|w\|_{H^1}^2
  \end{align}
  for all $w\in H^1$ (note that $N_\gamma$ does not depend on $w$).
\end{lemma}

\begin{proof}
  The lemma can be shown by a contradiction argument and the compact embedding $H^1\hookrightarrow L^2$.
  We refer to the proof of \cite[Chapter II, Lemma 2.4]{LaSoUr68} for details.
\end{proof}

For any $\gamma>0$, let $N_\gamma\in\mathbb{N}$ be given in Lemma \ref{L:Fried}.
We set $w=u_L(t)-u_M(t)$ in \eqref{E:Fried} and integrate the resulting inequality with respect to time.
Then, we get
\begin{multline*}
  \int_0^T\|u_L(t)-u_M(t)\|_{L^2}^2\,dt \\
  \leq (1+\gamma)\sum_{k=1}^{N_\gamma}\int_0^T|(u_L(t)-u_M(t),\phi_k)_{L^2}|^2\,dt+\gamma\int_0^T\|u_L(t)-u_M(t)\|_{H^1}^2\,dt.
\end{multline*}
Let $L,M\to\infty$ in this inequality.
Then, the first term on the right-hand side converges to zero, since $\{(u_L(t),\phi_k)_{L^2}\}_{L=1}^\infty$ converges uniformly on $[0,T]$ for each $k\in\mathbb{N}$ and $N_\gamma$ is independent of $L$ and $M$.
Also, the last integral is bounded by \eqref{E:Gal_ZT}.
Hence,
\begin{align*}
  \limsup_{L,M\to\infty}\int_0^T\|u_L(t)-u_M(t)\|_{L^2}^2\,dt \leq C\gamma \quad\text{for any}\quad \gamma>0,
\end{align*}
which shows that $\{u_L\}$ is Cauchy in $L^2(0,T;L^2)$ and thus converges strongly in the same space.
By this result, \eqref{E:GaZT_We}, and the uniqueness of a weak limit, we have $u_L\to u$ strongly in $L^2(0,T;L^2)$ as $L\to\infty$ and thus, up to a subsequence, $u_L\to u$ a.e. in $\Omega\times(0,T)$ (recall that we suppress $\varepsilon$).
Hence, $|u_L|^2u_L\to|u|^2u$ a.e. in $\Omega\times(0,T)$.
Moreover,
\begin{align*}
  \|\,|u_L|^2u_L\|_{L^{4/3}(0,T;L^{4/3})} = \|u_L\|_{L^4(0,T;L^4)}^3 \leq C
\end{align*}
by \eqref{E:Gal_Ena}.
Therefore, we can apply Lemma \ref{L:AC_Non} to get
\begin{align} \label{E:GaCu_We}
  \lim_{L\to\infty}|u_L|^2u_L = |u|^2u \quad\text{weakly in}\quad L^{4/3}(0,T;L^{4/3}).
\end{align}

\subsection{Characterization of the limit} \label{SS:Gal_Cha}
Let us show that $u$ is a weak solution to \eqref{E:GL_CTD} on $[0,T)$.
For each fixed $\ell\in\mathbb{N}$, let $L\geq\ell$.
We multiply \eqref{E:Gal_App} by $\theta\in C_c^1([0,T))$ and integrate the resulting equality with respect to time.
Then, by integration by parts,
\begin{multline*}
  -\int_0^T(u_L(t),\theta'(t)\phi_\ell)_{L^2}\,dt+\int_0^T(\nabla u_L(t),\theta(t)\nabla\phi_\ell)_{L^2}\,dt \\
  +\lambda\int_0^T\bigl((|u_L(t)|^2-1)u_L(t),\theta(t)\phi_\ell\bigr)_{L^2}\,dt = (\mathcal{P}_Lu_0,\theta(0)\phi_\ell)_{L^2}.
\end{multline*}
We send $L\to\infty$ and use \eqref{E:GaZT_We}, \eqref{E:GaCu_We}, and $\mathcal{P}_Lu_0\to u_0$ in $L^2$ to get
\begin{multline*}
  -\int_0^T(u(t),\theta'(t)\phi_\ell)_{L^2}\,dt+\int_0^T(\nabla u(t),\theta(t)\nabla\phi_\ell)_{L^2}\,dt \\
  +\lambda\int_0^T\bigl((|u(t)|^2-1)u(t),\theta(t)\phi_\ell\bigr)_{L^2}\,dt = (u_0,\theta(0)\phi_\ell)_{L^2}
\end{multline*}
for all $\ell\in\mathbb{N}$.
Hence, we have
\begin{multline} \label{E:Gal_Ano}
  \int_0^T\bigl(u(t),\partial_t\psi(t)\bigr)_{L^2}\,dt+\int_0^T\bigl(\nabla u(t),\nabla\psi(t)\bigr)_{L^2}\,dt \\
  +\lambda\int_0^T\bigl((|u(t)|^2-1)u(t),\psi(t)\bigr)_{L^2}\,dt = \bigl(u_0,\psi(0)\bigr)_{L^2}
\end{multline}
for all $\psi$ in the space (recall that $\mathcal{L}(\{\phi_\ell\})=\{\sum_{\ell=1}^L\alpha_\ell\phi_\ell \mid L\in\mathbb{N}, \, \alpha_\ell\in\mathbb{R}\}$)
\begin{align*}
  \mathcal{U} = \{\textstyle\sum_{k=1}^K\theta_k(t)w_k \mid K\in\mathbb{N}, \, \theta_k\in C_c^1([0,T)), \, w_k\in\mathcal{L}(\{\phi_\ell\})\}.
\end{align*}
Let $\psi\in C_c^1([0,T);H^1\cap L^4)$.
Then, since $\mathcal{L}(\{\phi_\ell\})$ is dense in $H^1\cap L^4$, we can prove, as in the proof of \cite[Lemma 2.2]{Mas84}, that there exist functions $\psi_K\in\mathcal{U}$ such that
\begin{align*}
  \lim_{K\to\infty}\|\psi-\psi_K\|_{C([0,T],H^1\cap L^4)} = \lim_{K\to\infty}\|\partial_t\psi-\partial_t\psi_K\|_{L^2(0,T;H^1\cap L^4)} = 0.
\end{align*}
Hence, \eqref{E:Gal_Ano} holds for the above $\psi$, since it is valid for each $\psi_K\in\mathcal{U}$.
In particular,
\begin{gather*}
  \left|\int_0^T\bigl(u(t),\partial_t\psi(t)\bigr)_{L^2}\,dt\right| \leq C(u,\lambda,T)\|\psi\|_{Z_T}, \\
  C(u,\lambda,T) = \|\nabla u\|_{L^2(0,T;L^2)}+\lambda\Bigl(\|u\|_{L^4(0,T;L^4)}^3+\|u\|_{L^2(0,T;L^2)}\Bigr)
\end{gather*}
for all $\psi\in C_c^1(0,T;H^1\cap L^4)$ by $\psi(0)=0$ and H\"{o}lder's inequality, which shows that
\begin{align*}
  \partial_tu\in [Z_T]', \quad u\in E_T \subset C([0,T];L^2).
\end{align*}
Here, $E_T=E_T(\Omega)$ is given by \eqref{E:Def_ET} and we used Lemma \ref{L:ET_Con} (recall that we suppress $\varepsilon$).
Moreover, when $\psi\in C_c^1(0,T;H^1\cap L^4)$, we apply \eqref{E:ET_IbP} to \eqref{E:Gal_Ano} to get
\begin{multline} \label{E:Gal_WF}
  \int_0^T\langle\partial_tu(t),\psi(t)\rangle_{H^1\cap L^4}\,dt+\int_0^T\bigl(\nabla u(t),\nabla\psi(t)\bigr)_{L^2}\,dt \\
  +\lambda\int_0^T\bigl((|u(t)|^2-1)u(t),\psi(t)\bigr)_{L^2}\,dt = 0.
\end{multline}
Since $C_c^1(0,T;H^1\cap L^4)$ is dense in $Z_T$, it follows that \eqref{E:Gal_WF} also holds for all $\psi\in Z_T$.

It remains to verify the initial condition.
Let $\theta\in C_c^1([0,T))$ satisfy $\theta(0)=1$.
For any $\psi_0\in C_c^\infty$, we take $\psi(t)=\theta(t)\psi_0\in Z_T$ in \eqref{E:Gal_WF} and use \eqref{E:ET_IbP} and $\psi(0)=\psi_0$ to get
\begin{multline*}
  -\int_0^T\bigl(u(t),\partial_t\psi(t)\bigr)_{L^2}\,dt+\int_0^T\bigl(\nabla u(t),\nabla\psi(t)\bigr)_{L^2}\,dt \\
  +\lambda\int_0^T\bigl((|u(t)|^2-1)u(t),\psi(t)\bigr)_{L^2}\,dt = (u(0),\psi_0)_{L^2}.
\end{multline*}
By this equality and \eqref{E:Gal_Ano}, we have $(u(0),\psi_0)_{L^2}=(u_0,\psi_0)_{L^2}$ for all $\psi\in C_c^\infty$.
Since $C_c^\infty$ is dense in $L^2$, we find that $u(0)=u_0$ in $L^2$ and $u$ is a weak solution to \eqref{E:GL_CTD} on $[0,T)$.

\section{Proofs of auxiliary lemmas} \label{S:Pf_Aux}
This section gives the proofs of Lemmas \ref{L:ET_Den}, \ref{L:Trn_Sp}, \ref{L:Trn_Ti}.

\begin{proof}[Proof of Lemmas \ref{L:ET_Den}]
  As mentioned in Section \ref{SS:FS_TS}, we omit the proof of (i), since it can be shown by standard cut-off and mollification arguments.

  Let us show (ii).
  In what follows, we suppress the domain $S$ and the superscript $N$ of function spaces on $S$ for the sake of simplicity.
  For example, we write $E_T$ and $H^1\cap L^4$ instead of $E_T(S)$ and $(H^1\cap L^4)(S)^N$.
  We take a function $\theta\in C^\infty([0,T])$ such that
  \begin{align*}
    0\leq \theta\leq 1 \quad\text{on}\quad [0,T], \quad \theta(t) =
    \begin{cases}
      1 &(0\leq t\leq T/3), \\
      0 &(2T/3\leq t\leq T).
    \end{cases}
  \end{align*}
  Let $u\in E_T$.
  We set $u_1=\theta u$ and $u_2=(1-\theta)u$.
  Clearly, $u_1,u_2\in Z_T$.
  Moreover,
  \begin{align*}
    \int_0^T\bigl(u_1(t),\partial_t\psi(t)\bigr)_{L^2}\,dt &= \int_0^T\bigl(u(t),\partial_t[\theta\psi](t)\bigr)_{L^2}\,dt-\int_0^T\bigl(\partial_t\theta(t)u(t),\psi(t)\bigr)_{L^2}\,dt \\
    &= -\int_0^T\langle\partial_tu(t),[\theta\psi](t)\rangle_{H^1\cap L^4}\,dt-\int_0^T\bigl(\partial_t\theta(t)u(t),\psi(t)\bigr)_{L^2}\,dt
  \end{align*}
  for each $\psi\in C_c^1(0,T;H^1\cap L^4)$, and thus, by $\theta\in C^\infty([0,T])$,
  \begin{align*}
    \left|\int_0^T\bigl(u_1(t),\partial_t\psi(t)\bigr)_{L^2}\,dt\right| \leq c\Bigl(\|\partial_tu\|_{[Z_T]'}+\|u\|_{L^2(0,T;L^2)}\Bigr)\|\psi\|_{Z_T},
  \end{align*}
  which shows $\partial_tu_1\in [Z_T]'$.
  Similarly, $\partial_tu_2\in[Z_T]'$, and we easily find that $\partial_tu=\partial_tu_1+\partial_tu_2$.
  Hence, it is sufficient to approximate $u_1$ and $u_2$ separately in $E_T$.
  In what follows, we only give the approximation of $u_1$, since $u_2$ can be approximated in the same way.
  Moreover, we rewrite $u_1$ as $u$ for the sake of simplicity.

  Now, let $u\in E_T$ satisfy $u(t)=0$ when $2T/3\leq t\leq T$.
  We extend $u$ to $(0,2T)$ by setting $u(t)=0$ for $t\geq T$.
  Clearly, $u\in Z_{2T}$.
  Moreover, we see that $\partial_tu\in[Z_{2T}]'$.
  Indeed, we take a function $\Theta\in C^\infty([0,2T])$ such that
  \begin{align*}
    0\leq \Theta \leq 1 \quad\text{on}\quad [0,2T], \quad \Theta(t) =
    \begin{cases}
      1 &(0\leq t\leq 3T/4) \\
      0 &(5T/6 \leq t\leq 2T).
    \end{cases}
  \end{align*}
  Let $\psi\in C_c^1(0,2T;H^1\cap L^4)$.
  Since $\Theta\psi\in C_c^1(0,T;H^1\cap L^4)$, and since $u(t)=0$ for $t\in[2T/3,T]$ and $\partial_t\Theta(t)=0$ for $t\in[0,2T/3]$, it follows that
  \begin{align*}
    \int_0^{2T}\bigl(u(t),\partial_t\psi(t)\bigr)_{L^2}\,dt &= \int_0^{2T/3}\bigl(u(t),\partial_t\psi(t)\bigr)_{L^2}\,dt = \int_0^{2T/3}\bigl(u(t),\partial_t[\Theta\psi](t)\bigr)_{L^2}\,dt \\
    &= \int_0^T\bigl(u(t),\partial_t[\Theta\psi](t)\bigr)_{L^2}\,dt = -\int_0^T\langle\partial_tu(t),[\Theta\psi](t)\rangle_{H^1\cap L^4}\,dt
  \end{align*}
  and thus, by $\partial_tu\in[Z_T]'$ and $\Theta\in C^\infty([0,2T])$, we have
  \begin{align*}
    \left|\int_0^{2T}\bigl(u(t),\partial_t\psi(t)\bigr)_{L^2}\,dt\right| \leq \|\partial_tu\|_{[Z_T]'}\|\Theta\psi\|_{Z_T} \leq c\|\partial_tu\|_{[Z_T]'}\|\psi\|_{Z_{2T}}.
  \end{align*}
  Therefore, $\partial_tu\in[Z_{2T}]'$, and we observe by \eqref{E:Dual_ZT} that $\partial_tu$ is of the form
  \begin{align*}
    \partial_tu = v_1+v_2, \quad v_1\in L^2(0,2T;[H^1]'), \quad v_2\in L^{4/3}(0,2T;L^{4/3}).
  \end{align*}
  Note that we do not know whether $v_1(t)=0$ and $v_2(t)=0$ for $t\geq T$.
  Although we may get $\partial_tu(t)=0$ for $t\geq T$, it may be possible that $v_1(t)\neq0$ and $v_2(t)=-v_1(t)$.
  However, this does not matter for the approximation of $\partial_tu$ on the shorter time interval $(0,T)$.

  Let $h\in (0,T/2)$.
  For $t\in(-h,2T-h)$, we set
  \begin{align*}
    u_h(t) = u(t+h), \quad v_{1,h}(t) = v_1(t+h), \quad v_{2,h}(t) = v_2(t+h).
  \end{align*}
  Then, we easily observe that
  \begin{gather*}
    u_h \in L^2(-h,2T-h;H^1)\cap L^4(-h,2T-h;L^4), \\
    v_{1,h} \in L^2(-h,2T-h;[H^1]'), \quad v_{2,h} \in L^{4/3}(-h,2T-h;L^{4/3}),
  \end{gather*}
  and $\partial_tu_h=v_{1,h}+v_{2,h}$ on $(-h,2T-h)$.
  For $\tau\in(0,h/2)$, let
  \begin{align*}
    w_{h,\tau}(t) = \int_{-\infty}^\infty\frac{1}{\tau}\rho\left(\frac{t-s}{\tau}\right)w_h(s)\,ds, \quad t\in\mathbb{R}, \quad w_h=u_h,v_{1,h},v_{2,h}
  \end{align*}
  be the mollification of $w_h$, where $\rho$ is a standard 1D mollifier and $w_h$ is extended to $\mathbb{R}$ by zero outside of $(-h,2T-h)$.
  Then, we have $u_{h,\tau}\in C^\infty(\mathbb{R};H^1\cap L^4)$ and
  \begin{align} \label{Pf_ED:uht}
    \|u-u_{h,\tau}\|_{Z_T}\to0, \, \|v_1-v_{1,h,\tau}\|_{L^2(0,T;[H^1]')}\to0, \, \|v_2-v_{2,h,\tau}\|_{L^{4/3}(0,T;L^{4/3})} \to 0
  \end{align}
  as $h,\tau\to0$ by the integrability of $u$, $v_1$, and $v_2$.
  Moreover, since $[0,T]\subset(-h,2T-h)$ by $h\in(0,T/2)$, we can show that $\partial_tu_{h,\tau}=v_{1,h,\tau}+v_{2,h,\tau}$ on $(0,T)$ in a standard manner by testing functions supported in $(0,T)$ and using Fubini's theorem.
  Hence,
  \begin{gather*}
    \partial_tu-\partial_tu_{h,\tau} = (v_1-v_{1,h,\tau})+(v_2-v_{2,h,\tau}) \quad\text{on}\quad (0,T), \\
    v_1-v_{1,h,\tau} \in L^2(0,T;[H^1]'), \quad v_2-v_{2,h,\tau} \in L^{4/3}(0,T;L^{4/3}).
  \end{gather*}
  Now, we recall that $[Z_T]'$ is of the form \eqref{E:Dual_ZT} and the norm $\|\cdot\|_{X_0+X_1}$ is given by \eqref{E:Def_X0X1} for Banach spaces $X_0$ and $X_1$.
  Hence, it follows from \eqref{Pf_ED:uht} that
  \begin{align*}
    \|\partial_tu-\partial_tu_{h,\tau}\|_{[Z_T]'} \leq \|v_1-v_{1,h,\tau}\|_{L^2(0,T;[H^1]')}+\|v_2-v_{2,h,\tau}\|_{L^{4/3}(0,T;L^{4/3})} \to 0
  \end{align*}
  as $h,\tau\to0$.
  By this result and \eqref{Pf_ED:uht}, we find that $u_{h,\tau}\to u$ in $E_T$ as $h,\tau\to0$.
  Therefore, the statement (ii) of Lemma \ref{L:ET_Den} follows.
\end{proof}

To prove Lemmas \ref{L:Trn_Sp} and \ref{L:Trn_Ti}, we prepare auxiliary functions.

\begin{lemma} \label{L:Lambda}
  Let $C_0,\gamma>0$ be constants.
  For $z\in\mathbb{R}$, we define
  \begin{align} \label{E:Def_Lam}
    \lambda_\gamma(z) =
    \begin{cases}
      0 &\text{if} \quad z\leq C_0, \\
      \sqrt{(z-C_0)^2+\gamma^2}-\gamma &\text{if}\quad z\geq C_0,
    \end{cases}
    \quad \Lambda_\gamma(z) = \frac{\lambda_\gamma(z)}{z}.
  \end{align}
  Then, $\lambda_\gamma,\Lambda_\gamma\in C^1(\mathbb{R})$ and
  \begin{align} \label{E:Lambda}
    |\lambda_\gamma(z)| \leq |z|, \quad |\Lambda_\gamma(z)| \leq 1, \quad \lim_{\gamma\to0}\lambda_\gamma(z) = (z-C_0)_+, \quad \lim_{\gamma\to0}\Lambda_\gamma(z) = \frac{(z-C_0)_+}{z}
  \end{align}
  for all $z\in\mathbb{R}$.
  Moreover, for all $z\in\mathbb{R}$,
  \begin{align} \label{E:La_De}
    \begin{alignedat}{2}
      |\lambda_\gamma'(z)| &\leq \chi_{(C_0,\infty)}(z), &\quad \lim_{\gamma\to0}\lambda_\gamma'(z) &= \chi_{(C_0,\infty)}(z), \\
      |\Lambda_\gamma'(z)| &\leq \chi_{(C_0,\infty)}(z)\frac{2}{z}, &\quad \lim_{\gamma\to0}\Lambda_\gamma'(z) &= \chi_{(C_0,\infty)}(z)\frac{C_0}{z^2}.
    \end{alignedat}
  \end{align}
  Here and in what follows, $\chi_I(z)$ denotes the characteristic function of $I\subset\mathbb{R}$.
\end{lemma}

\begin{proof}
  We have \eqref{E:Lambda} by direct calculations, $z-C_0\leq z$, and
  \begin{align} \label{Pf_La:Rat}
    0 \leq \sqrt{\xi^2+\gamma^2}-\gamma = \frac{\xi^2}{\sqrt{\xi^2+\gamma^2}+\gamma} \leq |\xi|, \quad \xi\in\mathbb{R}.
  \end{align}
  Also, it is easy to observe (with a few discussions when $z=C_0$) that
  \begin{align*}
    \lambda_\gamma'(z) = \chi_{(C_0,\infty)}(z)\frac{z-C_0}{\sqrt{(z-C_0)^2+\gamma^2}}, \quad \Lambda'(z) = \frac{\lambda_\gamma'(z)}{z}-\frac{\lambda_\gamma(z)}{z^2}
  \end{align*}
  for all $z\in\mathbb{R}$.
  Hence, $\lambda_\gamma,\Lambda_\gamma\in C^1(\mathbb{R})$ and \eqref{E:La_De} follows.
\end{proof}

\begin{lemma} \label{L:Sig_F}
  Let $C_0,\gamma>0$ be constants such that $2\gamma\leq C_0$.
  For $a\in\mathbb{R}^N$, we define
  \begin{align} \label{E:Def_SFg}
    \sigma_\gamma(a) = \sqrt{|a|^2+\gamma^2}-\gamma, \quad F_\gamma(a) = \Lambda_\gamma\bigl(\sigma_\gamma(a)\bigr)a.
  \end{align}
  where $\Lambda_\gamma$ is given by \eqref{E:Def_Lam}.
  Then, $\sigma_\gamma\in C^1(\mathbb{R}^N)$, $F_\gamma\in C^1(\mathbb{R}^N)^N$, and
  \begin{align} \label{E:Sig_Fg}
    |\sigma_\gamma(a)| \leq |a|, \quad |F_\gamma(a)| \leq |a|, \quad \lim_{\gamma\to0}\sigma_\gamma(a) = |a|, \quad \lim_{\gamma\to0}F_{\gamma}(a) = \frac{(|a|-C_0)_+}{|a|}\,a
  \end{align}
  for all $a\in\mathbb{R}^N$.
  Moreover, for all $a\in\mathbb{R}^N$ and $k=1,\dots,N$,
  \begin{align} \label{E:Sig_De}
    \left|\frac{\partial\sigma_\gamma}{\partial a_k}(a)\right| \leq 1, \quad \lim_{\gamma\to0}\frac{\partial\sigma_\gamma}{\partial a_k}(a) =
    \begin{cases}
      a_k/|a| &\text{if} \quad a\neq0, \\
      0 &\text{if} \quad a=0.
    \end{cases}
  \end{align}
  Also, the $j$-th component $F_{\gamma,j}(a)=\Lambda_\gamma\bigl(\sigma_\gamma(a)\bigr)a_j$ satisfies
  \begin{align} \label{E:Fg_De}
     \left|\frac{\partial F_{\gamma,j}}{\partial a_k}(a)\right| \leq 5, \quad \lim_{\gamma\to0}\frac{\partial F_{\gamma,j}}{\partial a_k}(a) &= \frac{(|a|-C_0)_+}{|a|}\,\delta_{jk}+\chi_{(C_0,\infty)}(|a|)\frac{C_0a_ja_k}{|a|^3}
  \end{align}
  for all $a\in\mathbb{R}^N$ and $j,k=1,\dots,N$, where $\delta_{jk}$ is the Kronecker delta.
\end{lemma}

\begin{proof}
  We have \eqref{E:Sig_Fg} by direct calculations, \eqref{E:Lambda}, and \eqref{Pf_La:Rat}.
  Also,
  \begin{align*}
    \frac{\partial\sigma_\gamma}{\partial a_k}(a) = \frac{a_k}{\sqrt{|a|^2+\gamma^2}}, \quad \frac{\partial F_{\gamma,j}}{\partial a_k}(a) = \Lambda_\gamma\bigl(\sigma_\gamma(a)\bigr)\delta_{jk}+\Lambda_\gamma'\bigl(\sigma_\gamma(a)\bigr)\frac{\partial\sigma_\gamma}{\partial a_k}(a)a_j.
  \end{align*}
  By these expressions and \eqref{E:Lambda}, \eqref{E:La_De}, and \eqref{E:Sig_Fg}, we find that \eqref{E:Sig_De} and the second relation of \eqref{E:Fg_De} are valid.
  To get the first relation of \eqref{E:Fg_De}, we see that
  \begin{align} \label{Pf_SFg:FD}
    \left|\frac{\partial F_{\gamma,j}}{\partial a_k}(a)\right| \leq 1+\chi_{(C_0,\infty)}\bigl(\sigma_\gamma(a)\bigr)\frac{2}{\sigma_\gamma(a)}\,|a_j|
  \end{align}
  by \eqref{E:Lambda}, \eqref{E:La_De}, and \eqref{E:Sig_De}.
  Moreover, since $\sigma_\gamma(a)\leq|a|$ and $2\gamma\leq C_0$,
  \begin{align*}
    \chi_{(C_0,\infty)}\bigl(\sigma_\gamma(a)\bigr) \leq \chi_{(C_0,\infty)}(|a|) \leq \chi_{(2\gamma,\infty)}(|a|).
  \end{align*}
  We also observe that $\sigma_\gamma(a)\geq|a|-\gamma\geq|a|/2$ if $|a|\geq2\gamma$.
  Hence,
  \begin{align*}
    \chi_{(C_0,\infty)}\bigl(\sigma_\gamma(a)\bigr)\frac{1}{\sigma_\gamma(a)} \leq \chi_{(2\gamma,\infty)}(|a|)\frac{1}{\sigma_\gamma(a)} \leq \chi_{(2\gamma,\infty)}(|a|)\frac{2}{|a|}.
  \end{align*}
  We apply this to \eqref{Pf_SFg:FD} and use $|a_j|/|a|\leq1$ to get the first relation of \eqref{E:Fg_De}.
\end{proof}

Now, let us give the proofs of Lemmas \ref{L:Trn_Sp} and \ref{L:Trn_Ti}.

\begin{proof}[Proof of Lemma \ref{L:Trn_Sp}]
  Let $C_0>0$ be a constant.
  We take a small $\gamma>0$ such that $2\gamma\leq C_0$, and define the mappings $F$ and $F_\gamma$ by \eqref{E:Def_Trn} and \eqref{E:Def_SFg}, respectively.

  Let $u\in H^1(\Omega_\varepsilon)^N$.
  For $i=1,\dots,n$ and $j=1,\dots,N$, we set
  \begin{align*}
    G_{i,j}(u) = \frac{(|u|-C_0)_+}{|u|}\frac{\partial u_j}{\partial x_i}+\sum_{k=1}^N\chi_{(C_0,\infty)}(|u|)\frac{C_0u_ju_k}{|u|^3}\frac{\partial u_k}{\partial x_i} \quad\text{in}\quad \Omega_\varepsilon.
  \end{align*}
  Note that $G_{i,j}(u)\in L^2(\Omega_\varepsilon)$ by $|G_{i,j}(u)|\leq c|\nabla u|$ with some constant $c>0$.

  Since $F_\gamma\in C^1(\mathbb{R}^N)^N$ and the inequalities in \eqref{E:Sig_Fg} and \eqref{E:Fg_De} hold, we can show that
  \begin{align*}
    F_\gamma(u) \in H^1(\Omega_\varepsilon)^N, \quad \frac{\partial}{\partial x_i}\Bigl(F_{\gamma,j}(u)\Bigr) = \sum_{k=1}^N\frac{\partial F_{\gamma,j}}{\partial a_k}(u)\frac{\partial u_k}{\partial x_i} \quad\text{a.e. in}\quad \Omega_\varepsilon
  \end{align*}
  as in the proof of \cite[Lemma 7.5]{GilTru01}.
  Moreover, by \eqref{E:Sig_Fg} and \eqref{E:Fg_De}, we have
  \begin{align*}
    \lim_{\gamma\to0}F_\gamma(u) = F(u), \quad \lim_{\gamma\to0}\frac{\partial}{\partial x_i}\Bigl(F_{\gamma,j}(u)\Bigr) = G_{i,j}(u) \quad\text{a.e. in}\quad \Omega_\varepsilon.
  \end{align*}
  It also follows from \eqref{E:Sig_Fg}, \eqref{E:Fg_De}, and direct calculations that
  \begin{align*}
    |F_\gamma(u)-F(u)| &\leq |F_\gamma(u)|+|F(u)| \leq 2|u|, \\
    \left|\frac{\partial}{\partial x_i}\Bigl(F_{\gamma,j}(u)\Bigr)-G_{i,j}(u)\right| &\leq \left|\frac{\partial}{\partial x_i}\Bigl(F_{\gamma,j}(u)\Bigr)\right|+|G_{i,j}(u)| \leq c|\nabla u|
  \end{align*}
  a.e. in $\Omega_\varepsilon$, where $c>0$ is a constant independent of $\gamma$.
  Since $|u|,|\nabla u|\in L^2(\Omega_\varepsilon)$, we can apply the dominated convergence theorem to deduce that
  \begin{align*}
    \lim_{\gamma\to0}\|F_\gamma(u)-F(u)\|_{L^2(\Omega_\varepsilon)} = \lim_{\gamma\to0}\left\|\frac{\partial}{\partial x_i}\Bigl(F_{\gamma,j}(u)\Bigr)-G_{i,j}(u)\right\|_{L^2(\Omega_\varepsilon)} = 0.
  \end{align*}
  Therefore, letting $\gamma\to0$ in
  \begin{align*}
    \int_{\Omega_\varepsilon}F_{\gamma,j}(u)\frac{\partial\varphi}{\partial x_i}\,dx = -\int_{\Omega_\varepsilon}\varphi\,\frac{\partial}{\partial x_i}\Bigl(F_{\gamma,j}(u)\Bigr)\,dx, \quad \varphi\in C_c^\infty(\Omega_\varepsilon),
  \end{align*}
  we find that $F(u)\in H^1(\Omega_\varepsilon)^N$ and \eqref{E:Trn_Dxi} holds.
\end{proof}

\begin{proof}[Proof of Lemma \ref{L:Trn_Ti}]
  Let $C_0>0$ be a constant and $F$ be given by \eqref{E:Def_Trn}.
  Also, let $F_j$ be the $j$-th component of $F$ for $j=1,\dots,N$.
  If $u\in Z_T(\Omega_\varepsilon)$, then $F(u)\in Z_T(\Omega_\varepsilon)$ since
  \begin{align*}
    |F(u)| \leq |u|, \quad \left|\frac{\partial}{\partial x_i}\Bigl(F_j(u)\Bigr)\right| \leq c|\nabla u| \quad\text{a.e. in}\quad \Omega_\varepsilon\times(0,T)
  \end{align*}
  for $i=1,\dots,n$ and $j=1,\dots,N$ by \eqref{E:Trn_Dxi}.

  Next, let $u\in E_T(\Omega_\varepsilon)$.
  To get \eqref{E:Trn_Dt}, it is sufficient to prove that
  \begin{align} \label{Pf_TT:Goal}
    -\frac{1}{2}\int_0^T\theta'(t)\|(|u(t)|-C_0)_+\|_{L^2(\Omega_\varepsilon)}^2\,dt = \int_0^T\theta(t)\bigl\langle\partial_tu(t), F\bigl(u(t)\bigr)\bigr\rangle_{(H^1\cap L^4)(\Omega_\varepsilon)}\,dt
  \end{align}
  for all $\theta\in C_c^\infty(0,T)$.
  Let us show \eqref{Pf_TT:Goal} in three steps.
  We write
  \begin{align*}
    Q_{\varepsilon,T} = \Omega_\varepsilon\times(0,T), \quad \overline{Q}_{\varepsilon,T} = \overline{\Omega}_\varepsilon\times[0,T]
  \end{align*}
  in the rest of the proof for the sake of simplicity.

  \textbf{Step 1:} for $u\in C^\infty(\overline{Q}_{\varepsilon,T})^N$, we show that
  \begin{align} \label{Pf_TT:Smth}
    -\frac{1}{2}\int_0^T\theta'(t)\|(|u(t)|-C_0)_+\|_{L^2(\Omega_\varepsilon)}^2\,dt = \int_0^T\theta(t)\Bigl(\partial_tu(t),F\bigl(u(t)\bigr)\Bigr)_{L^2(\Omega_\varepsilon)}\,dt.
  \end{align}
  Let $\gamma>0$ satisfy $2\gamma\leq C_0$, and let $\lambda_\gamma$ and $\sigma_\gamma$ be given by \eqref{E:Def_Lam} and \eqref{E:Def_SFg}, respectively.
  Since $\lambda_\gamma\circ\sigma_\gamma\circ u$ is of class $C^1$, we can differentiate it pointwisely to get
  \begin{align*}
    \frac{1}{2}\frac{\partial}{\partial t}\Bigl(\bigl\{[\lambda_\gamma\circ\sigma_\gamma]\bigl(u(t)\bigr)\bigr\}^2\Bigr) &= [\lambda_\gamma\circ\sigma_\gamma]\bigl(u(t)\bigr)[\lambda_\gamma'\circ\sigma_\gamma]\bigl(u(t)\bigr)\sum_{k=1}^N\frac{\partial\sigma_\gamma}{\partial a_k}\bigl(u(t)\bigr)\frac{\partial u_k}{\partial t}(t) \\
    &= \frac{[\lambda_\gamma\circ\sigma_\gamma]\bigl(u(t)\bigr)[\lambda_\gamma'\circ\sigma_\gamma]\bigl(u(t)\bigr)}{\sqrt{|u(t)|^2+\gamma^2}}\,u(t)\cdot\frac{\partial u}{\partial t}(t) \\
    &= \widetilde{F}_\gamma\bigl(u(t)\bigr)\cdot\frac{\partial u}{\partial t}(t)
  \end{align*}
  in $\overline{Q}_{\varepsilon,T}$, where
  \begin{align*}
    \widetilde{F}_\gamma(a) = \frac{[\lambda_\gamma\circ\sigma_\gamma](a)[\lambda_\gamma'\circ\sigma_\gamma](a)}{\sqrt{|a|^2+\gamma^2}}\,a, \quad a\in\mathbb{R}^N.
  \end{align*}
  We multiply the above equality by $\theta\in C_c^\infty(0,T)$ and integrate both sides over $Q_{\varepsilon,T}$.
  Then, after integration by parts with respect to time, we have
  \begin{align} \label{Pf_TT:Gam}
    -\frac{1}{2}\int_0^T\theta'(t)\bigl\|[\lambda_\gamma\circ\sigma_\gamma]\bigl(u(t)\bigr)\bigr\|_{L^2(\Omega_\varepsilon)}^2\,dt = \int_0^T\theta(t)\Bigl(\partial_tu(t),\widetilde{F}_\gamma\bigl(u(t)\bigr)\Bigr)_{L^2(\Omega_\varepsilon)}\,dt.
  \end{align}
  Let $\gamma\to0$ in this equality.
  Then, by \eqref{E:Lambda}, \eqref{E:La_De}, and \eqref{E:Sig_Fg}, we have
  \begin{align*}
    \lim_{\gamma\to0}[\lambda_\gamma\circ\sigma_\gamma](u) = (|u|-C_0)_+, \quad \lim_{\gamma\to0}\widetilde{F}_\gamma(u) &= F(u),
  \end{align*}
  and $|[\lambda_\gamma\circ\sigma_\gamma](u)|\leq|u|$ and $|\widetilde{F}_\gamma(u)|\leq|u|$ in $Q_{\varepsilon,T}$, where we also used
  \begin{align*}
    \chi_{(C_0,\infty)}(z)(z-C_0)_+ = (z-C_0)_+, \quad 0 \leq \chi_{(C_0,\infty)}(z) \leq 1, \quad z\in\mathbb{R}.
  \end{align*}
  Hence, noting that $u\in C^\infty(\overline{Q}_{\varepsilon,T})^N$ and $\theta\in C_c^\infty(0,T)$, we can apply the dominated convergence theorem to \eqref{Pf_TT:Gam} to find that \eqref{Pf_TT:Smth} holds.

  \textbf{Step 2:} for $u\in C^\infty([0,T];(H^1\cap L^4)(\Omega_\varepsilon)^N)$, we prove \eqref{Pf_TT:Goal}.
  Since $u$ is of class $H^1$ on the space-time domain $Q_{\varepsilon,T}$ with Lipschitz boundary, there exist functions
  \begin{align} \label{Pf_TT:H1Co}
    u_k \in C^\infty(\overline{Q}_{\varepsilon,T})^N \quad\text{such that}\quad \lim_{k\to\infty}\|u-u_k\|_{H^1(Q_{\varepsilon,T})} = 0.
  \end{align}
  Then, since $u_k\to u$ strongly in $L^2(Q_{\varepsilon,T})^N$, we can take a subsequence of $\{u_k\}$, which is again denoted by $\{u_k\}$, and a function $f\in L^2(Q_{\varepsilon,T})$ such that
  \begin{align*}
    |u_k| \leq f, \quad \lim_{k\to\infty}u_k = u \quad\text{a.e. in}\quad Q_{\varepsilon,T},
  \end{align*}
  as in the proof of the completeness of $L^2$.
  Hence,
  \begin{align*}
    \lim_{k\to\infty}(|u_k|-C_0)_+ = (|u|-C_0)_+, \quad \lim_{k\to\infty}F(u_k) = F(u) \quad\text{a.e. in}\quad Q_{\varepsilon,T}.
  \end{align*}
  Moreover, since $0\leq(z-C_0)_+\leq|z|$ for $z\in\mathbb{R}$, we have
  \begin{align*}
    &|(|u_k|-C_0)_+-(|u|-C_0)_+| \leq (|u_k|-C_0)_++(|u|-C_0)_+ \leq |u_k|+|u| \leq f+|u|, \\
    &|F(u_k)-F(u)| \leq |F(u_k)|+|F(u)| \leq |u_k|+|u| \leq f+|u| \quad\text{a.e. in}\quad Q_{\varepsilon,T},
  \end{align*}
  where $f,|u|\in L^2(Q_{\varepsilon,T})$.
  Hence, by the dominated convergence theorem,
  \begin{align} \label{Pf_TT:L2St}
    \begin{alignedat}{2}
      \lim_{k\to\infty}(|u_k|-C_0)_+ &= (|u|-C_0)_+ &\quad &\text{strongly in} \quad L^2(Q_{\varepsilon,T}), \\
      \lim_{k\to\infty}F(u_k) &= F(u) &\quad &\text{strongly in} \quad L^2(Q_{\varepsilon,T})^N.
    \end{alignedat}
  \end{align}
  Also, each $u_k$ satisfies \eqref{Pf_TT:Smth} by Step 1.
  Hence, we send $k\to\infty$ in that equality and apply \eqref{Pf_TT:H1Co} and \eqref{Pf_TT:L2St} to find that \eqref{Pf_TT:Smth} also holds for $u$ (note that \eqref{Pf_TT:H1Co} includes the strong convergence of $\partial_tu_k$).
  Moreover, since
  \begin{align*}
    \Bigl(\partial_tu(t),F\bigl(u(t)\bigr)\Bigr)_{L^2(\Omega_\varepsilon)} = \bigl\langle\partial_tu(t), F\bigl(u(t)\bigr)\bigr\rangle_{(H^1\cap L^4)(\Omega_\varepsilon)}, \quad t\in[0,T]
  \end{align*}
  by $u\in C^\infty([0,T];(H^1\cap L^4)(\Omega_\varepsilon)^N)$, it follows that \eqref{Pf_TT:Goal} holds for this $u$.

  \textbf{Step 3:} let $u\in E_T(\Omega_\varepsilon)$.
  By Lemma \ref{L:ET_Den}, (ii), there exist functions
  \begin{align} \label{Pf_TT:ET_Ap}
    w_k\in C^\infty([0,T];(H^1\cap L^4)(\Omega_\varepsilon)^N) \quad\text{such that}\quad \lim_{k\to\infty}\|u-w_k\|_{E_T(\Omega_\varepsilon)} = 0.
  \end{align}
  Recall that $\|u\|_{E_T(\Omega_\varepsilon)}=\|u\|_{Z_T(\Omega_\varepsilon)}+\|\partial_tu\|_{[Z_T(\Omega_\varepsilon)]'}$ and
  \begin{align*}
    Z_T(\Omega_\varepsilon) = L^2(0,T;H^1(\Omega_\varepsilon)^N)\cap L^4(0,T;L^4(\Omega_\varepsilon)^N).
  \end{align*}
  Hence, by a diagonal argument and discussions as in the proof of the completeness of $L^p$, we can take a subsequence of $\{w_k\}$, which is again denoted by $\{w_k\}$, and functions
  \begin{align*}
    f_1\in L^4(Q_{\varepsilon,T}) \subset L^2(Q_{\varepsilon,T}), \quad f_2 \in L^2(Q_{\varepsilon,T})
  \end{align*}
  (note that $Q_{\varepsilon,T}=\Omega_\varepsilon\times(0,T)$ is bounded) such that
  \begin{align} \label{Pf_TT:wkae}
    |w_k| \leq f_1, \quad |\nabla w_k| \leq f_2, \quad \lim_{k\to\infty}w_k = u, \quad \lim_{k\to\infty}\frac{\partial w_k}{\partial x_i} = \frac{\partial u}{\partial x_i} \quad\text{a.e. in}\quad Q_{\varepsilon,T}
  \end{align}
  for $i=1,\dots,n$.
  Using these relations, we can show that
  \begin{align} \label{Pf_TT:wkL4}
    \begin{alignedat}{2}
      \lim_{k\to\infty}(|w_k|-C_0)_+ &= (|u|-C_0)_+ &\quad &\text{strongly in} \quad L^2(Q_{\varepsilon,T}), \\
      \lim_{k\to\infty}F(w_k) &= F(u) &\quad &\text{strongly in} \quad L^4(Q_{\varepsilon,T})^N
    \end{alignedat}
  \end{align}
  as in Step 2.
  Let $i=1,\dots,n$ and $j=1,\dots,N$.
  We would like to show
  \begin{align*}
    \lim_{k\to\infty}\frac{\partial}{\partial x_i}\Bigl(F_j(w_k)\Bigr) = \frac{\partial}{\partial x_i}\Bigl(F_j(u)\Bigr) \quad\text{strongly in}\quad L^2(Q_{\varepsilon,T}),
  \end{align*}
  but this is not possible since the convergence a.e. in $Q_{\varepsilon,T}$ does not necessarily hold because of the discontinuity of the characteristic function $\chi_{(C_0,\infty)}$ appearing in \eqref{E:Trn_Dxi}.
  Instead, we can get the weak convergence.
  Indeed, by the second convergence of \eqref{Pf_TT:wkL4} and integration by parts, we see that
  \begin{align} \label{Pf_TT:WeCi}
    \lim_{k\to\infty}\left(\frac{\partial}{\partial x_i}\Bigl(F_j(w_k)\Bigr),\varphi\right)_{L^2(Q_{\varepsilon,T})} = \left(\frac{\partial}{\partial x_i}\Bigl(F_j(u)\Bigr),\varphi\right)_{L^2(Q_{\varepsilon,T})}
  \end{align}
  for all $\varphi\in C_c^\infty(Q_{\varepsilon,T})$.
  Note that $Q_{\varepsilon,T}=\Omega_\varepsilon\times(0,T)$ and that \eqref{E:Trn_Dxi} is not used here.
  Since $C_c^\infty(Q_{\varepsilon,T})$ is dense in $L^2(Q_{\varepsilon,T})$, and since
  \begin{align*}
    \left\|\frac{\partial}{\partial x_i}\Bigl(F_j(w_k)\Bigr)\right\|_{L^2(Q_{\varepsilon,T})} \leq c\|\nabla w_k\|_{L^2(Q_{\varepsilon,T})} \leq c\|f_2\|_{L^2(Q_{\varepsilon,T})}
  \end{align*}
  by \eqref{E:Trn_Dxi} and \eqref{Pf_TT:wkae}, where the last term is independent of $k$, we can get
  \begin{align*}
    \lim_{k\to\infty}\frac{\partial}{\partial x_i}\Bigl(F_j(w_k)\Bigr) = \frac{\partial}{\partial x_i}\Bigl(F_j(u)\Bigr) \quad\text{weakly in}\quad L^2(Q_{\varepsilon,T})
  \end{align*}
  by \eqref{Pf_TT:WeCi} for $\varphi\in C_c^\infty(Q_{\varepsilon,T})$ and a density argument.
  By this result, the second convergence of \eqref{Pf_TT:wkL4}, and $L^4(Q_{\varepsilon,T})\hookrightarrow L^2(Q_{\varepsilon,T})$, we find that
  \begin{align} \label{Pf_TT:FwkWe}
    \lim_{k\to\infty}F(w_k) = F(u) \quad\text{weakly in}\quad Z_T(\Omega_\varepsilon).
  \end{align}
  Now, since each $w_k$ satisfies \eqref{Pf_TT:Goal} by Step 2, we send $k\to\infty$ in that equality and apply \eqref{Pf_TT:ET_Ap}, \eqref{Pf_TT:wkL4}, and \eqref{Pf_TT:FwkWe}.
  Then, we find that \eqref{Pf_TT:Goal} holds for $u\in E_T(\Omega_\varepsilon)$ (note that \eqref{Pf_TT:ET_Ap} includes the strong convergence of $\partial_tw_k$).
\end{proof}

\section*{Acknowledgments}
The work of the author was supported by JSPS KAKENHI Grant Number 23K12993.

\bibliographystyle{abbrv}
\bibliography{GLCTD_Neu_Ref}

\begin{thebibliography}{10}

\bibitem{AlBrGa10}
S.~Alama, L.~Bronsard, and B.~Galv\~{a}o Sousa.
\newblock Thin film limits for {G}inzburg-{L}andau with strong applied magnetic
  fields.
\newblock {\em SIAM J. Math. Anal.}, 42(1):97--124, 2010.

\bibitem{BenSha88}
C.~Bennett and R.~Sharpley.
\newblock {\em Interpolation of operators}, volume 129 of {\em Pure and Applied
  Mathematics}.
\newblock Academic Press, Inc., Boston, MA, 1988.

\bibitem{BerLof76}
J.~Bergh and J.~L\"{o}fstr\"{o}m.
\newblock {\em Interpolation spaces. {A}n introduction}, volume No. 223 of {\em
  Grundlehren der Mathematischen Wissenschaften}.
\newblock Springer-Verlag, Berlin-New York, 1976.

\bibitem{BoyFab13}
F.~Boyer and P.~Fabrie.
\newblock {\em Mathematical tools for the study of the incompressible
  {N}avier-{S}tokes equations and related models}, volume 183 of {\em Applied
  Mathematical Sciences}.
\newblock Springer, New York, 2013.

\bibitem{ChDuGu96}
S.~J. Chapman, Q.~Du, and M.~D. Gunzburger.
\newblock A model for variable thickness superconducting thin films.
\newblock {\em Z. Angew. Math. Phys.}, 47(3):410--431, 1996.

\bibitem{ChElQi98}
Z.~Chen, C.~M. Elliott, and T.~Qi.
\newblock Justification of a two-dimensional evolutionary {G}inzburg-{L}andau
  superconductivity model.
\newblock {\em RAIRO Mod\'{e}l. Math. Anal. Num\'{e}r.}, 32(1):25--50, 1998.

\bibitem{Cia97}
P.~G. Ciarlet.
\newblock {\em Mathematical elasticity. {V}ol. {II}}, volume~27 of {\em Studies
  in Mathematics and its Applications}.
\newblock North-Holland Publishing Co., Amsterdam, 1997.
\newblock Theory of plates.

\bibitem{Cia00}
P.~G. Ciarlet.
\newblock {\em Mathematical elasticity. {V}ol. {III}}, volume~29 of {\em
  Studies in Mathematics and its Applications}.
\newblock North-Holland Publishing Co., Amsterdam, 2000.
\newblock Theory of shells.

\bibitem{CiDaGr18}
D.~Cioranescu, A.~Damlamian, and G.~Griso.
\newblock {\em The periodic unfolding method}, volume~3 of {\em Series in
  Contemporary Mathematics}.
\newblock Springer, Singapore, 2018.
\newblock Theory and applications to partial differential problems.

\bibitem{Con11}
A.~Contreras.
\newblock On the first critical field in {G}inzburg-{L}andau theory for thin
  shells and manifolds.
\newblock {\em Arch. Ration. Mech. Anal.}, 200(2):563--611, 2011.

\bibitem{ConSte10}
A.~Contreras and P.~Sternberg.
\newblock Gamma-convergence and the emergence of vortices for
  {G}inzburg-{L}andau on thin shells and manifolds.
\newblock {\em Calc. Var. Partial Differential Equations}, 38(1-2):243--274,
  2010.

\bibitem{Du94}
Q.~Du.
\newblock Global existence and uniqueness of solutions of the time-dependent
  {G}inzburg-{L}andau model for superconductivity.
\newblock {\em Appl. Anal.}, 53(1-2):1--17, 1994.

\bibitem{DuGun93}
Q.~Du and M.~D. Gunzburger.
\newblock A model for superconducting thin films having variable thickness.
\newblock {\em Phys. D}, 69(3-4):215--231, 1993.

\bibitem{DziEll13}
G.~Dziuk and C.~M. Elliott.
\newblock Finite element methods for surface {PDE}s.
\newblock {\em Acta Numer.}, 22:289--396, 2013.

\bibitem{Eri62}
J.~L. Ericksen.
\newblock Hydrostatic theory of liquid crystals.
\newblock {\em Arch. Rational Mech. Anal.}, 9:371--378, 1962.

\bibitem{GilTru01}
D.~Gilbarg and N.~S. Trudinger.
\newblock {\em Elliptic partial differential equations of second order}.
\newblock Classics in Mathematics. Springer-Verlag, Berlin, 2001.
\newblock Reprint of the 1998 edition.

\bibitem{Glo11}
D.~Glotov.
\newblock Vortices in the three-dimensional thin-film {G}inzburg-{L}andau model
  of superconductivity.
\newblock {\em Z. Angew. Math. Phys.}, 62(5):891--907, 2011.

\bibitem{GoMoSt17}
D.~Golovaty, J.~A. Montero, and P.~Sternberg.
\newblock Dimension reduction for the {L}andau--de {G}ennes model on curved
  nematic thin films.
\newblock {\em J. Nonlinear Sci.}, 27(6):1905--1932, 2017.

\bibitem{JimMor02}
S.~Jimbo and Y.~Morita.
\newblock Ginzburg-{L}andau equation with magnetic effect in a thin domain.
\newblock {\em Calc. Var. Partial Differential Equations}, 15(3):325--352,
  2002.

\bibitem{LaSoUr68}
O.~A. Lady\v{z}enskaja, V.~A. Solonnikov, and N.~N. Ural$'$ceva.
\newblock {\em Linear and quasilinear equations of parabolic type}.
\newblock Translations of Mathematical Monographs, Vol. 23. American
  Mathematical Society, Providence, R.I., 1968.
\newblock Translated from the Russian by S. Smith.

\bibitem{Lee18}
J.~M. Lee.
\newblock {\em Introduction to {R}iemannian manifolds}, volume 176 of {\em
  Graduate Texts in Mathematics}.
\newblock Springer, Cham, 2018.
\newblock Second edition of [ MR1468735].

\bibitem{Les68}
F.~M. Leslie.
\newblock Some constitutive equations for liquid crystals.
\newblock {\em Arch. Rational Mech. Anal.}, 28(4):265--283, 1968.

\bibitem{Lin89}
F.-H. Lin.
\newblock Nonlinear theory of defects in nematic liquid crystals; phase
  transition and flow phenomena.
\newblock {\em Comm. Pure Appl. Math.}, 42(6):789--814, 1989.

\bibitem{LinLiu95}
F.-H. Lin and C.~Liu.
\newblock Nonparabolic dissipative systems modeling the flow of liquid
  crystals.
\newblock {\em Comm. Pure Appl. Math.}, 48(5):501--537, 1995.

\bibitem{Mas84}
K.~Masuda.
\newblock Weak solutions of {N}avier-{S}tokes equations.
\newblock {\em Tohoku Math. J. (2)}, 36(4):623--646, 1984.

\bibitem{Miu17}
T.-H. Miura.
\newblock Zero width limit of the heat equation on moving thin domains.
\newblock {\em Interfaces Free Bound.}, 19(1):31--77, 2017.

\bibitem{Miu20_03}
T.-H. Miura.
\newblock Navier-{S}tokes equations in a curved thin domain, {P}art {III}:
  thin-film limit.
\newblock {\em Adv. Differential Equations}, 25(9-10):457--626, 2020.

\bibitem{Miu23}
T.-H. Miura.
\newblock Error estimate for classical solutions to the heat equation in a
  moving thin domain and its limit equation.
\newblock {\em Interfaces Free Bound.}, 25(4):633--670, 2023.

\bibitem{MoShWa23}
S.~Moll, K.~Shirakawa, and H.~Watanabe.
\newblock Existence of solutions to a phase-field model of 3{D} grain boundary
  motion governed by a regularized 1-harmonic type flow.
\newblock {\em J. Nonlinear Sci.}, 33(5):Paper No. 68, 43, 2023.

\bibitem{Mor04}
Y.~Morita.
\newblock Stable solutions to the {G}inzburg-{L}andau equation with magnetic
  effect in a thin domain.
\newblock {\em Japan J. Indust. Appl. Math.}, 21(2):129--147, 2004.

\bibitem{NiNePrLoVo18}
I.~Nitschke, M.~Nestler, S.~Praetorius, H.~L\"{o}wen, and A.~Voigt.
\newblock Nematic liquid crystals on curved surfaces: a thin film limit.
\newblock {\em Proc. A.}, 474(2214):20170686, 20, 2018.

\bibitem{NiReVo19}
I.~Nitschke, S.~Reuther, and A.~Voigt.
\newblock Hydrodynamic interactions in polar liquid crystals on evolving
  surfaces.
\newblock {\em Phys. Rev. Fluids}, 4:044002, Apr 2019.

\bibitem{OckOck95}
H.~Ockendon and J.~R. Ockendon.
\newblock {\em Viscous flow}.
\newblock Cambridge Texts in Applied Mathematics. Cambridge University Press,
  Cambridge, 1995.

\bibitem{Ped87}
J.~Pedlosky.
\newblock {\em Geophysical Fluid Dynamics}.
\newblock Springer New York, NY, second edition, 1987.

\bibitem{RicRub00}
G.~Richardson and J.~Rubinstein.
\newblock The mixed boundary condition for the {G}inzburg {L}andau model in
  thin films.
\newblock {\em Appl. Math. Lett.}, 13(3):97--99, 2000.

\bibitem{Rob01}
J.~C. Robinson.
\newblock {\em Infinite-dimensional dynamical systems}.
\newblock Cambridge Texts in Applied Mathematics. Cambridge University Press,
  Cambridge, 2001.
\newblock An introduction to dissipative parabolic PDEs and the theory of
  global attractors.

\bibitem{RubSch98}
J.~Rubinstein and M.~Schatzman.
\newblock Asymptotics for thin superconducting rings.
\newblock {\em J. Math. Pures Appl. (9)}, 77(8):801--820, 1998.

\end{thebibliography}

\end{document}